\documentclass[11pt,reqno]{amsart}
\usepackage[margin=1.25in]{geometry}
\usepackage{amsaddr}
\usepackage[utf8]{inputenc}
\usepackage{color}
\usepackage{amsmath,amssymb,amsthm}
\usepackage{graphicx}
\usepackage{dsfont}
\usepackage{relsize}
\usepackage{enumitem}

\usepackage{tikz}
\usetikzlibrary{math,calc}

\usepackage[colorlinks=true,
linkcolor=webgreen,
filecolor=webbrown,
citecolor=webgreen,
urlcolor=webgreen]{hyperref}
\definecolor{webgreen}{rgb}{0,.5,0}
\definecolor{webbrown}{rgb}{.6,0,0}

\usepackage{inconsolata}
\usepackage{bbm}
\usepackage{mathtools}

\usepackage{shuffle}
\usepackage[OT2,T1]{fontenc}
\DeclareSymbolFont{cyrletters}{OT2}{wncyss}{m}{n}
\DeclareMathSymbol{\sha}{\mathalpha}{cyrletters}{"58}

\usepackage{etoolbox}
\patchcmd{\subsection}{-.5em}{.5em}{}{}
\patchcmd{\subsection}{.5\linespacing}{1.25\linespacing}{}{}
\patchcmd{\section}{.7\linespacing}{1.5\linespacing}{}{}
\makeatletter
\patchcmd{\@startsection}{\@afterindenttrue}{\@afterindentfalse}{}{}
\makeatother

\newtheorem{thm}{Theorem} 
\newtheorem*{thm*}{Theorem}
\newtheorem{prop}{Proposition}[section]

\newtheorem{lem}[prop]{Lemma}

\newtheorem{lemma}[prop]{Lemma}
\newtheorem{cor}[prop]{Corollary}
\newtheorem{corollary}[prop]{Corollary}

\newtheorem{alsotheorem}[prop]{Theorem}

\newtheorem{obs}[prop]{Observation}
\newtheorem*{obs*}{Observation}

\newtheorem{ex}{Example}
\theoremstyle{definition}
\newtheorem{definition}[prop]{Definition}

\newtheorem{rem}[prop]{Remark}
\theoremstyle{remark}

\newtheorem*{rmk*}{Remark}
\newtheorem*{ex*}{Example}

\newcommand{\R}{\mathbb{R}}

\newcommand{\A}{M}

\newcommand{\del}{\partial}
\newcommand{\sh}{\sha}

\newcommand{\grad}{\nabla}

\newcommand{\mergeset}{\mathcal{M}}
\newcommand{\mergecnt}{\mu}

\newcommand{\TA}{{\text{\bfseries\ttfamily a}}}
\newcommand{\TB}{{\text{\bfseries\ttfamily b}}}
\newcommand{\TC}{{\text{\bfseries\ttfamily c}}}
\newcommand{\TD}{{\text{\bfseries\ttfamily d}}}

\newcommand{\TI}{{\text{\bfseries\ttfamily i}}}
\newcommand{\TJ}{{\text{\bfseries\ttfamily j}}}
\newcommand{\TS}{{\text{\bfseries\ttfamily s}}}
\newcommand{\TT}{{\text{\bfseries\ttfamily t}}}
\newcommand{\TX}{{\text{\bfseries\ttfamily x}}}
\newcommand{\TY}{{\text{\bfseries\ttfamily y}}}
\newcommand{\TZ}{{\text{\bfseries\ttfamily z}}}
\newcommand{\To}{{\text{\bfseries\ttfamily 0}}}
\newcommand{\Ti}{{\text{\bfseries\ttfamily 1}}}
\newcommand{\Tii}{{\text{\bfseries\ttfamily 2}}}
\newcommand{\Tiii}{{\text{\bfseries\ttfamily 3}}}

\newcommand{\THH}{{\text{\bfseries\ttfamily H}}}
\newcommand{\TTT}{{\text{\bfseries\ttfamily T}}}
\newcommand{\TEE}{{\text{\bfseries\ttfamily E}}}
\newcommand{\TNN}{{\text{\bfseries\ttfamily N}}}
\newcommand{\TWW}{{\text{\bfseries\ttfamily W}}}
\newcommand{\TSS}{{\text{\bfseries\ttfamily S}}}

\newcommand{\Mp}{\mathcal{M}}
\newcommand{\Np}{\mathcal{N}}
\newcommand{\shp}{\mathcal{L}}
\newcommand{\RTR}{\mathcal{R}}
\newcommand{\proj}{\mathcal{P}}

\newcommand{\npt}{\texorpdfstring{\spacefactor 1007}{}}
\newcommand{\Exp}{\operatorname{\mathbb{E}}}
\newcommand{\PR}{\operatorname{\mathbb{P}}}
\newcommand{\Var}{\operatorname{V}}
\newcommand{\Cov}{\operatorname{Cov}} 
\newcommand{\word}{\operatorname{word}} 
\newcommand{\kpa}{\boldsymbol{\kappa}}
\newcommand{\lmda}{\boldsymbol{\lambda}}
\DeclareMathOperator{\rank}{rank}
\DeclareMathOperator{\stab}{stab}
\DeclareMathOperator{\SPAN}{span}


\title{Spectral Analysis of Word Statistics}

\author{Chaim Even-Zohar}
\address{
Department of Mathematics, Technion \\[0.5em]
Haifa 3200003, Israel \\[0.5em]
\texttt{chaime@technion.ac.il}}

\author{Tsviqa Lakrec}
\address{     
Section de math\'ematiques, Universit\'e de Gen\`eve
\\[0.5em]
UNI DUFOUR
24, rue du G\'en\'eral Dufour
Case postale 64
1211 Gen\`eve 4, Suisse
\\[0.5em]
\texttt{tsviqa@gmail.com}}

\author{Ran J. Tessler}
\address{
Department of Mathematics, Weizmann Institute of Science\\[0.5em]
POB 26, Rehovot 7610001, Israel \\[0.5em]
\texttt{ran.tessler@weizmann.ac.il }}
\date{\today}

\begin{document}

\begin{abstract}
\vspace{2em}

Given a random text over a finite alphabet, we study the frequencies at which fixed-length words occur as subsequences. As the data size grows, the joint distribution of word counts exhibits a rich asymptotic structure. We investigate all linear combinations of subword statistics, and fully characterize their different orders of magnitude using diverse algebraic tools.

Moreover, we establish the spectral decomposition of the space of word statistics of each order. We provide explicit formulas for the eigenvectors and eigenvalues of the covariance matrix of the multivariate distribution of these statistics. Our techniques include and elaborate on a set of algebraic word operators, previously studied and employed by Dieker and Saliola (Adv~Math, 2018).

Subword counts find applications in Combinatorics, Statistics, and Computer Science. We revisit special cases from the combinatorial literature, such as intransitive dice, random core partitions, and questions on random walk. Our structural approach describes in a unified framework several classical statistical tests. We propose further potential applications to data analysis and machine learning.
\vspace{1em}
\end{abstract}

\maketitle

\section{Introduction}

\subsection{Word Statistics}

Sequences over a finite alphabet are ubiquitous in pure and applied mathematics and lie at the core of many probabilistic models. They may represent steps of a random walk, words of group generators, discrete-valued time series, DNA segments, or output of pseudorandom generators, to mention a few examples. In the analysis of such sequences, one often considers various numerical \emph{statistics}, in order to capture the main features of the data, extract meaningful information, apply further processing, or make informed decisions. It is therefore important to examine general families of such statistics and thoroughly understand their expected behavior.

\smallskip
\emph{Subword counts} give rise to a broad family of word statistics, which this work investigates. Given a finite alphabet $\Sigma = \{\TA,\TB,\TC,\dots\}$, a pattern $u \in \Sigma^k$, and a longer text $w \in \Sigma^n$, we consider $\#u(w)$, the number of occurrences of~$u$ as a subsequence of~$w$. The copies of~$u$ that we count do not have to appear consecutively in the text, nor to be disjoint. For example \mbox{$\#\texttt{fee}(\texttt{referee}) = 3$}. Many well-studied word statistics are special cases of these counts, or finite linear combinations of them.

\smallskip
\emph{Randomized models} provide a natural setting for investigating words and their statistics. They help us analyze these fundamental objects via typical instances, and guide us in developing relevant tools for applications. Here are two basic models for a random word~$w \in \Sigma^n$, that appear naturally in various contexts and applications. 
\begin{itemize}[leftmargin=15pt]
\setlength{\itemsep}{0.5em}
\item 
\textbf{One-Sample: $\mathcal{W}(n,\mathbf{p})$} where $\mathbf{p}=(p_\TA, p_\TB, \dots) \in (0,1)^{|\Sigma|}$ and $\sum_{\TX} p_\TX = 1$ \\[0.25em] 
The letters of~$w$ are independent, and every letter $w_i=\TX$ with probability~$p_\TX$. 
\item
\textbf{Multi-Sample: $\mathcal{W}'(\mathbf{n})$} where $\mathbf{n}=(n_\TA, n_\TB,\dots) \in \mathbb{N}^{|\Sigma|}$ and $\sum_\TX n_\TX = n$ \\[0.25em] 
Every word~$w$ with exactly $n_\TX=\#\TX(w)$ for every~$\TX$, is equally likely.
\end{itemize}
The word models $\mathcal{W}$ and $\mathcal{W}'$ parallel the two best-studied random graph models on $n$ labeled vertices. For graphs, $\mathcal{G}(n,p)$ selects every edge with probability~$p$ independently, and $\mathcal{G}'(n,m)$ selects exactly $m$ edges uniformly from all possible ways~\cite{janson2011random}. While the two kinds of models share many asymptotic properties, they differ in some important aspects, especially regarding subgraph counts, and in our models -- subword counts.

\subsection{Spaces of Subword Counts}

We start with a general presentation of our approach to subword statistics. Some new results and special cases will be mentioned, but full formal statements are deferred to the subsequent~\S\ref{results}. 

Let $k \in \mathbb{N}$, and consider the random variables $\#u$, for all $k$-letter words $u \in \Sigma^k$. For the sake of this general discussion, the distribution of the underlying $w \in \Sigma^n$ may be either $\mathcal{W}(n,\mathbf{p})$ or~$\mathcal{W}'(\mathbf{n})$. In the latter model, we let $n_\TX = p_\TX n$ and the same general statements apply up to minor changes. 

How is the subword count $\#u$ distributed as the text size $n$ grows? By summation over all $\tbinom{n}{k}$ potential occurrences, one can see that the expected value and variance are
$$ \Exp\left[\#u\right] \;=\; \frac{p_{u_1}\!\cdots p_{u_k}}{k!}\, n^k \pm O(n^{k-1})\,, \;\;\;\;\;\;\;\; \Var\left[\#u\right] \;=\; O\left(n^{2k-1}\right) $$

It follows that the vector of \emph{subword frequencies}, $\#u/\tbinom{n}{k}$ for $u \in \Sigma^k$, satisfies a law of large numbers: 
$$ \mathbf{X}_k \;:=\; \left\{\tfrac{\displaystyle \#u}{\tbinom{n}{k}}\right\}_{u \in \Sigma^k} \;\;\;\xrightarrow[\text{\;in probability\;}]{n \to \infty} \;\;\; \Exp\left[\mathbf{X}_k\right] \;=\; \mathbf{p}^{\otimes k} $$

It is then natural to study interactions between different subword counts. In general, there is a nonzero correlation between $\#u$ and $\#v$, even in the limit as $n \to \infty$. These correlations are encoded in the following $|\Sigma|^k$-dimensional central limit theorem, as we will see later on.
$$ \sqrt{n}\left(\mathbf{X}_k - \Exp\left[\mathbf{X}_k\right]\right) \;\;\;\xrightarrow[\text{\;in distribution\;}]{n \to \infty} \;\;\; \mathcal{N}\left(\mathbf{0},\, \lim_{n \to \infty} n  \Cov\left[\mathbf{X}_k\right] \right) $$

However, the multivariate Gaussian limit reveals only a small part of the asymptotic picture. It turns out that the rank of the limiting covariance matrix is much lower than~$|\Sigma|^k$, so that the limit law is supported on a low-dimensional subspace. In terms of linear combinations of the form $\sum_u f_u\, \#u$ with $f_u \in \mathbb{R}$, many of those are significantly more concentrated than their individual constituents, and should be scaled differently. 

Let $\mathbb{R}\Sigma^k$ denote the space of formal linear combinations of $k$-letter words over~$\Sigma$. Every $f = \sum_u f_u u \in \mathbb{R}\Sigma^k$ defines a scalar random variable $\#f$ by linearity. One desirable goal is to find the typical order of magnitude of all~$\#f$. The first step in our approach is \emph{grading} the space of all subword combinations. This grading provides an orthogonal decomposition $\mathbb{R}\Sigma^k = \bigoplus_{r} V_r$ such that $\Exp[(\#f/n^k)^2] = \Theta(1/n^r)$ for every nonzero $f \in V_r$.

The next goal is to analyze the random variables within each component, that is, $n^{r/2} \mathbf{X}_k$ projected onto~$V_r$. The spaces of word statistics in our models come with natural inner product structures. The most fundamental and most practical objective is a basis of statistics that diagonalizes the covariance matrix of this multivariate distribution, in the spirit of principal component analysis, PCA. Thus, the second step is a \emph{spectral} decomposition of each component~$V_r$. 

Having a full explicit decomposition of this form, the precise leading term of the variance $\Var[\#f]$ can be readily obtained for any feature $f$ which is a scalar projection of~$\mathbf{X}_k$. It lets one identify and compare the various ``modes'' of the joint distribution, which reveals much of its structure. 

Our main contribution is the implementation of this plan. We provide gradings of word statistics by scale, and diagonalizations by second moments, as stated below in~\S\ref{results}. These are demonstrated on diverse examples in~\S\ref{examples}. Several previously-studied word statistics naturally arise as special cases, including some of order smaller than $1/\sqrt{n}$. Also, new families of word statistics constructed this way seem to be meaningful and useful. 

The analysis of multivariate statistical features of ordered or sequential data is a direct practical application of our work. Linear decompositions of data on combinatorial structures have been studied since the seminal monograph by Diaconis \cite[\S8]{diaconis1988group}, which introduced the use of algebraic tools such as representations of the symmetric group. However, the crucial issue of choosing bases for components has mostly been left arbitrary, depending on matters of convenience, or ad~hoc interpretations. Our proposed approach, which turns to the second moment structure of typical data distributions, aims to provide a systematic treatment that seems very natural from a practical perspective. In fact, the random word models we use make it particularly well-suited for extracting features in the high-noise regime.

\subsection{Main Results}
\label{results}

We now present the scaling decompositions and the spectral decompositions of the subword statistics of random words. The one-sample model $\mathcal{W}(n,\mathbf{p})$, where the letters are independent, is treated in Theorems \ref{main1}~and~\ref{main2}. Theorems \ref{main3}~and~\ref{main4} concern the more involved setting of the multisample model $\mathcal{W}'(\mathbf{n})$, with randomly ordered letters. All the components can be obtained by straightforward elementary computations, using Gaussian elimination and combinatorial manipulations on words. The details of the constructions are deferred to~\S\ref{defs}. 

\medskip

Let $w$ be a random word in the model $\mathcal{W}(n, \mathbf{p})$. Recall that $\Sigma = \{\TA,\TB,\cdots\}$ is a finite alphabet, so that $d := |\Sigma| \geq 2$, and the characters of $w$ are independent and distributed with $\mathbf{p} = (p_\TA, p_\TB, \cdots) \in \mathbb{R}^d$. We count the subwords $u \in \Sigma^k$ occurring in $w \in \Sigma^n$, and study the normalized statistic,
$$ \bar\#u(w) \;:=\; \frac{\#u(w)}{\tbinom{n}{k}} \;\in\; [0,1]$$
Moreover, we study all linear combinations of the random variables $\bar\#u$ for $u \in \Sigma^k$. Every formal sum $f = \sum_u f_u u$ in the $d^k$-dimensional space
$$ W_k \;:=\; \mathbb{R}\Sigma^k $$
defines such statistics $\#f$ and $\bar\#f$ by linearity. 

Working with a single length $k \in \mathbb{N}$ is not a real restriction. Indeed, Proposition~\ref{gradingw} gives compatible linear embeddings $W_k \hookrightarrow W_{k+1}$. Therefore, every $W_k$ contains all $W_{j}$ for $j < k$. The space of all subword statistics is thus denoted
$$ W \;:=\; \bigcup_{k\in\mathbb{N}}\, W_k $$ 

In order to establish the scaling of $\#f$ for every $f \in W$, we study the structure of every~$W_k$. Definition~\ref{wkr} introduces a grading on the spaces $W_k$, which will yield a well-defined grading on $W$. Every space $W_k$ decomposes into $k+1$ subspaces, denoted as follows.
\begin{align*}
& W_k \;=\; W_{k0} \,\oplus\, W_{k1} \,\oplus\, \dots \,\oplus\, W_{kk} \\ 
& \dim W_{kr} \;=\; \tbinom{k}{r}(d-1)^r 
\end{align*}
This primary decomposition depends on the probability vector $\mathbf{p}$. The following theorem asserts that it determines the order of magnitude in $n$ of any statistic in $W_k$, and the different components are uncorrelated. 

\begin{thm}
[Grading under $\mathcal{W}(n,\mathbf{p})$]
\label{thm:main_grading_variance}
\label{main1}
$ $ \\[0.25em]
Let $k \in \mathbb{N}$ and $r \in \{0,1,\dots,k\}$. For every nonzero statistic $f \in W_{kr}$ there exists $C_{f,\mathbf{p}} > 0$ such that
$$ n^r \, \Exp_{\,\mathcal{W} (n,\mathbf{p})}\left[\left(\bar\#f\right)^2\right] \; \xrightarrow[\;n \to \infty\;]{}\; C_{f,\mathbf{p}}. $$
Moreover, for every $r' \neq r$ and $f' \in W_{kr'}$, \;$\Exp_{\,\mathcal{W}(n,\mathbf{p})}\left[\,\bar\#f\;\bar\#f'\,\right] \, = \, 0$.
\end{thm}

\begin{rmk*}
This decomposition also has the property that $W_{k0}, \dots, W_{kk}$ are pairwise orthogonal. Here we work with an inner product on $W_k$, naturally induced from the measure $\mathcal{W}(k,\mathbf{p})$, and denoted~$\langle f,f'\rangle_{\mathbf{p}}$, see Definition~\ref{inner}.
\end{rmk*}

We further refine each component $W_{kr}$ into $k-r+1$ orthogonal subspaces. For every $k \geq r \geq 1$, the following decomposition is given in Definition~\ref{wkrm}: 
\begin{align*}
& W_{kr} \;=\; W_{kr0} \,\oplus\, W_{kr1} \,\oplus\, \dots \,\oplus\, W_{kr(k-r)} \\ 
& \dim W_{krm} \;=\; \tbinom{r+m-1}{m}(d-1)^r
\end{align*}
This secondary decomposition yields a full asymptotic diagonalization of the covariance of~$W_k$, as follows.

\begin{thm}[Spectrum under $\mathcal{W}(n,\mathbf{p})$] $ $ \\[0.25em]
\label{main2}
Let $k \in \mathbb{N}$, $r \in \{1,\dots,k\}$, and $m,m' \in \{0,\dots,k-r\}$. For every $f \in W_{krm}$ and $f' \in W_{krm'}$ 
$$ \Exp_{\,\mathcal{W}(n,\mathbf{p})}\left[\left(n^{r/2}\,\bar\#f\right)\left(n^{r/2}\,\bar\#f'\right)\right] \;\; \xrightarrow[\;n \to \infty\;]{}\;\; \frac{(k!)^2 \, \left\langle f,f' \right\rangle_{\mathbf{p}}}{(k+m)!(k-r-m)!} $$
In particular, if $m'\neq m$ then this limit is 
$\left\langle f,f' \right\rangle_{\mathbf{p}} = 0$.
\end{thm}

In Theorem~\ref{structure} we present a concise and practical description of the spaces~$W_{krm}$, which provides insight into their structure. We establish an explicit isomorphism between $W_{krm}$ and $U_{krm} \otimes (\mathbb{R}^{d-1})^{\otimes r}$, where $U_{krm}$ are spaces of \emph{multivariate orthogonal polynomials on the discrete simplex}, described in Definitions \mbox{\ref{phikr}-\ref{ukrm}}.

\begin{rmk*}
We will see that if $f \in W_{kr}$ then $\bar\#f$ is a so-called \emph{U-statistic of rank~$r$}. This fact provides some additional information on the distribution of these random variables. See~\S\ref{ustats}-\S\ref{gustats}.
\end{rmk*}

\medskip

We now turn to the other model $\mathcal{W}'(\mathbf{n})$ where the random word $w$ has a prescribed \emph{composition} $\mathbf{n} =  (n_\TA,n_\TB,n_\TC,\dots)$, meaning $\#\TX(w)=n_\TX$ for every letter $\TX \in \Sigma$. Denote the set of such words by $\tbinom{\mathsmaller{\Sigma}}{\mathbf{n}}$, and denote their length by $n = |\mathbf{n}| := {\sum}_\TX n_\TX$. The number of words in the set $\tbinom{\mathsmaller{\Sigma}}{\mathbf{n}}$ is the multinomial coefficient $\tbinom{n}{\mathbf{n}} = n!/(n_\TA!n_\TB!\cdots)$, and each such word is equally likely in~$\mathcal{W}'(\mathbf{n})$.

As before, we count the occurrences of subwords $u \in \Sigma^k$ and analyze the random variables $\#u$, or $\#f$ for linear combinations $f = \sum_u f_u u$. However, in this model, it is sufficient to consider words~$u \in \tbinom{\mathsmaller{\Sigma}}{\kpa}$, fixing the composition $\kpa = (k_\TA,k_\TB,\dots)$ of~$u$. Indeed, Proposition~\ref{embed} shows how subwords of different compositions reduce to this case. We therefore work in the linear space of formal sums of words of composition~$\kpa$, denoted
$$ W_{\kpa} \;=\; W_{(k_\TA ,k_\TB,\dots)} \;:=\; \mathbb{R}\tbinom{\Sigma}{\kpa} $$
Note that $\dim W_{\kpa} = \tbinom{k}{\kpa}$ where $k = |\kpa|$. For $u\in \tbinom{\mathsmaller{\Sigma}}{\kpa}$, a natural choice of normalization is 
$$ \tilde\# u \;:=\; \frac{\#u}{\prod_{\TX  \in \Sigma} \tbinom{n_\TX}{k_\TX}} \;\in\; [0,1]$$
extended to $\tilde\#f$ for linear combinations $f = \sum_uf_uu \in W_{\kpa}$. Without loss of generality, we assume $k_\TA \geq k_\TB \geq \dots > 0$ unless stated otherwise.

Our primary decomposition of $W_{\kpa}$ is based on representations of the symmetric group~$S_k$. The space $W_{\kpa}$ admits an action of $S_k$ by reordering all $k$-letter words in its basis. The implied decomposition of $W_{\kpa}$ as a direct sum of simple $S_k$ representations is well-studied and briefly reviewed in~\S\ref{repr}. Definition~\ref{Wkkr} uses it to describe the following $k-k_\TA+1$ components of word statistics.
$$ W_{\kpa} \;=\; W_{\kpa0} \oplus W_{\kpa1} \oplus \dots \oplus W_{\kpa(k-k_\TA)} $$
The next theorem asserts that the word statistics in~$W_{\kpa r}$ have order of magnitude~$n^{-r/2}$, and that different components $W_{\kpa r}$ and $W_{\kpa r'}$ are asymptotically uncorrelated. By $\mathbf{n}/n \to \mathbf{p}$ we denote the assumption that the parameters $\mathbf{n} = (n_\TA,n_\TB,\dots)$ grow such that $n_{\TX}/n \to p_{\TX}>0$ as $n = |\mathbf{n}| \to \infty$, for every~$\TX$.

\begin{thm}
[Grading under $\mathcal{W}'(n_\TA, n_\TB, n_\TC, \dots)$]
$ $ 
\label{main3}\label{thm:main3}
\\[0.25em]
Let $f \in W_{\kpa r}$ be a nonzero statistic of composition $\kpa = (k_\TA,k_\TB,\dots)$ where $r \in \left\{0,\dots,|\kpa|{-}k_\TA\right\}$, and suppose that $\mathbf{n}/n \to \mathbf{p}$. Then, there exists $C'_{f,\mathbf{p}} > 0$ such that
$$ n^r \, \Exp_{\,\mathcal{W}' (\mathbf{n})}\left[\left(\tilde\#f\right)^2\right] \; \xrightarrow[\;n\to \infty\;]{}\; C'_{f,\mathbf{p}}\;. $$
Moreover, for every $r' \neq r$ and $f' \in W_{\kpa r'}$
$$ \Exp_{\,\mathcal{W}'(\mathbf{n})}\left[\left(n^{r/2}\,\tilde\#f\right)\left(n^{r'/2}\,\tilde\#f'\right)\right] \;\xrightarrow[\;n\to \infty\;]{}  \; 0 \;.$$
\end{thm}

\begin{rmk*}
The components $W_{\kpa 0},W_{\kpa 1},W_{\kpa 2},\dots$ are pairwise orthogonal with respect to the standard inner product of $W_{\kpa}$, denoted $\langle-,-\rangle$. See~\S\ref{repr}. 
\end{rmk*}

\begin{rmk*}
Similar to the first random model, in fact we will see that the random variables $\tilde \# f$ are \emph{generalized U-statistics of rank $r$}. See~\S\ref{gustats}.
\end{rmk*}

\medskip

The next result elaborates on the two-sample random model $\mathcal{W}'(n_\TA,n_\TB)$. Here $w$ is a uniformly random word of length $n = n_\TA+n_\TB$ with $\#\TA(w)=n_\TA$ and $\#\TB(w)=n_\TB$, and we count all subwords of composition $\kpa=(k_\TA,k_\TB)$ with $k_\TA \geq k_\TB \geq 1$, where $k = |\kpa| = k_\TA + k_\TB$. The primary decomposition of $W_{\kpa}$ already gives the components $W_{\kpa r}$ for $r \in \{0,\dots,k_\TB\}$.

The full decomposition of $W_{\kpa}$ will be given by Definition~\ref{wkrij}, that refines every $W_{\kpa r}$ into $(k-2r+1)r$ orthogonal subspaces as follows:
\begin{align*}
&W_{\kpa r} \;=\; \bigoplus_{i=0}^{k-2r} \; \bigoplus_{j=0}^{r-1} \; W_{\kpa rij} \;\;\;\;\;\;\;\;\;\; r \in \{1,\dots,k_\TB\} \\[0.5em]
&\dim W_{\kpa rij} \;=\; \frac{(k-2r-i+j+1)\,(k-i-j-2)!}{(k-i-r)!\,(r-j-1)!} \end{align*}
We do not consider the case $r=0$, because $W_{\kpa 0}$ is simply the 1-dimensional space of constant statistics.

This decomposition yields the following full asymptotic diagonalization of the covariance matrix. In writing $f \in W_{\kpa rij}$ it is implied that $r,i,j$ are any numbers in the applicable ranges $r \in \{1,\dots,k_\TB\}$, $i \in \{0,\dots,k-2r\}$, and $j\in\{0,\dots,r-1\}$, where as usual $k=k_\TA+k_\TB$ and $n=n_\TA+n_\TB$.

\begin{samepage}
\begin{thm}[Spectrum under $\mathcal{W}'(n_\TA,n_\TB)$] $ $ \label{main4 nisayon}
\label{main4}
\label{thm:main4}
\nopagebreak
\\[0.25em]
Let $\kpa=(k_\TA,k_\TB)$. For every two word statistics $f \in W_{\kpa rij}$ and $f' \in W_{\kpa r'i'j'}$
$$ \Exp_{\,w\in\mathcal{W}(n_\TA,n_\TB)}\left[ \left(\left(\tfrac{n_\TA n_\TB}{n}\right)^{r/2} \tilde\#f \right)\;\left( \left(\tfrac{n_\TA n_\TB}{n}\right)^{r'/2} \tilde\#f'\right) \right] \;\;\; \xrightarrow[\;\;n_\TA,n_\TB \;\to\; \infty\;\;]{} \;\;\; \Lambda_{\kpa rij} \left\langle f,f' \right\rangle $$
where 
$$ \Lambda_{\kpa rij} \;:=\; \frac{(k_\TA!)^2\,(k_\TB!)^2\,(k-2r)!\,(k-2r+1)!}{(k_\TA-r)!\,(k_\TB-r)!\,i!\,(2k-r-i-j)!\,(k-2r+1+j)!} \vspace{0.5em} $$
In particular, if $(r',i',j') \neq (r,i,j)$ then this limit is $\left\langle f,f' \right\rangle = 0$.
\end{thm}
\end{samepage}

\begin{rmk*}
This spectral decomposition of $W_{\kpa r}$ does not depend on $p_\TA$ and $p_\TB$, if these are respectively the limits of $n_\TA/n$ and $n_\TB/n$, as in Theorem~\ref{main3}. This remarkable property is not true in general, in the case of three samples or more.

In fact, the limit in this theorem is taken with respect to any $n_\TA$ and $n_\TB$ such that $\min(n_\TA,n_\TB) \to \infty$. This is a relaxation of the assumption of Theorem~\ref{main3} that $n_\TX/n$ converges to a positive constant $p_\TX$ for every~$\TX \in \Sigma$. For this reason, the formulation of Theorem~\ref{main4} restates the case~$r \neq r'$.
\end{rmk*}


\subsection{Examples}
\label{examples}

We list a variety of examples for subword statistics, as special cases of our treatment. We examine how they are scaled and classified according to the scheme of Theorems~\ref{main1}-\ref{main4}. All computations of the decompositions and second moments are straightforward from the definitions in~\S\ref{defs}, and can be done automatically. 

We keep the discussion brief, as our main purpose is not to study these particular examples but to demonstrate how various statistics from diverse contexts unify under one framework. Nevertheless, in several cases our perspective sheds new light on them, or points to potential generalizations.

\smallskip
\begin{ex}
Warm Up: Coin Flips
\end{ex}
\noindent
A sequence of $n$ tosses of a fair coin gives a word in $\{\THH,\TTT\}^n$, distributed by $\mathcal{W}(n,(\tfrac12,\tfrac12))$. The decomposition for $k=1$ gives $W_{10} =  \SPAN\{\THH+\TTT\}$ and $W_{11} =  \SPAN\{\THH-\TTT\}$. As Theorem~\ref{main1} claims, the former yields $\bar\#(\THH+\TTT) \equiv 1$, of constant order. The latter statistic, of order $1/\sqrt{n}$, is the ``observed bias'' of the coin under the fairness hypothesis. Computing the decomposition for $k=2$,
\begin{itemize}[leftmargin=15pt]
\item $W_{20}\; = \SPAN\{\THH\THH+\THH\TTT+\TTT\THH+\TTT\TTT\}$
\item $W_{210} = \SPAN\{\THH\THH-\TTT\TTT\}$
\item $W_{211} = \SPAN\{\THH\TTT-\TTT\THH\}$
\item $W_{220} = \SPAN\{\THH\THH+\TTT\TTT-\THH\TTT-\TTT\THH\}$
\end{itemize}

The first two come from $W_{10}$ and $W_{11}$ via the embedding $W_1 \hookrightarrow W_2$. The new statistic $\bar\#(\THH\TTT-\TTT\THH)$ may be interpreted as the tendency of tails to occur after heads. It also scales as $1/\sqrt{n}$, but Theorem~\ref{main2} implies that its variance is $\tfrac13$ of that of $\bar\#(\THH\THH-\TTT\TTT)$, and these two statistics are uncorrelated. By Theorem~\ref{clt}, their joined distribution is asymptotically binormal. The fourth statistic scales as~$1/n$ and leads to the next example.

\smallskip
\begin{ex}
Pearson's $\chi^2$ Test Statistic
\end{ex}
\noindent
The following holds up to a \emph{constant} correction of smaller order in~$n$:
$$ \bar\#(\THH\THH+\TTT\TTT-\THH\TTT-\TTT\THH)\;=\; 2\bar\#(\THH\THH+\TTT\TTT)-1 \;\approx\; \frac{ (\bar\#\THH - 0.5)^2}{0.5} + \frac{(\bar\#\TTT - 0.5)^2}{0.5} $$ 
This is the classical Pearson's $\chi^2$ test statistic for fitting the frequencies of $\THH$ and $\TTT$ to the distribution $(0.5,0.5)$~\cite{pearson1900criterion}. This fact extends to any finite-dimensional distribution vector~$\mathbf{p}$. The combination $\sum_{\TX}\bar\# \TX\TX/p_{\TX} - 1$, which is essentially Pearson's $\chi^2$ statistic, always lies in~$W_{220}$. 

\smallskip
\begin{samepage}
\begin{ex}
Functions on the Boolean Hypercube
\end{ex}
\noindent
Consider a binary stream $w \in \{\To,\Ti\}^n$, distributed with $\mathcal{W}(n,(p,q))$. The subword statistics of $w$ correspond to $\mathbb{R}\{\To,\Ti\}^k$, or equivalently Boolean functions $f:\{\To,\Ti\}^k \to \mathbb{R}$, so they take the form $\sum_u f(u)\,\bar\#u$. 
\end{samepage}

The primary decomposition of $\mathbb{R}\{\To,\Ti\}^k$ follows the so-called ``slices'' of the Fourier basis of Boolean functions. Namely, we expand all the ``monomials'' with $k-r$ times $(\To+\Ti)$ and $r$ times $(q\,\To-p\,\Ti)$ to obtain $\tbinom{\mathsmaller{k}}{r}$ combinations that span $W_{kr}$. For example, the expansion of $(q\,\To-p\,\Ti)^k$ from $W_{kk}$ gives the most concentrated statistic, with variance~$\sim k!/n^k$ by Theorems~\ref{main1}-\ref{main2}. For $p=q=\tfrac12$, it is the bias of the parities of $k$-bit subwords of~$w$.

The diagonalization in each slice introduces a finer decomposition into orthogonal subspaces, which, as we will see in \S\ref{poly}, correspond to special orthogonal polynomials.
For example, in order $1/\sqrt{n}$ we obtain the basis $\{P_i(1)f_1 + \dots + P_i(k)f_k\}_{0 \leq i < k}$, where $f_1,\dots,f_k$ are so-called ``dictatorship'' functions, and $P_i(x)$ are the orthogonal polynomials of the uniform measure on~$\{1,\dots,k\}$. As $n$ grows, these $k$ statistics tend to independent Gaussian distributions, see~\S\ref{gustats}.

\smallskip
\begin{ex}
Discrete L\'evy Area
\end{ex}
\noindent
A word in the four cardinal winds $\{\TEE,\TNN,\TWW,\TSS\}^n$, with $\mathbf{p} = (\tfrac14,\tfrac14,\tfrac14,\tfrac14)$, may represent a random walk of $n$ steps on the square grid $\mathbb{Z}^2$. Let
$$ a \;=\; (\TEE\TNN - \TWW\TNN + \TWW\TSS - \TEE\TSS) - (\TNN\TEE - \TNN\TWW + \TSS\TWW - \TSS\TEE) $$
Viewing the walk $w$ as a path in $\mathbb{R}^2$, the statistic $\#a(w)$ is the signed area $\int x\, dy - \int y\, dx$. This is a discrete analogue of the important \emph{L\'evy area} of a two-dimensional Brownian motion \cite{levy1951wiener,guivarc1977marches}. An automated computation can show that $a \in W_{22}$. Theorems~\ref{main1} and~\ref{main2} give the scaling and the asymptotic variance: $\Exp[\bar\#a^2] \sim 1/n^2$. This statistic of second order has a particularly simple limit law, $\#a/n \to A$ with $f_{A}(x)=\text{sech}(\pi x)$, same as the continuous L\'evy area.

If the walk $w$ terminates at the origin, then either half of $a$ yields $\#a' := \#a/2 = \oint x\,dy$, the enclosed algebraic area. Such closed random walks are modeled by words in the multisample $\mathcal{W}'(n,n,n,n)$. The terms of $a'$ have different compositions: $(1,1,0,0)$, $(0,1,1,0)$, etc., but all can be embedded in the space $W_{\kpa}$ for $\kpa=(1,1,1,1)$ using Proposition~\ref{embed}. 

Again, an automated computation shows that $a' \in W_{\kpa 2}$ and $\#a'$ scales as $n$ by Theorem~\ref{main3}. However, in this case, the normalized area $\#a'/4n$ tends to the logistic distribution, with density $f_{A'}(x) = \pi\, \text{sech}^2(2 \pi x)$, similar to a Brownian excursion in the continuous case. This limit was studied in the context of random knots~\cite{even2016invariants}, because $\#a'$ has the same distribution as a two-component linking number generated from petal diagrams and random permutations. Extensions to walks in~$\mathbb{Z}^d$ are interesting in the context of random 3-manifolds obtained via surgery from $d$-component links. A~related statistic was studied in~\cite{mingo1998distribution}.

\smallskip

\begin{ex}
Two-Sample Statistical Tests
\label{statistical}
\end{ex}
\noindent
Consider two real-valued samples $X_1,\dots,X_n$ and $Y_1,\dots,Y_m$ drawn independently from unknown continuous distributions, denoted by the random variables $X$ and~$Y$. The relative order of the observations induces a word $w$ over $\{\TX,\TY\}$ of length~$n+m$. For example, if $X_2 < Y_3 < Y_1 < X_1 < Y_2$ then $w=\TX\TY\TY\TX\TY$. If the two distributions coincide, denoted $X \sim Y$, then $w$~is exactly as in the random model~$\mathcal{W}'(n,m)$. This is the null hypothesis of several nonparametric tests for comparing two distributions. 

Persson~\cite{persson1979new} represents several two-sample test statistics in terms of subword counts in~$w$. We review these statistics below.
\begin{itemize}[leftmargin=15pt]
\itemsep0.25em
\item \textbf{Mann--Whitney U \cite{mann1947test, wilcoxon1945individual}.} 
This test statistic, $U = \#\TY\TX$ estimates how much $P(Y<X)$ deviates from $1/2$ for randomly selected $X$ and~$Y$. The equivalent combination $u = (\TY\TX - \TX\TY)/2$ lies in the component $W_{\kpa 1}$ where $\kpa=(1,1)$. The null distribution of~$\tilde\#u$ is asymptotically normal with variance ${\tfrac{1}{12}}(\tfrac{1}{m} + \tfrac{1}{n})$, reproduced by Theorems~\ref{main3}-\ref{main4}.
\item \textbf{Cram\'er--von Mises criterion~\cite{lehmann1951consistency}.}
The above $U$ might fail to detect $X \not\sim Y$ when the probability of $X<Y$ happens to be exactly~$1/2$. However, given four independent replications $X$, $X'$, $Y$, and $Y'$, the probability that $\max(X,X')<\min(Y,Y')$ or $\max(Y,Y')<\min(X,X')$ is $1/3$ if and only if $X \sim Y$. Otherwise, it is greater than $1/3$ by an $L^2$ difference between the distribution functions $F_X$ and~$F_Y$. This difference can be estimated by $2\,\tilde\#t$ for the following centralized combination in $W_{(2,2)}$:
$$ t \;=\; \tfrac13(\TX\TX\TY\TY + \TY\TY\TX\TX) - \tfrac16(\TX\TY\TY\TX + \TY\TX\TX\TY + \TX\TY\TX\TY + \TY\TX\TY\TX) $$
Theorems~\ref{main3}-\ref{main4} give $t \in W_{(2,2)201}$ and $\Var[\tilde\#t] \;\sim\; \tfrac{1}{45}(\tfrac{1}{m} + \tfrac{1}{n})^2$ in agreement with~\cite{anderson1962distribution}.
\item \textbf{Watson's $U^2$ \cite{watson1962goodness}.}
Now suppose that $\{X_i\}$ and $\{Y_i\}$ are samples on the circle~$S^1$. In this case, the previous test for $X \sim Y$ depends on an arbitrary choice of a starting point. Another notion of difference by Watson can be estimated by $\tilde\#s$, for the following \emph{rotation invariant} combination.
$$ s \;=\; \tfrac{1}{12}(\TX\TX\TY\TY + \TY\TY\TX\TX + \TX\TY\TY\TX + \TY\TX\TX\TY) - \tfrac16(\TX\TY\TX\TY + \TY\TX\TY\TX) $$
This is not a principal direction of the covariance, but $s = v + \tfrac14 t$ for $v \in W_{(2,2)200}$, and in fact $W_{(2,2)2} = \SPAN\{s,t\}$. By Theorem~\ref{main4}, $\Var[\tilde\#v] \sim \tfrac{1}{720}(\tfrac{1}{m} + \tfrac{1}{n})^2$, so $\Var[\tilde\#s] \sim \tfrac{1}{360}(\tfrac{1}{m} + \tfrac{1}{n})^2$. 
\end{itemize}
\vspace{0.25em}
Finally, we mention the possibility of similarly analyzing Cram\'er--von Mises type tests for the classical $K$-sample problem~\cite{kiefer1959k, puri1965some}. It is also possible to study such functionals with higher $L_p$~norms of $(F_X-F_Y)$ as word statistics. This may be interesting because the infinity norm gives another popular two-sample test by Kolmogorov--Smirnov.

\newcommand{\boxes}[1]{\raisebox{-2pt}{\tikz{
\foreach \i/\k in {#1} \foreach \j in {1,...,\i} 
\draw (0.12*\j, -0.12*\k) rectangle (0.12*\j+0.12, -0.12*\k+0.12);}}}

\smallskip
\begin{ex}
Simultaneous Core Partitions 
\end{ex}
\noindent
Representing a \emph{partition} with square boxes, for example \;\boxes{7/0,5/1,4/2}\,, the \emph{hook} of each box is the set of boxes directly to its right or below it, e.g.~\boxes{4/0,1/1,1/2}\, is the hook of box~4 in row~1 of our example. A~partition that avoids hooks with exactly $p$ boxes is \emph{$p$-core}, which arose in the study of $p$-modular representations. If $s$ and $t$ are coprime, then the number of partitions that are simultaneously $s$-core and $t$-core is finite, and equals~$\tfrac{1}{s+t}\tbinom{s+t}{s}$ as shown by Anderson~\cite{anderson2002partitions} using a clever bijection to words in~$\Sigma = \{\TS,\TT\}$.

Starting from a random word in the model $\mathcal{W}'(s,t)$, one can apply a suitable rotation and reverse Anderson's map, and obtain a uniformly distributed $(s,t)$-core partition. The perspective of word statistics is particularly useful for understanding its properties. It has been shown in~\cite{even2020sizes} that  $\tfrac{1}{24}(s^2-1)(t^2-1)-\tfrac12(\#\TS\TT\TS\TT + \#\TT\TS\TT\TS)$ gives the number of boxes in the random partition. This has proven a curious relation between the size distribution of $(s,t)$-core partitions and the null distribution of Watson's~$U^2$, due by Zeilberger~\cite{ekhad2015explicit}, and has simplified other results on this problem.

\smallskip
\begin{ex}
Intransitive Dice
\label{dice}
\end{ex}
\noindent
Consider three dice labeled $\{\TA,\TB,\TC\}$ with $n$ faces each, and assign the values $\{1,\dots,3n\}$ at random to their faces. Such a set of dice is described by a random word $w$ in $\mathcal{W}'(n,n,n)$, with $i$ assigned to the die~$w_i$. The \emph{bias} of $\TA$ vs $\TB$ is a random variable, measuring how much $\TA$ is more likely to win in a single match with $\TB$, defined $\beta_{\TA\TB} = \tilde\#(\TB\TA-\TA\TB)$ as a subword statistic. The set of dice $\TA,\TB,\TC$ is \emph{intransitive} if $\beta_{\TA\TB},\beta_{\TB\TC},\beta_{\TC\TA}$ are all positive or all negative, a surprising possibility brought up by Efron~\cite{gardner2001colossal}. For example, if $n=3$ and $w=\TB\TA\TC\TA\TC\TB\TC\TB\TA$ then $\beta_{\TA\TB}=\beta_{\TB\TC}=\beta_{\TC\TA}=\tfrac19$. 

In order to analyze the joint distribution of the biases, we first embed them in a common space $W_{\kpa}$ for $\kpa={(1,1,1)}$. We map $\TB\TA-\TA\TB \,\mapsto\, \TC\TB\TA+\TB\TC\TA+\TB\TA\TC-\TC\TA\TB-\TA\TC\TB-\TA\TB\TC$ and same for the other biases, in accordance with Proposition~\ref{embed}. The $1/\sqrt{n}$ scaling of each bias follows from the decomposition of Theorem~\ref{main3}. Their covariance matrix and higher moments were computed \emph{exactly for every~$n$} by Zeilberger~\cite{ekhad2017treatise}, who conjectured a multinormal limit distribution based on the leading terms. Indeed, this multinormal limit follows from our formulation and Theorem~\ref{clt}: 
$$ \sqrt{3n} \;\;\tilde\#\!\left[\begin{matrix} \TB\TA-\TA\TB \\ \TC\TB-\TB\TC \\ \TA\TC-\TC\TA \end{matrix}\right] \;\; \xrightarrow[\;\text{in distribution}\;]{n \to \infty} \;\; \mathcal{N}\left(\left[\begin{matrix} 0 \\ 0 \\ 0 \end{matrix}\right],\; \left[\begin{array}{rrr}
     2 & -1 & -1  \\
     -1 & 2 & -1  \\
     -1 & -1 & 2 
\end{array}\right]\right)$$

This asymptotic covariance is degenerate, with matrix-rank~2. The limit distribution is supported on the plane $x+y+z=0$ in~$\mathbb{R}^3$. As an immediate consequence, intransitivity occurs with probability tending to zero as $n \to \infty$, because this plane meets the two intransitive octants only at the origin.

One way to amplify the phenomenon of intransitivity is to consider random models that reduce the typical bias, as done in a recent PolyMath project~\cite{conrey2016intransitive,polymath2017blog,hkazla2020probability}. However, these models impose quantitative conditions on the face values, and abandon the distribution-free formulation of the problem. 

The viewpoint of subword spaces suggests to capture some notion of intransitivity by looking at the smaller order component $\beta_{\TA\TB}+\beta_{\TB\TC}+\beta_{\TC\TA}$, arising from the subword combination,
$$ g \;=\; \TC\TB\TA + \TB\TA\TC + \TA\TC\TB - \TA\TB\TC - \TB\TC\TA - \TC\TA\TB \;\in\; W_{\kpa 2} $$
We note that this statistic can be nonzero also in the event of transitive dice, but it may be viewed as the ``intransitive component'' in their biases.

Zeilberger~\cite{zeilberger2016doron} studied $\#g$ as a special case of the \emph{Gepner statistic}, and derived the \emph{Gepner polynomials} for its moments, whose leading terms suggest that $\#g/n^2$ converges in law to a logistic distribution. Indeed, $\#g$ may be viewed as the L\'evy area of the walk $w$ projected on the above supporting plane. This is discussed in Appendix~\S\ref{dice3}. Generalizations to sets of four dice or more are also natural in this representation, and left for another time.

\smallskip
\begin{ex}
Path Signature and Machine Learning
\end{ex}
\noindent 
Finally, we describe a potential application to machine learning, which will be investigated in future work. In many application areas, the data takes the form of a long random-like text over a finite alphabet. This may either be a stream of symbols that comes with a natural ordering or ``time'' parameter, or a mixture of $d$ samples of real-valued data points, as in Example~\ref{statistical}. Suppose that one wishes to classify, model, estimate a parameter, or learn a function of such sequences, say, by applying a neural network. Then, the input sequence first has to be summarized as a vector of \emph{characteristic features}, of reasonable length. 

The \emph{signature method} is a generic way of extracting feature sets for sequential data. The basic idea is to embed the data as a path $[0,1] \to \mathbb{R}^d$, and then to use features from its \emph{signature}, which is the graded sequence of its iterated integrals. The coordinates of the signature are definite integrals of the path $(x_t,y_t,\dots)_{0\leq t \leq 1}$ such as $\int_{t}dx_t$, $\int_{t}dy_t$, $\int_{t<s}dx_tdx_s$, $\int_{t<s}dx_tdy_s$, and so on. This method has achieved success in several recent machine learning applications to financial data, clinical symptoms, handwriting recognition, and more \cite[for overviews]{levin2013learning, chevyrev2016primer}. The notion of path signature originates in the fundamental theory of rough paths \cite{chen1958integration, lyons1998differential}.

Although the signature method has been applied mainly to vector-valued time series and spatial data,  a text over $d$ symbols also naturally embeds as a path in $\mathbb{R}^d$. Every appearance of a letter $\TX$ contributes a unit step along the axis that corresponds to~$\TX$. The signature of the resulting path is essentially the set of subword statistics in the given text, where the $k$th level corresponds to subwords of $k$ letters.

Now, our results on the diagonalization of the space of subword statistics provide a suitable choice of basis for feature selection in the signature. Such a basis may be crucial for addressing several important challenges, such as how and where to truncate the coordinates of the signature, how to adjust input parameters in specific applications, how to interpret the contribution of the various characteristic features, etc. 

One prediction we would like to make is that our suggested basis of attributes will actually be most beneficial in the \emph{high-noise} regime. Indeed, our decomposition diagonalizes the joint distribution under randomness, so it seems particularly preferable to use it when the signal is hidden in strong random noise. Under such assumptions, our basis comprises uncorrelated features and distinguishes between statistics that scale differently with the data length.

\subsection{Acknowledgments}
C.~E.~was supported by the Lloyd’s Register Foundation / Alan Turing Institute programme on Data-Centric Engineering.
T.~L.~was supported by the ISF grant 891/15 and ERC 2020 grant HomDyn 833423.
R.~T.~(incumbent of the Lillian and George Lyttle Career Development Chair) was supported by the ISF grant No. 335/19 and by a research grant from the Center for New Scientists of Weizmann Institute. 

\section{Decompositions}
\label{defs}

This section includes the definitions and constructions required for our main results in~\S\ref{results}, given in full detail. At the same time, it may serve as an overview of the tools we are going to use in order to establish them.

\subsection{Primary Decomposition in the One-Sample Model}\label{sec:Wkr decomp} 

\begin{definition}
\label{inner}
We equip the space $W_k = \mathbb{R}\Sigma^k$ with an inner product, induced by the probability measure $\mathcal{W}(k,\mathbf{p})$. Given two formal combinations of words $u \in \Sigma^k$, $f = \sum_u f_u u$ and $f' = \sum_{u} f'_u u$, let $\langle f,f' \rangle_{\mathbf{p}} = \sum_{u} p_u f_u f'_u$ where $p_u = p_{u_1}p_{u_2}\cdots p_{u_k}$ is the probability of $u$ under this distribution.
\end{definition}

Our decomposition of $W_k$ is defined via a basis of the space~$\mathbb{R}\Sigma$, denoted by numerals $D := \{\Ti, \Tii, \dots, \TD\}$. Let $\Ti := \sum_{\TX\in\Sigma} \TX$, and let $\Tii, \Tiii, \dots, \TD$ complete it to an orthonormal basis with respect to the above inner product, so that $\langle\TI,\TJ\rangle_{\mathbf{p}} = \delta_{\TI\TJ}$ for $\TI,\TJ \in D$. Note that the words $D^k$ are an orthonormal basis of~$W_k$.

\begin{ex*}
For $\Sigma = \{\TA,\TB,\TC\}$ and $\mathbf{p}=(\tfrac13,\tfrac13,\tfrac13)$, one possible choice is:
$$ \Ti = \TA+\TB+\TC,\;\;\; \Tii = \sqrt{\tfrac{3}{2}}(\TA-\TB),\;\;\; \Tiii = \sqrt{\tfrac{1}{2}}(\TA+\TB-2\TC) $$
\end{ex*}

\begin{definition}
\label{wkr}
$W_{kr} = \mathrm{span}\,\{e \in D^k : \#\Ti(e) = k-r\}.$
\end{definition}

It readily follows from this definition that $\dim W_{kr} = \tbinom{k}{r}(d-1)^r$, and $W_{k0},\dots,W_{kk}$ are pairwise orthogonal, and together span~$W_k$. In conclusion, this vector space carries a grading, 
$$ W_k \;=\; \mathbb{R}\Sigma^k \;=\; \mathbb{R}D^k \;=\; W_{k0} \oplus W_{k1} \oplus \dots \oplus W_{kk} $$
as stated in \S\ref{results}. The subword combinations $f \in W_{kr}$ are the homogeneous elements of degree~$r$ in this grading. The content of Theorem~\ref{main1} is that they give rise to statistics $\bar\#f$ of order~$n^{-r/2}$, and in fact it follows that they are U-statistics of rank~$r$, see~\S\ref{ustats}.

\begin{rmk*}
The component $W_{kr}$ is independent of the choice of the basis elements $\{\Tii,\Tiii,\ldots,\TD\}$. Indeed, $W_{kr}$ is characterized as the linear span of all $a_1 a_2\cdots a_k$ where $k-r$ of the $a_i$ are $\Ti$ and the other $a_i \in \Ti^{\perp} \subset \mathbb{R}\Sigma$.
\end{rmk*}

\subsection{The Algebra of Words} 
\label{operators}

Our set of tools includes several operations on words and word spaces. This algebraic approach follows and elaborates on the recent work of Dieker and Saliola~\cite{dieker2018spectral}.

We have already implicitly denoted the \emph{concatenation} of $u \in \Sigma^k$ and $v \in \Sigma^j$ by $uv \in \Sigma^{k+j}$, which bilinearly extends to formal sums. 

Another well-known bilinear operation, the \emph{shuffle} product $u \shuffle v$, is the formal sum of all $\tbinom{k+j}{k}$ ways to merge $u$ and~$v$, extended bilinearly to word sums. It is formally defined by the recursive rule $u\TX \shuffle v\TY = (u\TX \shuffle v)\TY + (u \shuffle v\TY)\TX$, where the empty word $\phi$ satisfies $\phi \shuffle v = v \shuffle \phi = v$. Clearly $u \shuffle v = v \shuffle u$. For fixed $v \in \Sigma^r$ and $k \geq 0$, we define the following \emph{insertion} operator.

\begin{definition}
\label{shuffle}
$\sh_v : \mathbb{R}\Sigma^k \to \mathbb{R}\Sigma^{k+r}$ is defined by ~$\sh_v u = u \shuffle v$.
\end{definition}

\begin{ex*}
$\sh_\TB \TA\TA = \TA\TA\TB + \TA\TB\TA + \TB\TA\TA $, $\;\;\sh_{\TA\TB}\TA = 2\TA\TA\TB + \TA\TB\TA$
\end{ex*}

Conversely, we define a \emph{deletion} operator, which sums all possible ways to remove an occurrence of a given subword. Formally, it is defined on words by $\del_{v\TX}u\TY = (\del_{v\TX}u)\TY + \delta_{\TX\TY}\del_v u$ where $\del_\phi v = v$ and $\del_v \phi = 0$ if $v \neq \phi$.

\begin{definition}
$\del_v : \mathbb{R}\Sigma^k \to \mathbb{R}\Sigma^{k-r}$ is the linear extension of the above $\del_v$.
\end{definition}

\begin{ex*}
$\del_\TA \TA\TA\TB\TA = 2\TA\TB\TA + \TA\TA\TB$, $\;\;\del_{\TA\TB}\TB\TB\TA\TA = 0$
\end{ex*}

These operators will sometimes be denoted $\sh_v^{(k)}$ or $\del_v^{(k)}$ to emphasize that they act on the space $\mathbb{R}\Sigma^k$. We will often apply $\sh_\TX$ or $\del_\TX$ where $\TX$ is a single letter from $\Sigma$ or~$D$, and we will abbreviate $\sh = \sh_\Ti$ and $\del = \del_\Ti$, as defined on $\mathbb{R}D^k = \mathbb{R}\Sigma^k$. These two operators will play important roles our following definitions and proofs.

$\del$ and $\sh$ are closely related. Assuming $\langle u,v\rangle = \delta_{uv}$ for words $u,v,w$ over an alphabet~$D$, then $\del_w$ and $\sh_w$ are \emph{dual} with respect to this inner product. Indeed, $\left\langle \del_w u, v \right\rangle = \left\langle u, \sh_w v \right\rangle$ follows from straightforward counting of ways to merge $v$ and $w$ into $u$. This readily extends to $\left\langle \del_w f,g \right\rangle = \left\langle f, \sh_w g \right\rangle$, for any formal sums $f,g \in \mathbb{R}D^*$.

\subsection{Full Decomposition in the One-Sample Model} 
\label{full}

The decomposition of every subspace $W_{kr}$ is defined in terms of the operator~$\del$. Since this space is spanned by words in $D^k$ with $\#\Ti=k-r$, it is annihilated by $k-r+1$ applications of~$\del$. Hence, the iterated deletion operators $\del, \del^2, \del^3, \ldots$ give a decomposition of~$W_{kr}$ if we take in each kernel the orthogonal complement of the next one, with respect to the inner product~$\langle-,-\rangle_{\mathbf{p}}$.

\begin{definition}
\label{wkrm}
$W_{krm} := \left(\ker \del^{k-r-m+1}\right) \cap \left(\ker \del^{k-r-m}\right)^{\perp} \subseteq W_{kr} $ 
\end{definition}

\begin{ex*}
For $|\Sigma|=2$, $W_{210} = \mathrm{span}\{\Ti\Tii + \Tii\Ti\}$, $\;W_{211} = \mathrm{span}\{\Ti\Tii - \Tii\Ti\}$ 
\end{ex*}

Theorem~\ref{main2} asserts that these $k-r+1$ spaces $W_{kr0}, W_{kr1}, \dots, W_{kr(k-r)}$ asymptotically diagonalize the covariance matrix of the word statistics in~$W_{kr}$. It also determines the leading term of the variance within each component. 

\begin{rmk*}
We will see in the proof in~\S\ref{proof2} that $\dim W_{krm} = \tbinom{m+r-1}{m}(d-1)^r$.
\end{rmk*}

\subsection{Universal Grading for Subword Statistics} \label{grading}

We now use the other operator~$\sh$. Let $v \in \Sigma^j$ and $w \in \Sigma^n$ where $j<k\leq n$. The following rule may be interpreted as the law of total probability when picking random locations:
$$ \bar\#v(w) \;=\; \sum_{u \in \Sigma^k}\bar\#v(u)\;\bar\#u(w) $$

\begin{ex*}
$\bar\#\TA \;=\; \bar\#\TA\TA  + \tfrac12\,\bar\#\TA\TB  + \tfrac12\,\bar\#\TB\TA$
\end{ex*}

This rule suggests natural embeddings between word spaces, that linearly extend the map $v \mapsto \sum \bar\#v(u)u$. These maps can be specified in terms of the operator~$\sh_\Ti:W_k \to W_{k+1}$ by the following observation.

\begin{prop}
\label{gradingw}
$\bar\#f(w) =  \bar\#\left(\tfrac{1}{k+1}\sh f\right)(w)\;$ for $f \in W_{k}$, $w \in \Sigma^n$, $n>k$. 
\end{prop}

The map $\sh$ is one to one, as shown in Propositions~\ref{wv} and~\ref{lem123}. Therefore, the identifications $\frac{1}{k+1}\sh:W_{k} \hookrightarrow W_{k+1}$ yield a common vector space for word statistics,
$$ W \;:=\; \bigcup_{k \geq 0} W_k $$
It is straightforward from the definitions that these embeddings respect the primary decomposition of~$W_k$, meaning
$\sh\,W_{kr} \subseteq W_{(k+1)r}$
for all $r \leq k$. 

The following alternative definition for the secondary components $W_{krm}$ will follow from Proposition~\ref{wv} and Lemma~\ref{shwkrm}: 
$$ W_{krm} \;=\; \sh^{k-r-m} \ker \left( \del |_{W_{(r+m)r}} \right) $$
Therefore, these components are identified by
$\sh\,W_{krm} = W_{(k+1)rm}$
for every $m \leq k-r$. In conclusion, there exists a well-defined double grading,
$$ W \;=\; W_0 \,\oplus\, \bigoplus_{r = 1}^{\infty} \bigoplus_{m = 0}^{\infty} W_{(r+m)rm} $$
The identification by $\sh$ respects the orthogonality of components, each one scaling by a constant factor $\beta_{krm}/k^2$, see the proof of Lemma~\ref{shwkrm}. Note that this scaling is anisotropic between different~$W_{krm}$. For example, $W_{210}$ and $W_{211}$ as given in~\S\ref{full} rescale differently, and thus $\langle\Ti\Tii,\Tii\Ti\rangle_{\mathbf{p}}=0 \neq \langle\sh\Ti\Tii,\sh\Tii\Ti\rangle_{\mathbf{p}}$.

\subsection{Structure of the Components} 

Our next goal is to make the components $W_{krm}$ as explicit and meaningful as possible. Since they are defined in terms of $\del = \del_\Ti$, and regard all other $\{\Tii,\dots,\TD\}$ the same, they admit the tensor structure described below.

\begin{definition}
\label{vkrm}
In the special case $|\Sigma|=2$ we write $V$ instead of $W$:
\begin{enumerate}
\itemsep0.25em
\item 
$ V_{kr} := \mathrm{span}\,\left\{e \in \{\Ti,\Tii\}^k : \#\Ti(e) = k-r\right\} $
\item 
$V_{krm} := \left(\ker \del^{k-r-m+1}\right) \cap \left(\ker \del^{k-r-m}\right)^{\perp} \subseteq V_{kr} $ 
\end{enumerate}
\end{definition}

\begin{definition}
\label{phikr}
Let $k \geq r \geq 0$. The isomorphism
$$ \Phi_{kr} : W_{kr} \;\xrightarrow{\sim}\; V_{kr} \otimes (\Ti^\perp)^{\otimes r} $$
is defined via $\Phi_{kr}(e) = \pi(e) \otimes \rho(e)$
for every basis word $e\in D^k$, where
\label{pyro}
\begin{itemize}
\itemsep0.25em
\item 
$\pi(e) \in \{\Ti,\Tii\}^k$ is obtained from $e$ by replacing every $\TI \neq \Ti$ by $\Tii$.
\item 
$\rho(e) \in D^{k - \#\Ti(e)}$ is obtained by removing all occurrences of~$\Ti$.
\end{itemize}
\end{definition}

\begin{ex*}
$\Phi_{53}(\Ti\Tii\Tiii\Ti\Tiii) = \Ti\Tii\Tii\Ti\Tii \otimes \Tii\Tiii\Tiii\;$ , $\;\pi(\Ti\Ti\Ti) = \Ti\Ti\Ti$,
$\;\rho(\Ti\Ti\Ti) = \phi$
\end{ex*}

This factorization is compatible with the $\del_\Ti$ operator via:
$\Phi_{kr}(\del e) = (\del\pi(e)) \otimes \rho(e)$.
It also respects the inner product, because: $\langle \pi(e),\pi(e')\rangle_{\mathbf{p}} \cdot \langle \rho(e),\rho(e')\rangle_{\mathbf{p}} = \delta_{e,e'}$. Therefore,

\begin{prop}
\label{wv}
$\Phi_{kr}$ induces $W_{krm} \;\cong\; V_{krm} \otimes (\Ti^\perp)^{\otimes r}\;$ for every $m \in \{0,\dots,k-r\}$.
\end{prop}

\subsection{Discrete Orthogonal Polynomial Spaces} 
\label{poly}

It now remains to explore the structure of~$V_{krm}$. The following set of definitions characterizes $V_{krm}$ and thereby $W_{krm}$ using spaces of polynomials. 

\begin{definition}
\label{deltakr}
Consider the following discrete $r$-simplex in the integer grid.
$$ \Delta_{kr} \;:=\; \left\{\left(d_0,\dots,d_r\right) \in \mathbb{Z}^{r+1} \;\left|\; \begin{aligned} & d_0 \geq 0, \; d_1 \geq 0, \; \dots  \\ &d_0 + \ldots + d_r = (k-r) \end{aligned} \right.\right\} $$
Note that $\left|\Delta_{kr}\right| = \binom{k}{r}$. The \emph{bilinear pairing} with respect to $\Delta_{kr}$ of two $(r+1)$-variate real polynomials $P,Q \in \mathbb{R}[x_0,\dots,x_r]$ is defined as 
$$ \left\langle P, Q\right\rangle_{kr} \;:=\; \sum_{\,\mathbf{d} \in \Delta_{kr}} P(\mathbf{d}) \, Q(\mathbf{d}) $$
\end{definition}

\begin{ex*} $\Delta_{31} = \{(0,2),(1,1),(2,0)\}\;$,
$\;\left\langle x_1^2, \,1 \right\rangle_{31} = 4\cdot1 + 1\cdot1 + 0\cdot1 = 5$
\end{ex*}

\begin{ex*} $\Delta_{32} = \{(1,0,0),(0,1,0),(0,0,1)\}\;$,
$\;\left\langle x_1, \,x_2 \right\rangle_{32} = 0$
\end{ex*}

\begin{definition}
\label{ukrm}
Let $\mathbb{R}_m\left[x_1,\dots,x_r\right]$ be the subspace of polynomials of total degree at most~$m$ in $r$ variables, excluding $x_0$. The \emph{orthogonal polynomial spaces} $U_{krm}$ are recursively defined as follows.
$$ U_{krm} \;:=\; \left\{ P \in \mathbb{R}_m\left[x_1,\dots,x_r\right] \;\left|\; \begin{array}{c} \forall\,Q \in U_{kr0} \cup U_{kr1} \cup \dots \cup U_{kr(m-1)}, \vspace{0.5em} \\  \left\langle P, Q\right\rangle_{kr} = 0 \end{array} \right.\right\} $$
\end{definition}

\begin{ex*} Applying Gram--Schmidt:
$U_{310} = \mathrm{span}\left\{1\right\}$,  
$U_{311} = \mathrm{span}\left\{x_1-1\right\}$, \\
$U_{312} = \mathrm{span}\left\{3x_1^2-6x_1+1\right\}$, and $U_{31m} = \{0\}$ for $m>2$.
\end{ex*}

\begin{ex*} $U_{320} = \mathrm{span}\left\{1\right\}$,  
$U_{321} = \mathrm{span}\left\{3x_1-1,\,2x_2+x_1-1\right\}$.
\end{ex*}

\begin{definition}
\label{psikr}
Consider the map $\Psi_{kr} :\, \mathbb{R}[x_0,\dots,x_r] \to V_{kr}$ defined by
$$ \Psi_{kr}(P) \;:=\; \sum_{\mathbf{d} \in \Delta_{kr}} P(\mathbf{d})\;\Ti^{d_0}\Tii\,\Ti^{d_1}\Tii\,\Ti^{d_2} \cdots \Tii\,\Ti^{d_r} $$
\end{definition}

\begin{ex*}
$\Psi_{42} \left(x_0^2 \right) = 2^2\,\Ti\Ti\Tii\Tii + 1^2\,\Ti\Tii\Ti\Tii + 1^2\,\Ti\Tii\Tii\Ti$
\end{ex*}

\begin{ex*}
$\Psi_{31} \left(3x_1^2-6x_1+1\right) = \Tii\Ti\Ti - 2\cdot\Ti\Tii\Ti + \,\Ti\Ti\Tii$
\end{ex*}

\begin{rmk*}
Note that $\langle P, P' \rangle_{kr} = \left\langle \Psi_{kr}(P), \Psi_{kr}(P') \right\rangle_{\mathbf{p}}$.
\end{rmk*}

By the next theorem, the restriction of $\Psi_{kr}$ to $\R_{k-r}[x_1,\ldots,x_r]$ is an isometry to the word space~$V_{kr}$, such that the orthogonal polynomial spaces map to the components of word statistics.

\begin{alsotheorem}
\label{structure}
$\Psi_{kr}$ induces $ U_{krm} \cong V_{krm} \; $ for every $m \in \{0,\dots,k-r\}$.
\end{alsotheorem}

The proof of Theorem~\ref{structure} is given in~\S\ref{sec23}. Together with $\Phi_{kr}$ from Proposition~\ref{wv} above, it allows the construction of explicit orthogonal bases for the word statistics in every~$W_{krm}$. 

At the end of~\S\ref{sec23}, we discuss a more symmetric description of these spaces using \emph{homogeneous} polynomials in $\mathbb{R}_m[x_0,\dots,x_r]$. That representation suggests further refinements of the components.

\begin{rmk*}
In the case $r=1$, the resulting polynomials are, up to a simple reparameterization, the so-called discrete Chebyshev polynomials of the second type. These polynomials serve as eigenvectors of the standard representation $S^{(k-1,1)}$ of the symmetric group in the context of card shuffling, as was observed in \cite[\S5.2]{UR02}. Our decomposition generalizes them to the multivariate setting.
\end{rmk*}

\subsection{Statistics in the Multi-Sample Model} 
\label{multi}

Theorems~\ref{main3} and~\ref{main4} concern the random model $\mathcal{W}'(\mathbf{n})$, where $\mathbf{n} = (n_\TA,n_\TB,\dots)$ and every letter~$\TX$ occurs exactly $n_{\TX}$ times. The statistics under consideration are word combinations in $W_{\kpa} = \mathbb{R}\tbinom{\mathsmaller{\Sigma}}{\kpa}$ where $\kpa = (k_\TA,k_\TB,\dots)$, so that every letter~$\TX$ appears $k_{\TX}$ times. Before studying the structure of the space~$W_{\kpa}$, we explain why it is sufficient to consider this kind of combinations, with words of the same composition~$\kpa$.

Recall the normalized statistics $\tilde\#f(w) = \#f(w) / \prod_\TX \tbinom{n_\TX}{k_\TX}$ from \S\ref{results}. It follows that for every $f \in W_{\kpa}$ and $\TX \in \Sigma$,
$$ \tilde\#f(w) \;=\; \tilde\#\left[ \frac{1}{k_\TX+1}\sh_\TX f\right](w) $$
assuming that the word $w$ satisfies $\#\TX(w) > k_{\TX}$. Hence subword statistics of composition~$\kpa$ can also be expressed by statistics of composition $\kpa + \TX := (k_\TA, k_\TB,\dots,k_\TX+1,\dots)$. By iterating, one can similarly express $f$ by words of composition ${\kpa+\TX+\TY}$ for any $\TX,\TY \in \Sigma$, with possibly~$\TX=\TY$. Note that $\sh_\TX$ and $\sh_\TY$ commute, so the resulting combination only depends on~$\TX+\TY$. In general, the word statistic $f \in W_{\kpa}$ can be expressed in every $W_{\kpa'}$ such that $\kpa' \geq \kpa$ \emph{pointwise}, meaning
$k_\TX' \geq k_\TX$ for every $\TX \in \Sigma$. The maps $\sh_\TX$ are injective, as restrictions of those considered in~\S\ref{grading}. In conclusion, 

\begin{prop}
\label{embed}
For every $\kpa \leq \kpa'$, there exists an embedding $\iota: W_{\kpa} \hookrightarrow W_{\kpa'}$ such that $\tilde\# [\iota f](w) = \tilde\#f(w)$ for all $f \in W_{\kpa}$, $w \in \tbinom{\mathsmaller{\Sigma}}{\mathbf{n}}$, $\mathbf{n} \geq \kpa'$.
\end{prop} 

This proposition justifies our focus on the linear spaces $W_{\kpa}$.
Statistics that combines subword counts from different compositions can always be expressed by some combination in a common larger $\kpa$.

\begin{ex*}
$\tilde\#\TA\TA\TB + \tilde\#\TA\TB\TB \;=\; \tilde\#\left[2\,\TA\TA\TB\TB + \TA\TB\TA\TB + \frac12\,\TA\TB\TB\TA + \frac12\,\TB\TA\TA\TB \right]$ 
\end{ex*}

Moreover, one can use the embeddings $W_{\kpa} \hookrightarrow W_{\kpa'}$ to identify these spaces, with well-defined statistics $\tilde\#f$. By the commutativity mentioned before, these identifications are compatible with each other, and yield a common space for all word statistics on~$\mathcal{W}'(\mathbf{n})$:
$$ W_{*} \;:=\; \bigcup_{\kpa \geq \mathbf{0}} W_{\kpa}$$

Finally, we use the following notation for the standard inner product on every space~$W_{\kpa}$.

\begin{definition}
\label{inner2}
Let $\langle-,-\rangle$ be the inner product on $W_{\kpa}$ that makes the words in $\tbinom{\mathsmaller{\Sigma}}{\kpa}$ an orthonormal basis.
\end{definition}

Note that this inner product differs by a constant factor from Definition~\ref{inner}, used for the spaces $\mathbb{R}\Sigma^k$. In particular, it does not depend on a parameter of the model such as~$\mathbf{p}$. No confusion should arise because these spaces are studied in different random model.

\subsection{Reordering, Representations and Replacement} 
\label{repr}

The symmetric group acts on words by \emph{reordering}. Given a permutation $\tau \in S_k$ and a word $u = u_1\cdots u_k \in \Sigma^k$, we let $u \tau = u_{\tau(1)}\cdots u_{\tau(k)}$. 

This action linearly extends to the group ring $\mathbb{R}S_k$ and to formal sums in~$\mathbb{R}\Sigma^k$, or to the subspace $W_{\kpa} = \mathbb{R}\tbinom{\mathsmaller{\Sigma}}{\kpa}$. We denote $A := \mathbb{R}S_k$.

\begin{ex*}
$(\TA\TA\TB\TC + 8\,\TC\TB\TA\TA) \, (\textrm{id} - (2341)) \,=\, \TA\TA\TB\TC + 8\,\TC\TB\TA\TA - \TA\TB\TC\TA  - 8\,\TB\TA\TA\TC$
\end{ex*}

As another example, consider the following \emph{averaging} operator on the space $W_{\kpa}$, defined in terms of $A$'s action.

\begin{definition}
\label{stab}
For $I \subseteq \{1,\dots,k\}$, let $a_I := \sum_{\tau \in \stab I} \tau$,
where the pointwise stabilizer, $\stab I := \left\{\tau \in S_k \,\mid\, \forall i \in I, \tau(i) = i \right\}$.
\end{definition}

\begin{ex*}
$(\TA\TA\TA\TA\TB)a_{\{1,2\}} \,=\, 2\,\TA\TA\TA\TA\TB + 2\,\TA\TA\TA\TB\TA + 2\,\TA\TA\TB\TA\TA$
\end{ex*}

Since $W_{\kpa}$ is a right $A$-module, it decomposes into simple representations of the symmetric group~$S_k$. This decomposition is a classical topic. Here we recall some necessary definitions and results, and refer the reader to~\cite[Lecture~4]{fulton2013representation} or \cite[\S 2.11]{sagan2013symmetric}. Our notation is similar to that of~\cite[\S 5.4.2]{dieker2018spectral}, though we use letters rather than numerals.

We throughout use the \emph{lexicographical} ordering $\TA < \TB < \TC < \cdots$ on the finite alphabet~$\Sigma$. A~finite sequence of integers $\lmda = (\lambda_\TA, \lambda_\TB, \lambda_\TC, \dots)$ is called a \emph{partition} if $\lambda_\TA \geq \lambda_\TB \geq \lambda_\TC \geq \cdots > 0$, which is denoted by $\lmda \vdash k$ if $k = \sum_\TX \lambda_\TX$. As noted in the introduction, without loss of generality it is sufficient to study $W_{\kpa}$ where the word composition $\kpa$ is a \emph{partition}.

Every partition $\lmda$ corresponds to a simple $A$-module, as follows. Fix the word $\alpha_{\lmda} := \TA^{\lambda_\TA}\TB^{\lambda_\TB}\TC^{\lambda_\TC}\cdots \in \tbinom{\mathsmaller{\Sigma}}{\lmda}$, and consider the subgroup $Q_{\lmda} \leq S_k$ containing all permutations that permute the positions of the 1st occurrence of each letter in~$\alpha_{\lmda}$, the positions of the 2nd occurrences, and so on. Consider the element $b_{\lmda} := \sum_{\tau \in Q_{\lmda}} \mathrm{sign}(\tau) \tau \in A$. We remark that other choices of $\alpha$ having composition $\lmda$ and other ``transversal'' subgroups $Q$ work as well.

\begin{definition}
\label{specht}
The \emph{Specht} $A$-module of $\lmda \vdash k$ is $S^{\lmda} := \alpha_{\lmda} b_{\lmda} A \subseteq W_{\lmda}$.
\end{definition}

\begin{ex*}
$\lmda = (3,2)$, $\alpha_{\lmda} = \TA\TA\TA\TB\TB$, $Q_{\lmda} = S_{\{1,4\}} \times S_{\{2,5\}} \times S_{\{3\}}$, $b_{\lmda} = \mathrm{id} - (14) - (25) + (14)(25)$, $S^{\lmda} = (\TA\TA\TA\TB\TB - \TA\TB\TA\TB\TA - \TB\TA\TA\TA\TB + \TB\TB\TA\TA\TA)A$.
\end{ex*}

The modules $S^{\lmda}$ for $\lmda \vdash k$ are all the simple $A$-modules up to isomorphism. In order to find the simple $A$-submodules of~$W_{\kpa}$, we introduce another word operator. A \emph{table} of words $T = (t_\TA, t_\TB, \dots)$ assigns a word over  $\Sigma$ to every letter in~$\Sigma$. Every table $T$ has a \emph{shape} $\lmda = \lmda(T)  = (\lambda_\TA,\lambda_\TB,\dots)$ such that $\lambda_\TX = |t_\TX|$ for every $\TX \in \Sigma$, and a \emph{composition} $\kpa = \kpa(T) = (k_\TA,k_\TB,\dots)$ such that $k_{\TX} = \sum_{\TY}\#\TX(t_{\TY})$. Note that $\lmda$ and $\kpa$ do not have to be partitions.

\begin{ex*}
$T = \left[\begin{smallmatrix*}[l]t_\TA\\t_\TB\\t_\TC\end{smallmatrix*}\right] = \left[\begin{smallmatrix*}[l]\TA\TA\TB\\\TB\TC\\\TC\end{smallmatrix*}\right]$ with $\lmda = (3,2,1)$ and $\kpa = (2,2,2)$.
\end{ex*}

\begin{samepage}
\begin{definition}
\label{Theta}
Consider a table $T$ of shape $\lmda$ and composition $\kpa$. The \emph{replacement} operator, 
$$ \Theta[T] : W_{\lmda} \;\to\; W_{\kpa} $$
maps every word $u \in \tbinom{\mathsmaller{\Sigma}}{\lmda}$ to the sum of all words in $\tbinom{\mathsmaller{\Sigma}}{\kpa}$ that are obtained from $u$ by replacing the $\TA$s by the letters of $t_\TA$ in any order, the $\TB$s by the letters $t_\TB$ in any order, and so on. This extends linearly to~$W_{\lmda}$.
\end{definition}
\end{samepage}

\begin{ex*}
If $t_\TA = \TA^{k_\TA}$, $t_\TB = \TB^{k_\TB}$, etc., then $\Theta[T]$ is the identity map on~$W_{\kpa}$.
\end{ex*}

\begin{ex*}
With $t_\TA = \TA\TA\TB$, $t_\TB = \TB\TC$ and $t_\TC = \TC$ as above,
$$ \Theta\left[\begin{smallmatrix*}[l]\TA\TA\TB\\\TB\TC\\\TC\end{smallmatrix*}\right](\TC\TB\TB\TA\TA\TA) \;=\; \TC\TB\TC\TA\TA\TB + \TC\TB\TC\TA\TB\TA + \TC\TB\TC\TB\TA\TA + \TC\TC\TB\TA\TA\TB + \TC\TC\TB\TA\TB\TA + \TC\TC\TB\TB\TA\TA $$
\end{ex*}

$\Theta[T]$ is equivariant under the actions of~$S_k$. Hence $\Theta[T]:S^{\lmda} \to W_{\kpa}$ yields either~0 or an isomorphic copy of the $A$-module $S^{\lmda}$. A table $T = (t_\TA,t_\TB,\dots)$ is called \emph{semistandard} if $\kpa(T)$ and $\lmda(T)$ are partitions, every word $t_\TX$ is weakly increasing, and every column $(t_{\TA i},t_{\TB i},\dots)$ is strictly increasing. Thus, the two examples above are semistandard. \emph{Young's Rule} says that a semistandard $T$ gives a nonzero copy of $S^{\lmda}$ in $W_{\kpa}$, and together they provide a decomposition into simple $A$-modules, as follows:
$$ W_{\kpa} \;=\; \bigoplus_{\substack{T\;\mathrm{semistandard} \\[2pt] \kpa(T) = \kpa}}\, \Theta[T] S^{\lmda(T)} $$

\begin{ex*}
$W_{(2,1,1)} \;=\; S^{(2,1,1)} \;\oplus\; \Theta\left[\begin{smallmatrix*}[l]\TA\TA\\\TB\TC\end{smallmatrix*}\right]S^{(2,2)} \;\oplus\; \Theta\left[\begin{smallmatrix*}[l]\TA\TA\TB\\\TC\end{smallmatrix*}\right]S^{(3,1)} \;\oplus\; \Theta\left[\begin{smallmatrix*}[l]\TA\TA\TC\\\TB\end{smallmatrix*}\right]S^{(3,1)} \;\oplus\; \Theta\left[\begin{smallmatrix*}[l]\TA\TA\TB\TC\end{smallmatrix*}\right]S^{(4)}$
\end{ex*}

The multiplicity of $S^{\lmda}$ in $W_{\kpa}$ is the so-called \emph{Kostka number} $K_{\kpa\lmda}$, which is the number of semistandard tables with these $\kpa$ and $\lmda$. For example, $K_{211,31} = 2$ is demonstrated above. Note that $K_{\kpa\lmda}$ vanishes if $k_\TA > \lambda_\TA$.

\subsection{Primary Decomposition in the Multi-Sample Model} 
\label{decomposition3}

The decomposition $W_{\kpa} \;=\; W_{\kpa0} \oplus W_{\kpa1} \oplus \dots \oplus W_{\kpa(k-k_\TA)}$ is based on Young's Rule, grouping together submodules with the same $\lambda_\TA$ as follows.

\begin{definition}
\label{Wkkr}
Let $\kpa = (k_\TA,k_\TB,\dots) \vdash k$ and $r \in \{0,1,\dots,k-k_\TA\}$.
$$ W_{\kpa r} \;:=\; \bigoplus_{\substack{T\;\mathrm{semistandard} \\[2pt] \kpa(T) = \kpa \\[1pt] \lambda_\TA(T)=k-r}}\, \Theta[T] S^{\lmda(T)} $$
\end{definition}

The components $W_{\kpa r}$ are pairwise orthogonal, with respect to the inner product of Definition~\ref{inner2}. This follows from the orthogonality of copies of different $S^{\lmda}$ in Young's Rule.

This completes the required definitions for Theorem~\ref{main3}, which asserts that the normalized subword statistics $\tilde\#f$ for $f \in W_{\kpa r}$ have order~$n^{-r/2}$, and statistics of different components are asymptotically uncorrelated. 

The proof of Theorem~\ref{main3} in~\S\ref{proof3} provides equivalent descriptions of the components $W_{\kpa r}$ of~$W_{\kpa}$. See Lemmas~\ref{easy}, \ref{hard}, \ref{stabsum}. Together with the above-mentioned pairwise orthogonality, these properties provide some shortcuts for practical computation of the primary decomposition, rather than applying Young's Rule.

\subsection{Replacement, Projection, Lifting, and Card-Shuffling Operators} 
\label{moreoperators}

Before presenting the refined decomposition of subword statistics in the two-sample model, we enhance our toolbox with several more linear word operators from~\cite{dieker2018spectral}. First, here is an abbreviation for a special case of the replacement operator.

\begin{definition}
\label{thetaab}
Let $\TX,\TY \in \Sigma$ and $\lmda = (\lambda_\TA, \lambda_\TB,\dots)$. 
$$ \Theta_{\TX\TY} \;:=\; \Theta[T] : W_{\lmda} \;\to\; W_{\lmda-\TX+\TY} $$
where $T = (t_\TA,t_\TB,\dots)$ is such that $t_\TX = \TX\cdots\TX\TY$ and $t_\TZ = \TZ\cdots\TZ$ for every $\TZ$ other than~$\TX$.
\end{definition}

\begin{ex*}
$\Theta_{\TA\TB}: W_{(4,6)} \to W_{(3,7)}\;$ is the map $\Theta[T]$ for
$T = \left[\begin{smallmatrix*}[l]t_\TA\\t_\TB\end{smallmatrix*}\right] = \left[\begin{smallmatrix*}[l]\TA\TA\TA\TB\\\TB\TB\TB\TB\TB\TB\end{smallmatrix*}\right]$.
\end{ex*}

\begin{ex*}
$\Theta_{\TA\TB}\, \TS\TA\TB\TA\TB\TA = \TS\TB\TB\TA\TB\TA + \TS\TA\TB\TB\TB\TA + \TS\TA\TB\TA\TB\TB$
\end{ex*}

In other words, the linear operator $\Theta_{\TX\TY}$ maps every word to the sum of all words obtained by replacing one occurrence of $\TX$ by~$\TY$. It is defined on all word spaces $W_{\lmda}$ where it is understood to give 0 on words with no occurrence of~$\TX$.

\smallskip
In the two-letter case $\kpa = (k_\TA,k_\TB)$, it is easy write the decomposition of $W_{\kpa}$ into $W_{\kpa r}$ by applying Young's rule from \S\ref{repr}-\S\ref{decomposition3}. The result is as follows.

\begin{lem} 
\label{lem:Wkr=Specht}
Let $k_\TA \geq k_\TB \geq 0$.
$$ W_{(k_\TA,k_\TB)} \;=\; \Theta_{\TA\TB}^{k_{\TB}} S^{(k_\TA+k_\TB,0)}  \,\oplus \dots \oplus\, \Theta_{\TA\TB}^2 S^{(k_\TA+2,k_\TB-2)} \,\oplus\, \Theta_{\TA\TB} S^{(k_\TA+1,k_\TB-1)} \,\oplus\, S^{(k_\TA,k_\TB)}$$
Therefore, for $\kpa = (k_\TA,k_\TB)$ and $k = k_\TA+k_\TB$, we have $W_{\kpa r} = \Theta_{\TA\TB}^{k_\TB-r} S^{(k-r,r)}$.
\end{lem}

\begin{rmk*}
In the notation of~\cite{dieker2018spectral}, $\,M^{(k-r,r)} := W_{(k-r,r)}$. In fact $M^{(k-r,r)}$ and~$S^{(k-r,r)}$ more often denote these modules when equivalently viewed as submodules of $\mathbb{R}S_k$ or $\mathbb{C}S_k$ rather than linear word spaces. In the proof of Theorem~\ref{main4}, we occasionally use this notation when more appropriate.
\end{rmk*}

The \emph{projection} from the word space $W_{\kpa}$ to one of its direct summands is another useful operator at our service. In the two-letter case, for $r \in \{0,\dots,k_\TB\}$ we denote it by
$$ \proj_r : W_{(k_\TA,k_\TB)} \to W_{(k_\TA,k_\TB)r} $$
In the proof, we mostly consider the component $r=k_{\TB}$, in which case the projection is denoted $\proj^{\lmda} : M^{\lmda} \to S^{\lmda}$ where $\lmda = (k_{\TA},k_{\TB})$.
These projections on $S^{\lmda}$ are a special case of the isotypic projector associated with the Specht module~$S^{\lmda}$. They can be computed based on the character of this module, as follows.

\begin{definition}
\label{proj}
For $\lmda \vdash k$, let
$\proj^{\lmda} : f \mapsto f\,\pi_{\lmda}\;$ where $\displaystyle \;\pi_{\lmda} = \frac{\dim S^{\lmda}}{k!} \sum_{\sigma \in S_k} \overline{\chi_{\lmda}(\sigma)}\,\sigma \in A$. 
\end{definition}

The next operator, \emph{lifting} which has been introduced in~\cite{dieker2018spectral}, plays a crucial role in the spectral decomposition. In general, it maps words from $W_{\kpa}$ to $W_{\kpa + \TX}$, adding a letter $\TX$ to the composition of the word. In contrast to Proposition~\ref{embed}, it does not simply insert~$\TX$, but also includes correction terms that make sure that the result lies in the right Specht module. Below we define the cases relevant to the two-sample model.
\begin{definition}
\label{levelup}
The operators $\mathcal{L}_\TA:W_{(k_\TA,k_\TB)} \to W_{(k_\TA+1,k_\TB)}$ and $\mathcal{L}_\TB:W_{(k_\TA,k_\TB)} \to W_{(k_\TA,k_\TB+1)}$ are defined as follows.
\begin{align*}
\shp_{\TA}\, f \;&:=\; \sh_\TA \,f \\ 
\shp_{\TB}\, f \;&:=\; \sh_\TB \,f \;-\; \frac{1}{k_\TA - k_\TB + 1} \,\Theta_{\TA\TB} \,\sh_\TA\, f 
\end{align*}
\end{definition}

Finally, we mention the \emph{random to random} operator arising from the analysis of card shuffling, which is a main object of study in~\cite{dieker2018spectral}, and is needed here as well. Here we write it in the two-letter case.

\begin{definition}
\label{r2r}
$\;\mathcal{R} \,:=\, \sh_\TA\, \del_\TA \,+\, \sh_\TB\, \del_\TB$
\end{definition}

\begin{ex*}
$\mathcal{R}\,\TA\TA\TB \;=\; 5\,\TA\TA\TB + 3\,\TA\TB\TA + \TB\TA\TA $
\end{ex*}

This operator sums over all the ways to remove a letter from the word and insert it back at some place. One of its important properties is that it can be represented as right multiplication by an element of~$A$.

\subsection{Full Decomposition in the Two-Sample Model} 
\label{twosample}

In the two-letter case, we define the following refinement of the primary decomposition of $W_{(k_\TA,k_\TB)}$ from Definition~\ref{Wkkr}.

\begin{definition}\label{wkrij}
Let $\kpa=(k_\TA,k_\TB)$, such that $k_{\TA} \geq k_{\TB} \geq 1$ and $k = |\kpa| = k_\TA+k_\TB$. For every $r \in \{1,\dots,k_{\TB}\}$ we define the following $r(k-2r+1)$ submodules of $W_{\kpa r}$.  
$$ W_{\kpa rij} \;:=\; \Theta_{\TA\TB}^{k_b-r}\;\shp_\TB^j \; \sh_\TA^i \;\ker\left(\del_\TA {\Big |}_{\displaystyle W_{(k-r-i,r-j),r-j}}\right) \;\;\;\;\;\;\;\;\;\; \begin{aligned}
&i \in \{0,\dots,k-2r\}
\\    
&j \in \{0,\dots,r-1\}
\end{aligned} $$ 
\end{definition}

\begin{rmk*}
Note that $W_{(k-r-i,r-j),r-j} = S^{(k-r-i,r-j)}$, as in Lemma~\ref{lem:Wkr=Specht}.
\end{rmk*}

Theorem~\ref{main4} states that this is the spectral decomposition of the covariance matrix of statistics from $W_{(k_\TA,k_\TB)}$ in the random model $\mathcal{W}'(n_a,n_b)$. It will be shown as part of the proof in~\S\ref{proof4} that for every $r \in \{1,\dots,k_\TB\}$ these components yield an orthogonal decomposition:
\begin{align*}
&W_{\kpa r} \;=\; \bigoplus_{i=0}^{k-2r} \; \bigoplus_{j=0}^{r-1} \; W_{\kpa rij} \;\;\;\;\;\;\;\;\;\; r \in \{1,\dots,k_\TB\} 
\end{align*}
Regarding $r=0$, we remark that it will occasionally be convenient to denote the trivial component as $W_{\kpa 0k0} := W_{\kpa 0}$, so that in this case we only consider~$(i,j) = (k,0)$.

\subsection{Asymmetric U-Statistics}
\label{ustats}

Our work builds and expands on the general framework of \emph{U-statistics}, first studied by Hoeffding~\cite{hoeffding1948class}. These are sums of the form
$$ U_n \;=\; \frac{1}{\tbinom{n}{k}} \sum_{i_1 < i_2 < \cdots < i_k} h\left(X_{i_1},X_{i_2},\cdots,X_{i_k}\right) $$ where the random variables $X_1, X_2, \dots$ are iid in some probability space~$\mathcal{X}$, and the \emph{kernel} function $h:\mathcal{X}^k \to \mathbb{R}$ is \emph{symmetric} with respect to permuting the $k$ inputs. We throughout assume $\Exp h^2 < \infty$.

U-statistics have a well-developed theory that provides information on their asymptotic properties \cite{serfling1980approximation, lee1990u, koroljuk1994theory, janson1997gaussian}. In the generic case $\sqrt{n} U_n$ tends to a Gaussian, but \emph{degenerate} cases are scaled as $n^{r/2}U_n$ and tend to other limit laws, with the rank $r$ defined as follows.

\begin{definition}
\label{rank}
The \emph{rank} of $U_n$ is the smallest number $r$ of inputs of~$h$ such that $\Exp[h\,|\,X_1\dots X_r]$ is not almost surely constant.
\end{definition}

Writing the variance of $U_n$ as a double sum, and grouping together terms by the number of common inputs, as follows, gives a leading nonvanishing term of order~$n^{-\rank f}$.

\begin{prop}
\label{varun}
$ \displaystyle \; \Var[U_n] \;=\; \sum_{r=1}^k \frac{\tbinom{k}{r} \tbinom{n-k}{k-r}}{\tbinom{n}{k}} \Var\left[\,\Exp[h\,|\,X_1\dots X_r]\,\right]$
\end{prop}

U-statistics naturally extend to the \emph{asymmetric} setting, where the summation in $U_n$ is taken over a kernel $h$ that is no longer assumed to be symmetric. In this case, which is less frequently discussed in the literature, the order in which the samples $X_1,\dots,X_n$ are given does matter. 

The word statistics we study in the random model $\mathcal{W}(n,\mathbf{p})$ can be formulated as asymmetric U-statistics. If $\mathcal{X}$ is the finite probability space $(\Sigma,\mathbf{p})$, and $h(x_1,\dots,x_k) = \mathds{1}[x_1\cdots x_k = u]$ for $u \in \Sigma^k$, then $U_n$ is distributed exactly as $\bar\#u$. So is $\bar\# f$ for every $f \in W_k$, by taking linear combinations. Theorems~\ref{main1} and~\ref{main2} analyze the second moment behavior of asymmetric U-statistics, for any finite sample space~$\mathcal{X}$. We expect certain parts of our analysis to extend to ``infinite alphabets'' as well.

In the other direction, the theory of U-statistics gives some general asymptotic information on the statistics that we study, beyond their scaling and second-moment diagonalization. The distribution of $\bar\#f$ weakly converges, and the limit has the form of a multiple stochastic integral, admitting an infinite expansion of degree~$r$ in Gaussian variables, although sometimes it can be simplified~\cite[\S{XI.2}]{janson1997gaussian}.

For some purposes, asymmetric U-statistics are reduced to the symmetric formulation. In short, take $(X_i,Y_i)$ iid in the product space $\mathcal{X}' = \mathcal{X} \times U(0,1)$, and define $h':(\mathcal{X}')^k \to \mathbb{R}$ by feeding $X_1,\dots,X_k$ to $h$ sorted by their $Y_i$ coordinate. The resulting symmetric $U'_n$ is distributed as the asymmetric~$U_n$. Still, the structure arising from the asymmetric formulation deserves special investigation. See~\cite{janson2018renewal} dedicated to other phenomena in asymmetric U-statistics, and~\cite[\S7]{janson1991asymptotic} for a generalization to \emph{unsymmetric statistics on ordered graphs}. These two works focus on the case of rank~1.

\subsection{Generalized U-Statistics}
\label{gustats}

The word statistics in Theorems~\ref{main3} and~\ref{main4}, in the multisample random model $\mathcal{W}(n_\TA,n_\TB,\dots)$, require the class of so-called \emph{generalized U-statistics}, which are based on more than one sample~\cite{hoeffding1948class}. 

For example, consider two independent samples $X_1, \dots, X_n$ and $Y_1, \dots, Y_m$ iid in respective probability spaces $\mathcal{X}$ and $\mathcal{Y}$, and a kernel $h:\mathcal{X}^k\times\mathcal{Y}^l \to \mathbb{R}$ symmetric to permuting the $X$-inputs or the $Y$-inputs. Then the following random variable is a \emph{two-sample} U-statistic:
$$ U_{nm} \;=\; \frac{1}{\tbinom{n}{k}\tbinom{m}{l}}\; \sum_{\substack{i_1 < \cdots < i_k \\ j_1 < \cdots < j_l}} h\left(X_{i_1},\cdots,X_{i_k}; Y_{j_1},\cdots,Y_{j_l}\right) $$
Much of the theory of U-statistics extends to the multisample case, though its treatment in the literature is often quite terse. 

Recall the word statistics $\tilde\#f$ in the random model $\mathcal{W}'(n_\TA,n_\TB,\dots)$ as in Theorem~\ref{main3}, where $f \in W_{\kpa} = \mathbb{R}\tbinom{\Sigma}{\kpa}$ and $\kpa = (k_\TA,k_\TB,\dots)$. These random variables can be represented as multisample U-statistics, with the samples $X_{\TX i} \sim U(0,1)$ for each $\TX \in \Sigma$ and $i \in \{1,\dots,n_\TX\}$, and kernel functions having $k_\TX$ inputs from each sample $\{X_{\TX 1},\dots,X_{\TX n_\TX}\}$. The real samples are mapped to a word, by means of sorting by their $[0,1]$-values and reading their $\Sigma$-labels. We make this description more formal in~\S\ref{proof3}, and discuss their notion of rank.

This representation provides additional asymptotic information on the distribution beyond the scaling and second-moment structure, as in the one-sample case. For both cases, we mention the following multivariate central limit theorem, useful when the rank $r=1$. 

\begin{alsotheorem} \cite[page~142]{lee1990u} 
\label{clt}
Let $\mathbf{U}_\mathbf{n} = (U^{(1)}_{\mathbf{n}},\dots,U^{(\ell)}_{\mathbf{n}})$ be $\ell$ generalized U-statistics, on common samples $\{X_{\TX 1},\dots,X_{\TX n_\TX}\}$ where $n_{\TX}/n \to p_{\TX} > 0$ for every~$\TX$ as $n = |\mathbf{n}| \to \infty$. Then the vector
$$ \mathbf{Z}_{\mathbf{n}} \;:=\; \sqrt{n} \left(\mathbf{U}_\mathbf{n} - \Exp \mathbf{U}_\mathbf{n} \right) \;\;\xrightarrow[\;\text{\rm in distribution }\;]{n \to \infty}\;\; \mathbf{Z}$$
where $\mathbf{Z}$ is a multivariate Gaussian distribution of mean $\mathbf{0}$ and covariance $\lim_{n \to \infty} \Cov \left[ \mathbf{Z}_{\mathbf{n}} \right]$.
\end{alsotheorem}

We apply the theorem in the one-sample random model $\mathcal{W}(n,(p_\TA,p_\TB,\dots))$. For $k \in \mathbb{N}$, we consider the $|\Sigma|^k$ U-statistics $\{\bar\#u : u \in \Sigma^k\}$, and naturally consider $\mathbf{Z}_{\mathbf{n}}$ as distributed in $W_k = \mathbb{R}\Sigma^k$. It is easy to see from Theorem~\ref{main1} that the support of the multinormal limiting distribution $\mathbf{Z}$ is the subspace $W_{k1}$ of dimension~$(|\Sigma|-1)k$.

In the multisample random model $\mathcal{W}'(n_\TA,n_\TB,\dots)$, we apply the theorem to the generalized U-statistics $\{\tilde\#u : u \in \tbinom{\mathsmaller{\Sigma}}{\kpa}\}$ where $\kpa = (k_\TA,k_\TB,\dots)$. Similarly, considering $\mathbf{Z}_{\mathbf{n}}$ as an element of $W_{\kpa}$, the multinormal limit distribution is supported on the subspace $W_{\kpa 1}$ of dimension $(|\Sigma|-1)(|\kpa|-1)$, since it has $|\Sigma|-1$ copies of the standard representation $S^{(|\kpa|-1,1)}$.

\section{One-Sample}
\label{proofs12}
This section proves Theorems~\ref{main1} and~\ref{main2}, and the next proves Theorems~\ref{main3} and~\ref{main4}.

\subsection{Proof of Theorem \ref{main1}}
\label{proof1}

Throughout, $w \in \Sigma^n$ will denote a random word in $\mathcal{W}(n,\mathbf{p})$. Let $f$ be a formal sum of words: $f = \sum_u f_u u \in W_{k} = \mathbb{R}\Sigma^k$ where $u \in \Sigma^k$ and $f_u \in \mathbb{R}$. Since every occurrence of $u$ contributes one $f_u$ to $\#f$, one can write
$$ \#f(w) \;=\; \sum_{u \in \Sigma^k}f_u\,\#u(w)  \;\;=\; \sum_{1 \leq t_1 < t_2 < \dots < t_k \leq n} f\left(w_{t_1}w_{t_2}\cdots w_{t_n}\right) $$
Here we denote $f(v) = f_v$ for $f \in \mathbb{R}\Sigma^k$ and $v \in \Sigma^k$. We use this notation to indicate that $f$ is also seen as an element of the dual of $\mathbb{R}\Sigma^k$, which is identified with $\mathbb{R}\Sigma^k$ by letting $u(v) = \delta_{uv}$ for $u,v \in \Sigma^k$. This determines $f(g)$ by linearity for any $f,g \in \mathbb{R}\Sigma^k$, which will come useful later. 

Recall that we have also defined an inner product on such statistics. Namely $\langle f,f'\rangle_{\mathbf{p}} = \Exp_u[f(u)f'(u)] = \sum_{u} p_u f_u f_u'$ where $p_u = \prod_i p_{u_i}$ for each word $u  = u_1 u_2 \cdots u_k \in \Sigma^k$. We have let $D = \{\Ti, \Tii, \dots, \TD\}$ be an orthonormal basis of $\mathbb{R}\Sigma$, with $\Ti(\TX)=1$ for every $\TX\in\Sigma$. It follows that all words $e = e_1e_2\cdots e_k \in D^k$ form an orthonormal basis of $\R\Sigma^k$. In this section, we denote this inner product by $\langle-,-\rangle$ without the subscript $\mathbf{p}$.

We expand according to this orthonormal basis the application of $f \in \mathbb{R}\Sigma^k$ on a word $u \in \Sigma^k$ as above:
$$ f(u) \;=\; \sum_{e \in D^k} \langle f, e \rangle e(u) \;=\; \sum_{e \in D^k} \langle f, e \rangle e_1(u_1)e_2(u_2)\cdots e_k(u_k) $$
Note that we have used the multiplicativity of the functional $e(u)$ under concatenation, which is straightforward from its definition.

The proof will proceed by plugging this decomposition of $f(u)$ into the above expansion of $\#f(w)$, as follows.
$$ \#f(w) \;=\; \sum_{e \in D^k} \langle f, e \rangle \sum_{t_1 < \dots < t_k} e_1\left(w_{t_1}\right) e_2\left(w_{t_2}\right)\cdots e_k\left(w_{t_k}\right) $$
This representation allows the asymptotic analysis of the second moments $\Exp[(\#f)^2]$ and $\Exp[(\#f)(\#f')]$. What makes it particularly useful is the following observation, that the functionals $e_j(w_{t_j})$ have simple averages over random letters.

\begin{obs}\label{obs:pigul}
Consider a random letter $\TX \in \Sigma$ distributed according to the probability vector $\mathbf{p}$. For $\TI,\TJ \in D$,
\begin{enumerate}
\itemsep0.25em
\item\label{it:pigul1}
$\Exp_\TX[\TI(\TX)] = \delta_{\TI\Ti}$
\item\label{it:pigul2} 
$\Exp_\TX[\TI(\TX)\TJ(\TX)] = \delta_{\TI\TJ}$
\end{enumerate}
\end{obs}

This observation is immediate from the definition of~$D$ as an orthonormal basis with respect to the inner product that is based on~$\mathbf{p}$. 

\medskip 

We start with a lemma that analyzes the second moments of the statistics~$\#e$, for the basis elements $e \in D^k$. They are given in terms of $m_{\ell}(e,e')$, the number of ways to merge two words $e$ and $e'$ into a longer one of length~$\ell$, as defined here.

\begin{definition}
\label{mcoefs}
The $\ell$th \emph{merging coefficient} of $e \in D^k$ and $e' \in D^{k'}$ is
\begin{equation*}
m_{\ell}\left(e,e'\right) \;=\; \left|\left\{\left(I,I'\right)\;\middle|\;\;
\begin{aligned}
& I = \{i_1, i_2, \dots, i_k\} \;\;\; i_1 < i_2 < \dots \;\;\; \\
& I' = \{i_1', i_2', \dots, i_{k'}'\} \;\;\; i_1' < i_2' < \dots \;\;\; \\
& I \cup I' = \{1,2,\dots,\ell\} \\
& i_j = i_{j'}' \;\implies\; e_j = e_{j'}' \\
& i_j \in I \setminus I' \;\implies\; e_j = \Ti \\
& i_{j'}' \in I' \setminus I \;\implies\; e_{j'}' = \Ti  
\end{aligned}
\right\}\right|
\end{equation*}
\end{definition}

\begin{ex*}
$m_3(\Ti\Tii,\Tii\Ti) = 1$ since the only merging pair $(I,I')$ is $((1,2),(2,3))$.
\end{ex*}

\begin{ex*}
$m_3(\Ti\Tii,\Ti\Tii) = 2$ with $(I,I') = ((1,3),(2,3))$ or $((2,3),(1,3))$.
\end{ex*}

\begin{ex*}
$m_3(\Ti\Tii,\Ti\Tiii) = 0$ since there is no suitable merging.
\end{ex*}

\begin{lemma}
\label{lem:multiplication_of_u_statistics_r}
\label{merge}
If $e \in D^k$ and $e' \in D^{k'}$ then
$$ \Exp_w\left[\#e(w)\,\#e'(w)\right] \;=\; \sum_{\ell=\max\left(k,k'\right)}^{k+k'} m_{\ell}\left(e,e'\right)\,\binom{n}{\ell} $$
\end{lemma}

\begin{proof}
For $w\in\Sigma^n$, it holds that
\begin{multline*}
\label{e_eq:f_flat_times_flat}
\# e\left(w\right)\cdot\# e'\left(w\right) \;=\; \sum_{\substack{1\le t_1 <\dots<t_{k} \le n \\ 1\le t'_1 <\dots<t'_{k'} \le n}} {e}(w_{t_1},\dots,w_{t_{k}})\; e'(w_{t'_1},\dots,w_{t'_{k'}}) \\
\;=\; \sum_{\substack{1\le t_1 <\dots<t_{k} \le n \\ 1\le t'_1 <\dots<t'_{k'} \le n}} e_1(w_{t_1}) e_2(w_{t_2}) \cdots e_k(w_{t_k}) \; e'_1(w_{t'_1}) e'_2(w_{t'_2}) \cdots e'_{k'}(w_{t'_{k'}})
\end{multline*}
For each term in the sum, 
we denote its set of positions in $w$ by
$$ L \;:=\; \left\{l_1,\dots,l_{\ell}\right\} \;=\; \{t_1,\dots,t_{k}\}\cup\{t'_1,\dots,t'_{k'}\} $$
such that $1\le l_1<\dots<l_{\ell}\le n$, and we record which of these $\ell$ positions correspond to $e$ and which ones to $e'$ by 
\begin{align*}
&I \;:=\; \{i_1,\dots,i_k\} \;\;\text{such that}\;\; t_j = l_{i_j}\;\;\text{for}\;\; j \in \{1,\dots,k\} \\
&I' \;:=\; \{i'_1,\dots,i'_{k'}\} \;\;\text{such that}\;\; t'_j = l_{i'_j}\;\;\text{for}\;\; j \in \{1,\dots,k'\} \end{align*}
Note that
$i_1 < i_2 < \dots < i_k$ and $i'_1 < i'_2 < \dots < i'_{k'}$ and $I \cup I' \;=\; \{1,\dots,\ell\}$. The positions $\{t_j\}$ and $\{t'_j\}$ can be uniquely reconstructed from any such $\{i_j\}$ and $\{i'_j\}$ given~$L$. In other words, there is a bijection between the terms in the sum and such triplets~$(L,I,I')$. 

Using this notation, we restate the above sum,
$$ \# e\left(w\right)\cdot\# e'\left(w\right) \;=\;
\sum_{L,I,I'}\; {e_{1}} \left(w^{\phantom{8}}_{l_{i_1}}\right) \cdots {e_{k}} \left(w^{\phantom{8}}_{l_{i_k}}\right)\, {e'_{1}} \left(w^{\phantom{8}}_{l_{i'_{1}}}\right) \cdots {e'_{k'}} \left({w_{l_{i'_{k'}}}}\right) $$
and take expectation over $w$ with respect to the product measure $\mathcal{W}(n,\mathbf{p})$,
\begin{multline*}
\Exp_{w}\left[\# e\left(w\right)\cdot{\# e'\left(w\right)}\right] \;=\;
\sum_{L,I,I'}\; \Exp_w\left[\; \prod_{i_j\in I} {e_{j}} \left(w_{l^{\phantom{8}}_{i_j}}\right)\; \prod_{i'_{j'}\in I'} {e'_{j'}} \left(w_{l_{i'_{j'}}}\right)  \right] \\[0.5em]
\;=\; \sum_{L,I,I'} \; \prod_{i_j\in I\setminus I'}\Exp_\TX\left[ e_j^{\phantom{8}} (\TX)\right] \; \prod_{i'_{j'}\in I'\setminus I}\Exp_\TX\left[ e'_{j'} (\TX)\right] \prod_{i_j = i'_{j'}\in I\cap I'}\Exp_\TX\left[ e_j (\TX) e'_{j'} (\TX)\right]\end{multline*}
Here, we have used the independence of different letters in $w$, to separate the expectation of the product into expectations over single letters. By observation~\ref{obs:pigul}, the term corresponding to $(L,I,I')$ equals 1 if all the following hold and 0 otherwise.
\begin{enumerate}
\itemsep0.5em
\item 
$e_j = \Ti$ for every $i_j\in I\setminus I'$,
\item 
$e'_{j'} = \Ti$ for every $i'_{j'}\in I'\setminus I$,
\item 
$e_j = e'_{j'}$ for every $i_j = i'_{j'}\in I \cap I'$.
\end{enumerate}
By the definition of the coefficients ${m}_{\ell}(e,e')$, summing all such terms with a given~$\ell$ gives a contribution of 
$$\binom{n}{\ell}\;{m}_{\ell}(e,e'),$$ where the binomial comes from picking the positions $l_1,\ldots,l_{\ell}.$ Then we sum over $\ell\in\{\max(k,k'),\ldots,k+k'\}$,
and the lemma is proven.
\end{proof}

The special role of the letter $\Ti$ in the above expansion leads us to break up the basis words $e \in D^k$ as in Definition~\ref{phikr}. Recall that $e\mapsto \left(\pi(e), \rho(e)\right)$, such that in
$\pi(e) \in \{\Ti,\Tii\}^k$ every $\TI \neq \Ti$ is replaced by $\Tii$, and in
$\rho(e) \in D^{k - \#\Ti(e)}$ all~$\Ti$s are removed. Clearly, the length of $\rho(e)$ equals $\#\Tii(\pi(e))$, and the original word $e$ is reconstructable from such $\pi(e)$ and~$\rho(e)$. The following lemma uses this representation to determine if the merging coefficient vanishes. 

\begin{lemma}
\label{pepe}
Let $e \in D^k$ and $e' \in D^{k'}$. The merging coefficient $m_\ell(e,e') = 0$ unless
$$ \rho(e) \;=\; \rho(e') \;\in\; \left(D \setminus \{\Ti\} \right)^r $$
where 
$$ r \;=\; \#\Tii(\pi(e)) \;=\; \#\Tii(\pi(e')) \;\in\; \left\{0,\dots,\min(k,k')\right\}$$
and
$$ \ell \;\in\; \left\{\max(k,k') , \dots, k+k'-r \right\} $$
\end{lemma}

\begin{proof}
By the definition of the $\ell$th merging coefficient $m_\ell(e,e')$ of $e \in D^k$ and $e' \in D^{k'},$ it vanishes unless all non-$\Ti$ letters in $e$ and $e'$ appear with the same multiplicity and order, that is, $\rho(e) = \rho(e')$.

The length of $\rho(e)$ is the number of non-$\Ti$ letters in $e$, which is $r=\#\Tii(\pi)$, since all these letters are changed to $\Tii$ under $\pi$. The same reasoning applies to $e'$, so that $r = \#\Tii(\pi(e'))$. Clearly $r \leq k$ and $r \leq k'$.

The length of the merging, $\ell$, cannot be less than the longest of $e,e',$ on the one hand.
On the other hand, since each non-$\Ti$ letter of $e$ is mapped to the same place as the corresponding non-$\Ti$ letter of $e',$ the length cannot be greater than $k+k'-r.$
\end{proof}

The above lemma shows that $\rho$ induces a block structure on the covariance matrix of all $\Exp_w[\#e(w)\,\#e'(w)]$, since the entry of $e \in D^k$ and $e' \in D^{k'}$ vanishes if $\rho(e) \neq \rho(e')$. By Definition~\ref{pyro}, if $\rho(e) = \rho(e')$ then the merging coefficient $m_\ell(e,e') = m_\ell(\pi(e),\pi(e'))$, since the non-$\Ti$ letters are known to match and may be replaced by~$\Tii$s. Hence, it is enough to study these covariances only for $\rho^{-1}(\Tii\Tii\cdots\Tii)$, i.e., words over $\{\Ti,\Tii\}$ with $\#\Tii=r$. Then for each of the $(d-1)^r$ blocks that correspond to $\{\Tii,\dots,\TD\}^r$, Lemma~\ref{merge} gives exactly the same block of covariances.

Restricting to a specific length $k=k'$ and $r \in \{0,\dots,k\}$, we have $\tbinom{k}{r}$ words, corresponding to both the rows and columns of every block in the covariance matrix. Its leading terms are given the following notation.

\begin{definition}
\label{mkr}
Let $k \geq r \geq 0$. The \emph{merging matrix} $M_{kr}$ is an $\tbinom{k}{r}$-by-$\tbinom{k}{r}$ matrix of positive integers, whose rows and columns are indexed by all words $e, e' \in \{\Ti,\Tii\}^k$ with $\#\Tii(e) = \#\Tii(e') = r$. Its entries are given by
$$ \left[M_{kr}\right]_{e,e'} \;=\; m_{2k-r}(e,e')$$
\end{definition}

\begin{ex*}
$M_{21} = \left[\begin{smallmatrix} 2&1 \\ 1&2 \end{smallmatrix}\right]\;$, rows and columns indexed by $(\Ti\Tii,\Tii\Ti)$
\end{ex*}

\begin{ex*}
$M_{30} = \left[\begin{matrix} 20 \end{matrix}\right]\;$, $M_{31} = \left[\begin{smallmatrix} 6&3&1 \\ 3&4&3 \\ 1&3&6 \end{smallmatrix}\right]\;$, $M_{32} = \left[\begin{smallmatrix} 2&1&1 \\ 1&2&1 \\ 1&1&2 \end{smallmatrix}\right]\;$, $M_{33} = \left[\begin{matrix} 1 \end{matrix}\right]\;$
\end{ex*}

The merging matrix $M_{kr}$ is positive definite. This observation may be deduced from Proposition~\ref{spectral} in Section~\ref{proof2} below, where we study the properties of~$M_{kr}$ in more detail. Assuming this fact, we complete the current proof as follows.

\begin{proof}[Proof of Theorem~\ref{main1}]
Let $f \in W_{kr}$ and $f' \in W_{k'r'}$ as in the theorem. We first expand these statistics in the orthonormal basis $D^k$, and then restrict the summation to the subsets $D_{kr} := \left\{e \in D^k: \#\Ti(e) = k-r \right\}$, since $\langle f,e \rangle=0$ for $e \not\in D_{kr}$ by the definition of the spaces~$W_{kr}$.
\begin{align*}
\Exp_w &\left[\#f(w)\, \#f'(w)\right] \;=\; \Exp_w \left[ \sum_{e \in D^k} \langle f, e \rangle \#e(w) \; \sum_{e' \in D^{k'}} \langle f', e' \rangle \#e'(w) \right] \\[0.5em]
\;&=\; \sum_{e \in D_{kr}} \;\langle f, e \rangle \; \sum_{e' \in D_{k'r'}}  \langle f', e' \rangle \;\, \Exp_w \left[\#e(w)\, \#e'(w)\right]
\\[0em]
\;&=\; \sum_{e \in D_{kr}} \; \sum_{e' \in D_{k'r'}} \langle f, e \rangle  \, \left(\sum_{\ell=\max k,k'}^{k+k'} m_{\ell}\left(e,e'\right)\,\binom{n}{\ell}\right) \,\langle f', e' \rangle
\end{align*}
by Lemma~\ref{merge}. If $r \neq r'$ then the lengths of $\rho(e)$ and $\rho(e')$ differ, and all terms vanish by Lemma~\ref{pepe}. This proves the second part of the theorem in a slightly more general form, without assuming $k=k'$.

Let $r=r'$. Since terms with $\rho(e)\neq\rho(e')$ vanish by Lemma~\ref{pepe}, we divide the above summation into cases according to $g = \rho(e)=\rho(e') \in \{\Tii,\dots,\TD\}^r$.
\begin{align*}
\;&=\; \sum_{g \in (D\setminus\Ti)^r} \; \sum_{\substack{e \in D_{kr\phantom{'}} \\ \rho(e)=g}} \;\; \sum_{\substack{e' \in D_{k'r} \\ \rho(e')=g}} \; \sum_{\ell=\max k,k'}^{k+k'}\, \binom{n}{\ell} \;\langle f, e \rangle  \;  m_{\ell}\left(e,e'\right)\,\langle f', e' \rangle \\
\;&=\; \binom{n}{k+k'-r} \sum_{g}\; \sum_{e,e'}\;
\langle f, e \rangle \,
m_{k+k'-r}\left(e,e'\right) \langle f', e' \rangle \;+\; O\left(n^{k+k'-r-1}\right)
\end{align*}
since $\tbinom{n}{\ell} \sim \tfrac{n^\ell}{\ell!}$, and $m_\ell(e,e')$ vanish as well for $\ell > k+k'-r$ by Lemma~\ref{pepe}.
This yields a formula for the leading coefficient,
$$ C_{f,f'} \;:=\; \lim_{n \to \infty} n^r \;\Exp_w \left[\frac{\#f(w)}{\tbinom{n}{k}}\cdot\frac{\#f'(w)}{\tbinom{n}{k'}}\right] 
\;=\; \tfrac{k!\,k'!}{(k+k'-r)!} \;\sum_{g \in (D\setminus\Ti)^r} c_g(f,f') $$
where
$$ c_g(f,f') \;:=\; \sum_{\substack{e \in D_{kr\phantom{'}} \\ \rho(e)=g}} \;\; \sum_{\substack{e' \in D_{k'r} \\ \rho(e')=g}} \;
\langle f, e \rangle \,
m_{k+k'-r}\left(e,e'\right) \langle f', e' \rangle $$

We now consider the case where the statistics $f$ and~$f'$ are given by words of equal length~$k=k'$. Then $c_g(f,f')$ is computed by the square matrix $M_{kr}$ acting as a bilinear form on the vectors $(\langle f, e \rangle)_{e \in \rho^{-1}(g)}$ and $(\langle f', e' \rangle)_{e' \in \rho^{-1}(g)}$.
$$ c_{g}(f,f') \;=\; \sum_{e,e' \in \rho^{-1}(g)} \langle f, e \rangle \; \left[M_{kr}\right]_{\pi(e),\pi(e')}\; \langle f', e' \rangle $$ 
Here we have used $m_\ell(e,e') = m_\ell(\pi(e),\pi(e'))$ since $\rho(e)=\rho(e')$. Note that $\pi$ induces a bijection between the $\tbinom{k}{r}$ words $e \in \rho^{-1}(g)$ and the words in $\{\Ti,\Tii\}^k$ that index the rows and columns of~$M_{kr}$. 

For the first part of Theorem~\ref{main1}, we further specialize to $f=f'\in W_{kr}$, so that $C_{f,f} = C_{f,\mathbf{p}}$ in the statement of the theorem. Since $M_{kr}$ is positive definite, $c_{\rho(e)}(f,f) > 0$ if $e \in D_{kr}$ is such that $\langle f,e \rangle \neq 0$, and in general $c_g(f,f) \geq 0$ for all~$g$. By assumption $f \in W_{kr} = \mathrm{span}\,D_{kr}$ is nonzero, so $\langle f,e \rangle \neq 0$ for at least one word $e$. It follows that $C_{f,\mathbf{p}} > 0$ as required. 
\end{proof}

\subsection{Proof of Theorem~\ref{main2}} ~
\label{proof2}

This main focus here will be on the spectral decomposition of the matrix~$M_{kr}$ from Definition~\ref{mkr} above. Recall that the rows and columns of this matrix are indexed by all the $\tbinom{k}{r}$ words $e \in \{\Ti,\Tii\}^k$ that have $\#\Tii(e)=r$ and $\#\Ti(e)=k-r$. These words span a linear space~$V_{kr}$ endowed with the unique inner product~$\langle-,-\rangle$ that makes them an orthonormal basis. 

If $|\Sigma|=2$ then $V_{kr}$ coincides with the word statistics~$W_{kr}$. In general, $W_{kr}$~naturally factors into $V_{kr} \otimes (\Ti^{\perp})^{\otimes r}$ using Definition~\ref{pyro}. The main result of this section is the following decomposition of~$V_{kr}$, which leads to Theorem~\ref{main2} on the limiting second moments of all statistics in $W_{kr}$. 
\begin{prop}
\label{spectral}
Let $k \geq r \geq 1$. The matrix $M_{kr}$ has $k-r+1$ distinct eigenvalues
$$ \mu_{krm}\;=\;\binom{2k-r}{k+m} \;\;\;\;\;\;\;\; m \in \{0,\dots,k-r\} $$
corresponding to eigenspaces
$$ V_{krm} \;=\; \left(\ker \del^{k-r-m+1}\right) \cap \left(\ker \del^{k-r-m}\right)^{\perp} \;\subset\; V_{kr} $$ 
of dimensions
$$ \dim V_{krm} \;=\; \binom{r+m-1}{m} $$
\end{prop}

The proof of Proposition~\ref{spectral} will be given after several lemmas. We will investigate different aspects of the matrix $M_{kr}$, as well as some general properties of the algebraic word operators $\del$ and $\sh$, as defined in \S\ref{operators}. The first lemma presents useful closed form expressions for~$M_{kr}$. It will be stated after fixing some notation.

\begin{definition}
\label{d}
Given a word $e \in \{\Ti,\Tii\}^k$, denote by $\mathbf{d}(e) \in \mathbb{Z}$ the lengths of runs of~$\Ti$ between its occurrences of~$\Tii$, including at its two ends. 
\end{definition}

\begin{ex*}
$\mathbf{d}(\Ti\Tii\Ti\Ti\Ti\Ti\Tii\Tii\Ti) = (1,4,0,1) \in \Delta_{93}$ 
\end{ex*}

Note that if $\#\Tii(e)=r$ then $\mathbf{d}(e) = (d_0(e),\dots,d_r(e)) \in \mathbb{Z}^{r+1}$ and $\sum_i d_i(e) = (k-r)$. This yields a one-to-one correspondence between such words $e=\Ti^{d_0}\Tii\Ti^{d_1}\Tii\cdots\Tii\Ti^{d_r}$ and the points $(d_0,\dots,d_r)$ in the discrete simplex $\Delta_{kr}$ mentioned in Definition~\ref{deltakr}.

Recall from \S\ref{defs} the deletion operator $\del_w$, which combines all the ways to delete a subword $w$, and $\sh_w$ which combines all ways to insert~$w$. We define the following word operator.
\begin{definition}
\label{dm}
$ \mathcal{D}_{m} \;:=\; I + \del_{\Ti} + \del_{\Ti\Ti} + \del_{\Ti\Ti\Ti} + \dots + \del_{(\Ti^m)} $
\end{definition}

\begin{rmk*}
Here and in the rest of the proof, $I$ denotes the identity matrix or the identity operator.
\end{rmk*}

\begin{samepage}
\begin{lemma}
\label{mkrmkrmkr}
\label{Akr_identities}
Let $k \geq r \geq 0$. The matrix $M_{kr}$ admits the following equivalent descriptions.
\begin{enumerate}
\itemsep0.5em
\item 
For two words $e,e'$ in the standard basis of $V_{kr}$
$$ [M_{kr}]_{e,e'} \;=\; \prod_{i=0}^{r} \binom{d_i(e)+d_i(e')}{d_i(e)} $$
\item
As a bilinear map $M_{kr} : V_{kr} \times V_{kr} \to \mathbb{R}$
$$ M_{kr}(f,f') \;=\; \left\langle \mathcal{D}_{k-r}f, \mathcal{D}_{k-r}f' \right\rangle$$
\item \label{Akr_identity}
As a word operator $M_{kr} : V_{kr} \to V_{kr}$
$$ M_{kr}(f) \;=\; \sum_{j=0}^{k-r} \frac{1}{(j!)^2} \sh_\Ti^j \del_\Ti^j f $$
\end{enumerate}
\end{lemma}
\end{samepage}

\begin{rmk*}
It follows from the second part of the lemma that $M_{kr}$ is positive definite. This fact has already been exploited to show $C_{f,\mathbf{p}}>0$ in the proof of Theorem~\ref{main1}. 
\end{rmk*}

\begin{rmk*}
The first representation of $M_{kr}$ in the lemma arises in work by Janson and Nowicki~\cite{janson1991asymptotic} in a similar setting. They write the following matrix:
$$ [M_{kr}]_{e,e'} \;=\; (2k-r)! \idotsint\limits_{0<t_1<\dots <t_r<1} \; \prod_{i=0}^r \frac{ \left(t_{i+1}-t_i\right)^{d_i(e)+d_i(e')}}{d_i(e)!\,d_i(e')!} \,dt_1\cdots dt_r $$
This also implies that $M_{kr}$ is positive definite. Then they deduce that a certain variance term is nonzero, but without any quantitative information.
\end{rmk*}

\begin{proof}[Proof of Lemma~\ref{mkrmkrmkr}]
We start from Definitions~\ref{mcoefs} and~\ref{mkr}:  $$ [M_{kr}]_{e,e'} = m_{2k-r}(e,e') $$

(1) This is straightforward from the definition of $m_{2k-r}(e,e')$, and the observation that the positions of the $\Tii$s in the merged word coincide while the $\Ti$s between them are distributed among the two words. 

(2) The Vandermonde identity for binomial coefficients states that for any three nonnegative integers~$a,b,c$
$$ \binom{a+b}{c} \;=\; \sum_{d=0}^c\binom{a}{d}\binom{b}{c-d} $$
We apply it to the first statement and obtain
\begin{align*}
[M_{kr}]_{e,e'} \;&=\; \prod_{i=0}^{r} \;\sum_{j_i = 0}^{d_i(e)} \binom{d_i(e)}{j_i} \binom{d_i(e')}{d_i(e)-j_i} \\
\;&=\; \sum_{j=0}^{k-r} \; \sum_{\substack{\left(j_0, \dots, j_r \right) \\ j_0 + \ldots + j_r = j}} \,\prod_{i=0}^{r} \, \binom{d_i(e)}{j_i} \binom{d_i(e')}{j_i+d_i(e')-d_i(e)}
\end{align*}

Observe that the $j$th term counts the number of ways to delete $j$ occurrences of $\Ti$ from each of $e$ and $e'$, resulting in the same word. This is done by choosing some $j_i$ of the $d_i(e)$ ones in the $i$th run of the word~$e$, and $j_i+d_i(e')-d_i(e)$ of the $d_i(e')$ ones in the $i$th run of~$e'$. The total count is obtained by summing over any $j_i$ with $\sum_i j_i=j$. 

Therefore, by the definition of the deletion operator $\del_{\Ti\cdots\Ti}$ and the orthogonality of words over $\{\Ti,\Tii\}$, these numbers can be written as
$$ [M_{kr}]_{e,e'} \;=\; \sum_{j=0}^{k-r} \left\langle \del_{(\Ti^j)}\,e ,\, \del_{(\Ti^j)}\,e' \right\rangle \;=\; \left\langle \mathcal{D}_{k-r} e, \mathcal{D}_{k-r} e' \right\rangle $$
Here the last equality is by the orthogonality of words of different length. The result for general word combinations $f,f' \in V_{kr}$ follows.

(3) Note that $\del_{(\Ti^j)} = (\del_{\Ti})^j / j!$, since removing $j$ ones may be done in $j!$ different orders. Then the statement of the lemma follows from the previous one by
the duality of $\del_\Ti$ and~$\sh_\Ti$, see~\S\ref{operators}.
\end{proof}

Before we further study the matrix $M_{kr}$, we make a series of general useful observations on the properties of word operations.

\medskip 

The operators $\del_\Ti$ and~$\sh_\Ti$ are defined on any formal sum of words over any alphabet, and in particular on~$\bigoplus_{k,r}V_{kr}$. Since we usually focus on their restriction to a single space~$V_{kr}$, we denote:
\begin{enumerate}
\item 
$\sh_\Ti^{(k,r)} : V_{kr} \to V_{(k+1)r}$
\item 
$\del_\Ti^{(k,r)} : V_{kr} \to V_{(k-1)r}$
\end{enumerate}
For convenience, we let $V_{kr}:=\{0\}$ if $k<r$. We often short-hand $\del$ and $\sh$ when the domain is otherwise clear from the context. For example, $\sh \circ \del$ on~$V_{kr}$ means $\sh^{(k-1,r)} \circ \del^{(k,r)}$. Such compositions of $\del$ and $\sh$ \emph{from above} and \emph{from below} are further abbreviated to $A$ and~$B$, as follows.

\begin{definition}
\label{ab}
Let \vspace{0.25em}
\begin{enumerate}
\itemsep0.25em
\item 
$A^{(k,r)} \;:=\; \del^{(k+1,r)} \circ \sh^{(k,r)}$
\item 
$B^{(k,r)} \;:=\; \sh^{(k-1,r)}\circ\del^{(k,r)}$
\end{enumerate}
\end{definition}
 
\begin{lemma}\label{lem:commutation_relations}
Let $k \geq r \geq 0$. The following commutation relations of maps hold, when applied on~$V_{kr}$. \vspace{0.25em}
\begin{enumerate}
\itemsep0.5em
\item 
$A-B \;=\; \del \circ \sh - \sh \circ \del \;=\; (2k-r+1)\,I$
\item 
$\del \circ B - B \circ \del \;=\; (2k-r-1)\,\del$
\end{enumerate}
\end{lemma}

\begin{proof}
The first relation is verified by counting, or obtained as the special case $a=b=0$ of Lemma~36 in~\cite{dieker2018spectral}. The second is obtained from the first and the definition of~$B$.
\end{proof}

\begin{lemma}
\label{lem:surj,inj,kernel}
\label{lem123}
Let $k \geq r \geq 0$. \vspace{0.25em}
\begin{enumerate}
\itemsep0.5em
\item 
The map $\del: V_{kr} \to V_{(k-1)r}$ is surjective.
\item 
The map $\sh:V_{kr} \to V_{(k+1)r}$ is injective.
\item In $V_{kr}$,
$\;\ker B \;=\; \ker \del$
\end{enumerate}
\end{lemma}

\begin{proof}
The $\tbinom{k}{r}$ words that span $V_{kr}$ may be ordered lexicographically. For example, in~$V_{42}$,
$$ \Ti\Ti\Tii\Tii < \Ti\Tii\Ti\Tii < \Ti\Tii\Tii\Ti < \Tii\Ti\Ti\Tii < \Tii\Ti\Tii\Ti < \Tii\Tii\Ti\Ti $$

Consider a single word $e \in V_{(k-1)r}$. As above, $d_0(e)$ counts the leading $\Ti$s in~$e$. We apply the map $\del$ to the concatenated word~$\Ti e \in V_{kr}$. The terms in $\del (\Ti e)$ start either with $\Ti^{d_0(e)}\Tii$ or with~$\Ti^{d_0(e)+1}\Tii$, depending on which $\Ti$ is deleted. Hence they have the form
$$ \del (\Ti e) \;=\; \left(d_0(e)+1\right)e + R(e) $$
where the remainder term
$$ R(e) \;\in\; \mathrm{span}\left\{e' \in V_{(k-1)r} \;|\; e' < e\right\}$$
The restricted map
$$ \del : \textrm{span}\left\{\Ti e \;|\; e \in V_{(k-1)r}\right\} \;\to\; V_{(k-1)r} $$
is thus represented by a triangular matrix, with nonzero terms on the diagonal, and hence surjective. Therefore, so is the unrestricted $\del$ from all~$V_{kr}$. 

The second statement follows from the first one by duality, see~\S\ref{operators}. The third statement follows from the first two, using the definition of $B$ as a composition.
\end{proof}

We now use the above properties to investigate the spectral structure of~$B$. The commutation relations of $\del$ and $\sh$ allow us to regard them as \emph{annihilation and creation operators}.

\begin{lemma}
\label{lem:creation_annihilation}
Let $k \geq r \geq 0$, and $B = (\sh \circ \del) : V_{kr} \to V_{kr}$. The eigenvalues of $B$ are, in decreasing order,
$$ \beta_{krm} \;=\; (k-r-m)(k+m) \;\;\;\;\;\;\;\; m \in \{0,\dots,k-r\} $$
and the corresponding eigenspaces $V_{krm}$ are
$$ V_{krm} \;=\; \left(\ker \del^{k-r-m+1}\right) \cap \left(\ker \del^{k-r-m}\right)^{\perp} $$
and their dimensions are
$$ \dim V_{krm} \;=\; \binom{m+r-1}{m} $$
\end{lemma}

\begin{proof}
By Lemma~\ref{lem:surj,inj,kernel}, $\ker B = \ker \del$, while $\del : V_{kr} \to V_{(k-1)r}$ is surjective. Hence, the dimension of the kernel is
$$ \dim V_{kr} - \dim V_{(k-1)r} \;=\; \binom{k}{r} - \binom{k-1}{r} \;=\; \binom{k-1}{r-1}$$
This proves all the assertions of the lemma in the case $m=k-r$, where the eigenvalue is zero. This is the smallest eigenvalue since $B$ is nonnegative definite, as the composition of $\del$ and its dual.

If $k=r$, then there is nothing left to do. Otherwise, consider an eigenvector $Bv = \beta v$, of another eigenvalue $\beta > 0$. By the second commutation relation in Lemma~\ref{lem:commutation_relations},
$$ \beta \del v \;=\; \del B^{(k,r)} v \;=\; B^{(k-1,r)} \del v + (2k-r-1) \del v $$
That is, the $\beta$ eigenspace of $B^{(k,r)}$ maps via $\del$ to an eigenspace of $B^{(k-1,r)}$ with the shifted eigenvalue $\beta-(2k-r-1)$. Since $\del$ is an isomorphism from $(\ker \del)^\perp$ to $V_{(k-1)r}$, all the mappings between these eigenspaces must also be isomorphisms.

Since $k>r$, we use induction on $k$ to compute the nonzero eigenvalues of~$B = B^{(k,r)}$. For every $m \in \{0,1,\dots,k-r-1\}$,
\begin{align*}
\beta_{krm} \;&=\; \beta_{(k-1)rm} + (2k-r-1) \\ 
\;&=\; (k-1-r-m)(k-1+m) + (2k-r-1) \\
\;&=\; (k-r-m)(k+m)
\end{align*}
as required. Note that $\beta_{krm}$ are decreasing in $m$ by induction as well. 

For every such $m < k-r$, we use by induction the formula in the lemma for the eigenspaces $V_{(k-1)rm}$, and write  
$$ V_{(k-1)rm} \oplus \dots \oplus V_{(k-1)r(k-r-1)} \;=\; \ker \del^{k-r-m} \;\subseteq\; V_{(k-1)r} $$
Since $\del$ maps every nonkernel eigenspace $V_{krm}$ isomorphically to $V_{(k-1)rm}$, while $V_{kr(k-r)}$ is mapped to zero, it follows that
$$ V_{krm} \oplus \dots \oplus V_{kr(k-r-1)} \oplus V_{kr(k-r)} \;=\; \ker \del^{k-r-m+1} \;\subseteq\; V_{kr} $$
By the orthogonality of the eigenspaces, this yields the formula for $V_{krm}$ as in the lemma. By the same isomorphism, $\dim V_{krm} = \dim V_{(k-1)rm} = \tbinom{m+r-1}{m}$ which again follows by induction on~$k$.
\end{proof}

\begin{lemma}
\label{lem:R_m_in terms of B}
For $k \geq r \geq 1$, the following operators on $V_{kr}$ are equal. 
$$ \sh^j\,\del^j \;=\; \prod_{m=k-r-j+1}^{k-r}\left[\,B-(k-r-m)(k+m)\,I\,\right] $$
\end{lemma}

\begin{proof}
By repeated use of the commutation relation from Lemma~\ref{lem:commutation_relations}, which is $\sh\del = \del\sh - (2k'-r+1)I$ on the various $V_{k'r}$, one can transform $\sh^j\del^j$ into a monic polynomial of degree~$j$ in~$(\sh\del) = B:V_{kr} \to V_{kr}$. 

The eigenspaces of $\sh^j\del^j$ are therefore direct sums of eigenspaces of~$B$. By Lemma \ref{lem:creation_annihilation}, $\sh^j\del^j$ vanishes on~$V_{krm}$ for every
$$ m \in \{k-r,k-r-1,\dots,k-r-j+1\}$$ 
This means that the corresponding $\beta_{krm}$ must be the $j$ roots of the polynomial in~$B$ that expresses $\sh^j\del^j$. The right-hand side in the lemma is the only monic polynomial in $B$ with those $j$ roots.
\end{proof}

We finally return to the matrix $M_{kr}$. The combination of the previous lemma with Lemma~\ref{mkrmkrmkr} represents it as follows.

\begin{cor}\label{cor:A_B_same_eig_spaces}
Let $k \geq r \geq 1$. On $V_{kr}$,
$$ M_{kr} \;=\; \sum_{j=0}^{k-r} \frac{1}{(j!)^2} \prod_{i=k-r-j+1}^{k-r}\left[\,B-(k-r-i)(k+i)\,I\,\right]  $$
\end{cor}

This formula lets us deduce the spectral decomposition of $M_{kr}$ from that of~$B$, thereby proving our main proposition.

\begin{proof}[Proof of Proposition \ref{spectral}]
Since Corollary~\ref{cor:A_B_same_eig_spaces} expresses $M_{kr}$ as a polynomial in~$B$, its eigenspaces are direct sums of $V_{kr0}, \dots, V_{kr(k-r)}$. 

The eigenvalues of $M_{kr}$ corresponding to these spaces are computed by substituting those of $B$ in the polynomial. Namely, for $v \in V_{krm}$, we write $B v = \beta_{krm} v = (k-r-m)(k+m)v$ and obtain $M_{kr} v = \mu_{krm} v$, where
\begin{align*}
\mu_{krm}\;&=\; \sum_{j=0}^{k-r} \frac{1}{(j!)^2} \prod_{i=k-r-j+1}^{k-r}\left[(k-r-m)(k+m)-(k-r-i)(k+i)\right] \\
\;&=\; \sum_{j=0}^{k-r-m} \frac{1}{(j!)^2} \prod_{i=k-r-j+1}^{k-r}(i+m+r)(i-m) \\
\;&=\; \sum_{j=0}^{k-r-m} \frac{1}{(j!)^2}\, \frac{(k+m)!}{(k+m-j)!}\,\frac{(k-r-m)!}{(k-r-m-j)!} \\
\;&=\; \sum_{j=0}^{k-r-m} \binom{k+m}{k+m-j} \binom{k-r-m}{j} \;=\; \binom{2k-r}{k+m} \end{align*}
The sum is truncated at $k-r-m$ since the subsequent terms have zero factors where~$i=m$. The last equality is by the Vandermonde identity, see the proof of Lemma~\ref{mkrmkrmkr}.

The eigenvalues $\mu_{kr0},\dots,\mu_{kr(k-r)}$ are strictly decreasing, due to properties of binomial coefficients. Therefore, all the spaces $V_{krm}$ are distinct eigenspaces of~$M_{kr}$, exactly as for~$B$. Their descriptions and dimensions are therefore identical to those in Lemma~\ref{lem:creation_annihilation}, as stated in the proposition.
\end{proof}

As noted in the proof of Theorem~\ref{main1}, the covariance matrix of all statistics in $W_{kr}$ decomposes into $(d-1)^r$ blocks of $M_{kr}$. The following proof of Theorem~\ref{main2} uses the above decomposition of $M_{kr}$ to continue the derivation and completely diagonalize this matrix.

\begin{proof}[Proof of Theorem~\ref{main2}]
Let $f \in W_{krm}$ and $f' \in W_{krm'}$.
In the proof of Theorem~\ref{main1}, we have already shown that the limit in the statement of Theorem~\ref{main2} exists, and equals
$$ C_{f,f'} \;=\; \frac{(k!)^2}{(2k-r)!} \;\sum_{g \in (D\setminus\Ti)^r} \sum_{\substack{e,e' \in D_{kr} \\ \rho(e)=\rho(e')=g}} \langle f, e \rangle \; \left[M_{kr}\right]_{\pi(e),\pi(e')}\; \langle f', e' \rangle $$ 
where $\rho$ remove all~$\Ti$s, and $\pi$ replaces all non-$\Ti$s by~$\Tii$, as in Definition~\ref{pyro}, and $D_{kr} = \left\{e \in D^k: \#\Ti(e) = k-r \right\}$ are all words in the orthogonal basis of~$W_{kr}$. By the same Definition~\ref{phikr}, 
$$ \Phi_{kr}(e) \;=\; \pi(e) \otimes \rho(e) \;\in\; V_{kr} \otimes (\Ti^\perp)^{\otimes r} $$
It follows that one can equivalently write the sum as
\begin{align*}
C_{f,f'} \;&=\; \frac{(k!)^2}{(2k-r)!} \;\sum_{e,e' \in D_{kr}} \langle f, e \rangle \; \left\langle \left(M_{kr} \otimes I\right)\Phi_{kr}(e), \Phi_{kr}(e') \right\rangle \; \langle f', e' \rangle \\
\;&=\; \frac{(k!)^2}{(2k-r)!} \left\langle \left(M_{kr} \otimes I\right) \Phi_{kr}(f),  \Phi_{kr}(f') \right\rangle
\end{align*}

By Proposition~\ref{wv}, since $f \in W_{krm}$, its image $\Phi_{kr}(f) \in V_{krm} \otimes (\Ti^\perp)^{\otimes r}$. Therefore, the application of $M_{kr} \otimes I$ reduces to a multiplication by the eigenvalue~$\mu_{krm}$.
$$ C_{f,f'} \;=\; \frac{(k!)^2}{(2k-r)!} \,\mu_{krm} \left\langle \Phi_{kr}(f),  \Phi_{kr}(f') \right\rangle \;=\; \frac{(k!)^2}{(2k-r)!} \binom{2k-r}{k+m} \left\langle f,f' \right\rangle $$
which is the statement of the theorem.
\end{proof}

\subsection{Proof of Theorem~\ref{structure}}
\label{sec23}

We conclude this section with a proof of Theorem~\ref{structure}.
Our derivation of Theorem~\ref{main2} characterizes the components $W_{krm} \cong V_{krm} \otimes (\mathbb{R}^{d-1})^{\otimes r}$ using word operations. Theorem~\ref{structure} gives a more explicit construction of the spaces $V_{krm}$ as isometric images of spaces of orthogonal polynomials. Indeed, by Definitions~\ref{deltakr}-\ref{ukrm}, $U_{krm}$ is the degree-$m$ space of orthogonal polynomials on the discrete simplex $\Delta_{kr}$. By Definition~\ref{psikr}, this space maps to $V_{krm}$ via the explicit isometry $\Psi_{kr}$, which defines a combination of subwords using the evaluations of a given polynomial on the points of~$\Delta_{kr}$ as coefficients.

The proof uses discrete partial derivatives of polynomials in $\mathbb{R}[x_0,\dots,x_r]$. Let $\mathbf{e}_0,\dots,\mathbf{e}_r$ denote the unit vectors in $\mathbb{Z}^{r+1}$.

\begin{definition}
\label{partial} 
For $i \in \{1,\dots,r\}$, $P \in \R[x_0,\ldots,x_r]$, and $\mathbf{x} \in \mathbb{Z}^{r+1}$ 
$$ (\grad_i P)(\mathbf{x}) \;=\; P(\mathbf{x} + \mathbf{e}_i) - P(\mathbf{x}) $$
\end{definition}

\begin{ex*}
$\grad_1 (x_1^2 + x_2) = (x_1+1)^2 - x_1^2 = 2x_1 + 1$
\end{ex*}

Observe that $\grad_i$ takes a polynomial of degree $m$ to a polynomial of degree $m-1$. Clearly, if the variable $x_0$ does not appear in $P$ then $\grad_i P$ is equivalently given by $P(\mathbf{x} + \mathbf{e}_i) - P(\mathbf{x} + \mathbf{e}_0)$, which may be useful as these two points lie in the same discrete simplex $\Delta_{kr}$.

\begin{lemma}
\label{psi1}
$\Psi_{kr}:\mathbb{R}_{k-r}[x_1,\ldots,x_r] \to V_{kr}$ is an isometry.
\end{lemma}

\begin{proof}
We first show that $\Psi_{kr}$ is injective by induction on $k$. The case $k=r$ is obvious: $1 \mapsto \Tii\Tii\cdots\Tii$. For $k>r$, suppose that $P \in \mathbb{R}_{k-r}[x_1,\ldots,x_r]$ vanishes on~$\Delta_{kr}$. Then $\grad_1 P,\dots,\grad_r P$ are polynomials in $\mathbb{R}_{k-r-1}[x_1,\dots,x_r]$ that vanish on~$\Delta_{(k-1)r}$. By induction, each $\grad_iP = 0$. It follows that $P$ is constant on $\mathbb{Z}^{r+1}$, and since it vanishes on the simplex, it is the zero polynomial as needed.

The equality $\langle P, P' \rangle_{kr} = \left\langle \Psi_{kr}(P), \Psi_{kr}(P') \right\rangle$ follows from Definitions \ref{deltakr} and~\ref{psikr}. By injectivity of $\Psi_{kr}$, the symmetric bilinear pairing $\langle -, - \rangle_{kr}$ restricted to $\mathbb{R}_{k-r}[x_1,\ldots,x_r]$ is an inner product. Its isometric image is all of~$V_{kr}$ because the dimension of both spaces is~$\tbinom{k}{r}$.
\end{proof}

\begin{definition}
\label{sh*}
Let $\sh_{kr}^{*} : \mathbb{R}_{k-r}[x_1,\ldots,x_r] \to \mathbb{R}[x_1,\ldots,x_r]$ 
$$ \left(\sh_{kr}^{*}P\right)(\mathbf{x}) \;:=\; \left(1+k-r-\sum_{i=1}^r x_i\right) P(\mathbf{x}) + \sum_{i=1}^r x_i P(\mathbf{x} - \mathbf{e}_i) $$
\end{definition}

\begin{lemma}
\label{sh*1}
$\sh_{kr}^{*}$ is a linear automorphism of $\mathbb{R}_{k-r}[x_1,\ldots,x_r]$.
\end{lemma}

\begin{proof}
Linearity is clear. Let $P \in \mathbb{R}_{k-r}[x_1,\ldots,x_r]$ be a nonzero polynomial of total degree~$m$. Note that equivalently to Definition~\ref{sh*}
$$ \sh_{kr}^{*}P(\mathbf{x}) \;=\; (1+k-r)P(\mathbf{x}) - \sum_{i=1}^r x_i \grad_i P(\mathbf{x} - \mathbf{e}_i) $$
It follows that the total degree of $\sh_{kr}^{*}P$ is at most $m$. Let $c\,x_1^{a_1}x_2^{a_2}\cdots x_r^{a_r}$ be a top monomial in $P$, of total degree~$a_1 + \dots + a_r =m$. Each $x_i\grad_i$ term yields $a_ic\,x_1^{a_1}x_2^{a_2}\cdots x_r^{a_r}$ plus lower degree monomials. The total degree of $\sh_{kr}^{*}P$ is hence exactly~$m$, as it contains the monomial $(1+k-r-m)c\,x_1^{a_1}x_2^{a_2}\cdots x_r^{a_r}$ where $m \leq k-r$.
\end{proof}

\begin{lemma}
\label{sh*2}
$\sh \circ \Psi_{kr} = \Psi_{(k+1)r} \circ \sh_{kr}^*$ on $\,\mathbb{R}_{k-r}[x_1,\ldots,x_r]$.
\end{lemma}

\begin{proof}
Let $P \in \mathbb{R}_{k-r}[x_1,\ldots,x_r]$. Applying $\sh$ on Definition~\ref{psikr},
\begin{align*}
\sh\,\Psi_{kr}(P) \;&=\; \sum_{\mathbf{d} \in \Delta_{kr}} P(\mathbf{d})\, \sum_{i=0}^r (d_i+1) \Ti^{d_0}\Tii\Ti^{d_1}\Tii\cdots\Tii\Ti^{d_i+1}\Tii\cdots\Tii\Ti^{d_r} \\ 
\;&=\; \sum_{\mathbf{d} \in \Delta_{(k+1)r}} \left(\sum_{i=0}^r d_i \,P(\mathbf{d} - \mathbf{e_i})\right) \Ti^{d_0}\Tii\Ti^{d_1}\Tii\cdots\Tii\Ti^{d_r} \\ 
\;&=\; \sum_{\mathbf{d} \in \Delta_{(k+1)r}} \left(\sh_{kr}^*P\right)(\mathbf{d}) \;\Ti^{d_0}\Tii\Ti^{d_1}\Tii\cdots\Tii\Ti^{d_r} \;=\; \Psi_{(k+1)r}\left(\sh_{kr}^*P\right)
\end{align*}
as claimed.
\end{proof}

\begin{lemma}
\label{shwkrm}
$\sh : V_{krm} \xrightarrow{\;\sim\;} V_{(k+1)rm}\;$ for $m \leq k-r$.
\end{lemma}

\begin{proof}
This is implicit in the proof of Lemma~\ref{lem:creation_annihilation}. 
Recall that $V_{kr0},V_{kr1},\dots$ are pairwise orthogonal, and $\del V_{(k+1)rm} = V_{krm}$ for $m \leq k-r$. Let $f \in V_{krm}$ and $g\in (V_{(k+1)rm})^{\perp}$. From the duality $\langle \sh f, g \rangle = \langle f, \del g \rangle$ in~\S\ref{operators} it follows that $\sh f \perp V_{(k+1)rm'}$ for any $m \neq m' \leq k-r$, and similarly $\sh f \perp \ker\del = V_{(k+1)r(k-r+1)}$. Therefore, $\sh f \in V_{(k+1)rm}$ by orthogonality. 
Since $\sh$ is injective and the dimensions agree, the isomorphism follows.

We note that this restriction of $\sh$ is an isometry up to a scalar factor. By the proof of Lemma~\ref{lem:creation_annihilation}, $\langle \sh f, \sh f'\rangle = \beta_{(k+1)rm}\langle f, f'\rangle$ for $f,f' \in V_{krm}$.
\end{proof}

\begin{proof}[Proof of Theorem~\ref{structure}]
The proof proceeds by induction on $k$. For $k=r$ the space $U_{rr0} = \mathrm{span}\{1\}$ maps via $\Psi_{rr}$ to $W_{rr0} = \mathrm{span}\{\Tii\Tii\cdots\Tii\}$, and the claim holds.

We now prove the case of $k+1$ assuming~$k$. For each $m \in \{0,\dots,k-r\}$ we have an isomorphism
$$ \Psi_{kr}:U_{krm} \;\xrightarrow{\;\sim\;}\;  V_{krm} $$
Using Lemmas~\ref{sh*1}, \ref{sh*2}, and~\ref{shwkrm}, the following map is an isomorphism as well,
$$ \Psi_{(k+1)r} : \sh_{kr}^*\,U_{krm} \;\xrightarrow{\;\sim\;}\; \sh\,V_{krm} = V_{(k+1)rm} $$
Moreover, since the map $\Psi_{(k+1)r}$ is an isometry by Lemma~\ref{psi1}, and since the components $V_{(k+1)r0}$, $V_{(k+1)r1}, \dots, V_{(k+1)r(k-r)}$ are pairwise orthogonal, the polynomial spaces that we got, $\sh_{kr}^{*}U_{kr0},\sh_{kr}^{*}U_{kr1},\dots,\sh_{kr}^{*}U_{kr(k-r)}$, are also orthogonal with respect to the inner product $\langle-,-\rangle_{(k+1)r}$. 

The nonzero elements of $\sh_{kr}^{*}U_{krm}$ are polynomials of degree $m$ as noted in the proof of Lemma~\ref{sh*1}. By the orthogonality of the spaces $\sh_{kr}^{*}U_{krm}$ and by the definition of $U_{(k+1)rm}$, necessarily $\sh_{kr}^{*} U_{krm} = U_{(k+1)rm}$ for every $m \in \{0,\dots,k-r\}$. Therefore, the isomorphism
$$ \Psi_{kr}:U_{(k+1)rm} \;\xrightarrow{\;\sim\;}\;  V_{(k+1)rm} $$
holds for all $m \leq k-r$ as required. The remaining case $m = k-r+1$ follows by noting that the isometry $\Psi_{(k+1)r}$ must take the orthogonal complement of these $U_{(k+1)rm}$ in $\R_{k-r+1}[x_1,\ldots,x_r]$ bijectively to the orthogonal complement of their images $V_{(k+1)rm}$ in~$V_{(k+1)r}$.
\end{proof}

\medskip

\begin{rem}
\label{homogeneous}
\emph{On Homogeneous Discrete Orthogonal Polynomials}

\smallskip \noindent 
Although $\Psi_{kr}$ is defined on $\mathbb{R}[x_0,\dots,x_r]$, the polynomial spaces $U_{krm}$ leave $x_0$ out. This map evaluates them only on the discrete simplex $\Delta_{kr}$ that lies in the hyperplane $x_0 + \dots + x_r = (k-r)$, so
it is well defined on the quotient
$$ \mathbb{R}[x_0,\dots,x_r] / \langle x_0 + \dots + x_r - (k-r) \rangle $$
A natural alternative is hence given by homogeneous polynomials. Explicitly, $P \in U_{krm}$ is uniquely made $m$-homogeneous via $1 \mapsto (x_0+\cdots+x_r)/(k-r)$, and recovered by $x_0 \mapsto (k-r)-(x_1 + \dots + x_r)$. We therefore define the following.
\end{rem}

\begin{definition}
\label{hkrm}
Let $H_{krm}$ be the space of homogeneous polynomials of total degree~$m$, that are orthogonal to $H_{kr0},\dots,H_{kr(m-1)}$ with respect to $\langle-,-\rangle_{kr}$.
\end{definition}

\begin{cor}
\label{hu}
$H_{krm} \cong U_{krm}\;$ preserving $\Psi_{kr}$ and $\langle-,-\rangle_{kr}$
\end{cor}

\begin{ex*} Compare the following $H_{krm}$ to the corresponding $U_{krm}$ in~\S\ref{poly}: \\
$H_{310} = \mathrm{span}\left\{1\right\}$,  
$H_{311} = \mathrm{span}\left\{x_0-x_1\right\}$,
$H_{312} = \mathrm{span}\left\{x_0^2 - 10x_0x_1 + x_1^2\right\}$ \\
$H_{320} =  \mathrm{span}\left\{1\right\}$,  
$H_{321} = \mathrm{span}\left\{x_0 - 2x_1 + x_2,\; x_2-x_0\right\}$
\end{ex*}

Theorem~\ref{main2} fully diagonalizes word statistics, with eigenspaces coming from $U_{krm} \otimes (\Ti^\perp)^{r}$. The spaces $H_{krm}$ may refine this classification by allowing meaningful bases choices for each component.

\begin{ex*} $H_{krm} = H_{krm}^{\text{even}} \oplus H_{krm}^{\text{odd}}$ with respect to $(x_0,\dots,x_r) \mapsto (x_r,\dots,x_0)$. For example, $H_{310}$ and $H_{312}$ are even, $H_{311}$ is odd, and $H_{321}$ has one-dimensional components of either parity. This leads to word statistics $\#f(w)$ either invariant or flipping sign when reversing~$w$. 
\end{ex*}

\begin{ex*}
The symmetric group $S_{r+1}$ acts on $\{x_0,\dots,x_r\}$ preserving $\Delta_{kr}$ and thereby the decomposition into $H_{krm}$. Therefore, all $H_{krm}$ decompose into representations of $S_{r+1}$. 
\end{ex*}

\section{Multi-Sample}

\subsection{Proof of Theorem \ref{main3}}
\label{proof3}

Let $\Sigma = \{\TA,\TB,\TC,\dots\}$ be a finite alphabet, and $\mathbf{n} = (n_\TA,n_\TB,n_\TC,\dots)$ as usual. We first formulate the word statistics in the model~$\mathcal{W}'(\mathbf{n})$ as generalized U-statistic, as described in~\S\ref{gustats}.

Consider $|\Sigma|$ samples of independent random variables: $X_{\TA1}, X_{\TA2}, \dots, X_{\TA n_\TA};$ $X_{\TB1}, \dots, X_{\TB n_\TB};$ $X_{\TC1}, \dots, X_{\TC n_\TC};$ $\dots$ that are uniformly distributed in the unit interval~$[0,1]$. To be precise, 
$$ {\mathbf{X}}_{\mathbf{n}} \;:=\; \left\{ X_{\TX i}\right\}_{\TX \in \Sigma,\; i \in \{1,\dots, n_\TX\}} \;\,\sim\; U\left([0,1]^{n_\TA} \times [0,1]^{n_\TB} \times \dots \right) $$
Generically, such a sequence of $n = |\mathbf{n}|$ random variables induces an $n$-letter word, by reading their labels in order of occurrence along the unit interval. Namely, we define a map 
$$ \word : [0,1]^n \;\to\; \tbinom{\Sigma}{\mathbf{n}} $$
such that $\word({\mathbf{X}}_{\mathbf{n}})$ starts with the label $\TX$ of the smallest number $X_{\TX i}$, then the label of the second smallest number, and so on. 

\begin{ex*}
If $\mathbf{n}=(2,2)$ and $X_{\TA2} < X_{\TB1} < X_{\TB2} < X_{\TA1}$, then $\word({\mathbf{X}}_{\mathbf{n}}) = \TA\TB\TB\TA$. 
\end{ex*}

Clearly, the distribution of $\word({\mathbf{X}}_{\mathbf{n}})$ is uniform over all $\tbinom{n}{\mathbf{n}}$ words in~$\tbinom{\mathsmaller{\Sigma}}{\mathbf{n}}$, exactly as in the model~$\mathcal{W}'(\mathbf{n})$. Note that with probability one ${\mathbf{X}}_{\mathbf{n}}$ is \emph{generic}, with $n$ distinct numbers. Hence, we ignore ties in the definition of $\word({\cdots})$, or, if needed, break them lexicographically.

Consider $\kpa = (k_\TA,k_\TB,\dots)$ such that $k_{\TX} \leq n_{\TX}$ for every $\TX \in \Sigma$, and sets of indices $I_\TX = \{i_{\TX 1},i_{\TX 2},\dots,i_{\TX k_\TX}\} \subseteq \{1,\dots,n_{\TX}\}$. These sets let us \emph{restrict} ${\mathbf{X}}_{\mathbf{n}}$ to $k = |\kpa|$ variables as follows.
$$ \mathbf{X}_{\mathbf{n}}\left[I_\TA,I_\TB,\dots\right] \;:=\; \left(X_{\TA (i_{\TA 1})}, \dots, X_{\TA (i_{\TA k_{\TA}})}; X_{\TB (i_{\TB 1})}, \dots, X_{\TB (i_{\TB k_{\TB}})}; \dots \right) $$
There are \raisebox{0.2em}{$\prod_\TX\tbinom{n_\TX}{k_\TX}$} such restrictions. Each one of them induces a random $k$-letter subword $u = \word(\mathbf{X}_{\mathbf{n}}\left[I_\TA,I_\TB,\dots\right])$ that is uniformly distributed in~$\tbinom{\mathsmaller{\Sigma}}{\kpa}$, and occurs at the $k$ positions $\bigcup_\TX I_\TX$ in $w = \word(\mathbf{X}_{\mathbf{n}})$.

Let $f \in W_{\kpa} = \mathbb{R}\tbinom{\mathsmaller{\Sigma}}{\kpa}$. Taking a uniformly random word $w \in \mathcal{W}'(\mathbf{n})$ that is induced from a random sequence~$\mathbf{X}_{\mathbf{n}}$ as above, the random variable $\#f(w)$ takes the following form.
\begin{align*}
\#f(w) \;&=\; \sum_{u \in \tbinom{\mathsmaller{\Sigma}}{\kpa}} f_u \; \#u\left(\word({\mathbf{X}}_{\mathbf{n}})\right) \\[0.25em]
\;&=\; \sum_{u \in \tbinom{\mathsmaller{\Sigma}}{\kpa}} f_u \, \sum_{I_\TA,I_\TB,\dots} \mathbbm{1}\left[\,\word\left({\mathbf{X}}_{\mathbf{n}}\left[I_\TA,I_\TB,\dots\right]\right)=u\,\right] \\[0.25em]
\;&=\; \sum_{I_\TA,I_\TB,\dots} f_{\word\left({\mathbf{X}}_{\mathbf{n}}\left[I_\TA,I_\TB,\dots\right]\right)} 
\end{align*}

This formulation implies that the normalized $\tilde\#f = \#f/\prod_\TX\tbinom{n_\TX}{k_\TX}$ is a generalized U-statistic, as in~\S\ref{gustats}. Its kernel is the function $f_{\word({\cdots})}$, whose inputs are $k$ numbers in $[0,1]$, where $k_\TX$ inputs are labeled by each $\TX \in \Sigma$. Its output is the coefficient $f_u$ of the word $u \in \tbinom{\mathsmaller{\Sigma}}{\kpa}$, obtained by reading the labels of the given inputs according to their order on the interval~$[0,1]$. For convenience of notation, we sometimes write $f_{\word}\left(\mathbf{X}_{\kpa}\right)$ instead of $f_{\word\left(\mathbf{X}_{\kpa}\right)}$.

Using this form, all the word statistics in $W_{\kpa}$ can be expressed as generalized U-statistics on the same set of samples. We now define coefficients and functions that arise when computing their second moments.

\begin{definition}
\label{mrff}
Let $f, f' \in W_{\kpa}$ for $\kpa = (k_\TA,k_\TB,\dots)$, and let $\mathbf{r} = (r_\TA,r_\TB,\dots)$ be such that $0 \leq r_\TX \leq k_\TX$ for every~$\TX \in \Sigma$, abbreviated as $\mathbf{r} \leq \kpa$. We denote
$$ m_{\mathbf{r}}(f,f') \;:=\; \Exp\left[f_{\word}\left(\mathbf{X}_{\kpa}\right) f'_{\word} \left(\mathbf{X}_{\mathbf{r}}\cup\mathbf{X}'_{\kpa-\mathbf{r}}\right)\right] $$
where
$$ \mathbf{X}_{\kpa} \;=\; \left(\substack{X_{\TA 1} , \dots, X_{\TA k_\TA} \\ X_{\TB 1} , \dots, X_{\TB k_\TB} \\ \vdots }\right) \;\;\;\;\;\; \mathbf{X}_{\mathbf{r}}\cup\mathbf{X}'_{\kpa-\mathbf{r}} \;=\; \left(\substack{X_{\TA 1} , \dots, X_{\TA r_\TA}, X'_{\TA (r_\TA+1)}, \dots, X'_{\TA k_\TA} \\ X_{\TB 1} , \dots, X_{\TB r_\TB}, X'_{\TB (r_\TB+1)}, \dots, X'_{\TB k_\TA} \\ \vdots }\right) $$
such that all $\{X_{\TX i}\}$ and $\{X'_{\TX i}\}$ are independent random variables uniformly distributed in the interval~$[0,1]$. 
\end{definition}

Here is an equivalent way to write these coefficients, which is obtained by averaging separately the unique inputs of each function.
$$ m_{\mathbf{r}}(f,f') \;=\; \Exp_{\mathbf{X}_{\mathbf{r}}}\Bigl[\; \Exp_{\mathbf{X}_{\kpa}} \left[ f_{\word}\left({\mathbf{X}}_{\kpa}\right) \bigm| {\mathbf{X}_{\mathbf{r}}}\right] \cdot \Exp_{\mathbf{X}_{\kpa}} \left[ f'_{\word} \left(\mathbf{X}_{\mathbf{r}}\cup\mathbf{X}'_{\kpa-\mathbf{r}}\right) \bigm| {\mathbf{X}_{\mathbf{r}}} \right] \;\Bigr] $$
This leads to the following family of functions, where a subset of the inputs to $f$ are given and the others are averaged. 
\begin{definition}
\label{fr}
For $\kpa,\mathbf{r}$ as above, every $f \in W_{\kpa}$ is assigned a function
$$ f^{(\mathbf{r})} : \left([0,1]^{r_\TA} \times [0,1]^{r_\TB} \times \cdots\right) \;\to\; \mathbb{R} $$
$$ f^{(\mathbf{r})} \left(\mathbf{X}_{\mathbf{r}} \right) \;=\; \Exp_{\mathbf{X}_{\kpa}} \left[ f_{\word}(\mathbf{X}_{\mathbf{r}}\cup\mathbf{X}_{\kpa-\mathbf{r}}) \mid \mathbf{X}_{\mathbf{r}} \right] $$
and then
$$ m_{\mathbf{r}}(f,f') \;=\; \Exp_{\mathbf{X}_{\mathbf{r}}} \left[\, f^{(\mathbf{r})}( \mathbf{X}_{\mathbf{r}}) \cdot f'^{(\mathbf{r})}( \mathbf{X}_{\mathbf{r}}) \,\right] $$
$$ m_{\mathbf{r}}(f) \;:=\; m_{\mathbf{r}}(f,f) \;=\; \Exp_{\mathbf{X}_{\mathbf{r}}} \left[\, (f^{(\mathbf{r})} (\mathbf{X}_{\mathbf{r}}))^2 \,\right] $$
\end{definition}

\begin{rmk*}
The functions $f^{(\mathbf{r})}$ are defined almost everywhere in $[0,1]^r$ with respect to the uniform measure, since $f_{\word}$ is well-defined wherever no two coordinates are the same. 
\end{rmk*}

\begin{ex*}
$m_{(0,0,\dots)}(f) = f^{(0,0,\dots)}$ is the constant $\Exp \left[f_{\word}\right]$.
\end{ex*}

\begin{ex*}
$m_{\kpa}(f) = \Exp \left[(f_{\word})^2\right]$ since $f^{(\kpa)}$ is exactly $f_{\word}$.
\end{ex*}

\begin{ex*}
If $f = \TA\TA\TB - \TB\TA\TA$, then $f^{(0,1)}(b) = b^2 - (1-b)^2$ and $m_{(0,1)}(f) = \tfrac13$.
\end{ex*}

We write the second moments of generalized U-statistics as a sum over $\mathbf{r} = (r_\TA,r_\TB,\dots)$ with these coefficients, similar to Proposition~\ref{varun} in the one-sample case.

\begin{lemma}
\label{Eff}
For a random $w \in \tbinom{\mathsmaller{\Sigma}}{\mathbf{n}}$ distributed according to~$\mathcal{W}'(\mathbf{n})$,
$$ \Exp_w \left[\# f(w)\,\#f'(w)\right] \;=\; \sum_{\mathbf{r} \leq \kpa}\, m_{\mathbf{r}}(f,f') \prod_{\TX\in\Sigma} \binom{n_{\TX}}{r_{\TX}} \binom{n_{\TX}-r_{\TX}}{k_{\TX}-r_{\TX}} \binom{n_{\TX}-k_{\TX}}{k_{\TX}-r_{\TX}} $$
\end{lemma}

\begin{proof}
This formula follows by a straightforward expansion of $\#f\#f'$ as a double sum of $f_{\word}(\mathbf{X}_n[I_\TA,I_\TB,\dots]) \cdot f_{\word}'(\mathbf{X}_n[I'_\TA,I'_\TB,\dots])$,
and grouping together terms with the same numbers $r_\TX = |I_\TX \cap I'_\TX|$ of common inputs from each sample~$\{X_{\TX 1},\dots,X_{\TX n_\TX}\}$.
\end{proof}

In general, the notion of rank for a generalized U-statistic requires more than one number, differently from the one-sample case, cf.~Definition~\ref{rank}. Indeed, the term of $\mathbf{r}$ in Lemma~\ref{Eff} has order $\prod_\TX n_\TX^{2k_\TX - r_\TX}$ for large $n_\TX$s, unless $m_{\mathbf{r}}(f,f') = 0$, and without any assumptions on the relations between the $n_\TX$s, one cannot tell which term dominates. Here we retain the assumptions of Theorem~\ref{main3} that $n_\TX/n \to p_\TX>0$ for all~$\TX$. In this case, the leading terms have order $n^{2k-r}$ for the smallest $r = |\mathbf{r}|$ with at least one nonzero $m_\mathbf{r}(f,f')$, which gives rise to the following definition.

\begin{definition}
\label{rankf}
The \emph{rank} of a nonzero $f \in W_{\kpa}$ is the smallest $\ell$ such that $m_\mathbf{r}(f) \neq 0$ for some $\mathbf{r} \leq \kpa$ with $|\mathbf{r}| = \ell$. Equivalently, $\rank f$ is the smallest~$|\mathbf{r}|$ for which some $f^{(\mathbf{r})}$ is not almost surely zero.
\end{definition}

The following corollary summarizes the proof so far. The question of scaling $\#f$ has been reduced to finding the rank of a generalized U-statistic with kernel~$f_{\word}$.

\begin{cor}
\label{Ef2}
Let $w \in \tbinom{\mathsmaller{\Sigma}}{\mathbf{n}}$ be a random word in the model~$\mathcal{W}'(\mathbf{n})$, where $\mathbf{n}/n \to \mathbf{p} \in (0,1)^{|\Sigma|}$ as $n = |\mathbf{n}| \to \infty$. For every nonzero $f \in W_{\kpa}$
$$ \Exp_w \left[\tilde\#f(w)^2\right] \;=\; \frac{C'_{f,\mathbf{p}} + o_{n}(1)}{n^{\rank f}} $$
where
$$ C'_{f,\mathbf{p}} \;:=\; \sum_{|\mathbf{r}|={\rank f}}\; m_{\mathbf{r}}(f) \prod_{\TX \in \Sigma} \frac{k_{\TX}!^2}{r_{\TX}!(k_{\TX}-r_{\TX})!^2 p_\TX^{r_{\TX}}} \;>\; 0$$
\end{cor}

\begin{proof}
We recall the normalization $\tilde\#f = \#f/\prod_\TX\tbinom{n_\TX}{k_\TX}$ and use Lemma~\ref{Eff}. After simplifying all the binomials via 
$$ \binom{n_\TX-a}{b} \;=\; \left(\frac{p_\TX^b }{b!} + o_n(1)\right) n^b  \;\;\;\;\;\;\;\;\text{for}\;a,b\in\mathbb{N} $$
every term $\mathbf{r} \leq \kpa$ has order $n^{-|\mathbf{r}|}$. The terms with $|\mathbf{r}|<\rank f$ vanish since $m_{\mathbf{r}}(f)=0$ by the definition, while there exists at least one term with $|\mathbf{r}|=\rank f$ such that $m_{\mathbf{r}}(f) \neq 0$. The coefficients $m_{\mathbf{r}}(f)$ are always nonnegative, so the answer indeed has order $n^{-\rank f}$. The $o_n(1)$ term absorbs all terms with $|\mathbf{r}|>\rank f$. This yields the stated expression for~$C'_{f,\mathbf{p}}$.
\end{proof}

Our next goal is to establish a relation between the classification of statistics by rank and the decomposition of $W_{\kpa}$ into $W_{\kpa r}$, defined in~\S\ref{repr} using representations of~$S_k$. This is broken into Lemma~\ref{easy}, implying that the rank of $W_{\kpa r}$ is at least~$r$, and Lemma~\ref{hard}, that the rank is at most~$r$. A~similar line of argument was taken in~\cite{even2020patterns} in the study of permutation patterns.

\begin{lemma}
\label{easy}
Given $f \in W_{\kpa \ell}$, for every $\mathbf{r} \leq \kpa$ with $|\mathbf{r}| = r < \ell$ the corresponding $f^{(\mathbf{r})} = 0$ almost everywhere in $[0,1]^r$.
\end{lemma}

\begin{lemma}
\label{hard}
Given a nonzero $f \in W_{\kpa r}$, there exists $\mathbf{r} \leq \kpa$ with $|\mathbf{r}|=r$ such that $f^{(\mathbf{r})} \neq 0$ with positive probability in $[0,1]^r$.
\end{lemma}

\begin{proof}[Proof of Lemma~\ref{easy}]
Given $f$ and $\mathbf{r}$, we analyze the function $f^{(\mathbf{r})}$. Every generic input $\mathbf{X}_\mathbf{r} \in [0,1]^r$ to this function fixes a word $v = \word(\mathbf{X}_\mathbf{r})$, and conversely any word $v \in \tbinom{\mathsmaller{\Sigma}}{\mathbf{r}}$ is obtained from some input. This induces a partition of the domain:
$$ [0,1]^r \;=\; \bigcup_{v} D_v \;\;\text{where}\;\; D_v \;=\; \word^{-1}\left(v\right) \;\;\text{for}\;\; v \in \tbinom{{\Sigma}}{\mathbf{r}} $$

An input $\mathbf{X}_\mathbf{r} \in D_v$ is augmented to $\mathbf{X}_{\kpa} = \mathbf{X}_\mathbf{r} \cup \mathbf{X}_{\kpa - \mathbf{r}}$ in the conditional expectation $\Exp[f_{\word}(\mathbf{X}_{\kpa}) \mid \mathbf{X}_\mathbf{r}]$ defining $f^{(\mathbf{r})}$. This induces an occurrence of~$v$ in the random word $u = \word(\mathbf{X}_{\kpa})$. Let $I(\mathbf{X}_{\kpa}) \subseteq \{1,\dots,k\}$ denote the positions of this copy of $v$. In other words, we rank the $k$ numbers $\mathbf{X}_{\kpa}$ in increasing order, and $I(\mathbf{X}_{\kpa})$ comprises the $r$ rankings of the inputs~$\mathbf{X}_\mathbf{r}$. With probability one, there are no ties and $I$ is well defined. Clearly, given a generic $\mathbf{X}_{\mathbf{r}}$, any set of $I \subseteq \tbinom{[k]}{r}$ is obtained from some~$\mathbf{X}_{\kpa - \mathbf{r}}$. By the law of total expectation:
$$ f^{(\mathbf{r})}\left(\mathbf{X}_\mathbf{r}\right) \;=\; \sum_{I\subseteq \tbinom{[k]}{r}} \PR_{\mathbf{X}_{\kpa}}\left(I(\mathbf{X}_{\kpa})=I \mid \mathbf{X}_{\mathbf{r}} \right) \; \Exp_{\mathbf{X}_{\kpa}} \left[ f_{\word}(\mathbf{X}_{\kpa}) \mid \mathbf{X}_{\mathbf{r}}, I \right] $$

This expectation is computed by taking a uniform $\mathbf{X}_{\kpa - \mathbf{r}} \in [0,1]^{k-r}$. The conditioning only concerns how many of them fall within each interval between~$\mathbf{X}_{\mathbf{r}}$. Therefore, any reordering of $\mathbf{X}_{\kpa - \mathbf{r}}$ would preserve the overall expectation on the one hand, while permuting the letters in the non-$I$ positions of every $\word(\mathbf{X}_{\kpa})$ on the other hand. We use the stabilizer and $a_I \in A$ from Definition~\ref{stab}, and average over all such permutations:
\begin{align*}
f^{(\mathbf{r})}\left(\mathbf{X}_\mathbf{r}\right) \;=&\; \sum_{I} \PR\left(I \mid \mathbf{X}_{\mathbf{r}} \right) \; \frac{1}{(k-r)!} \sum_{\tau \in \stab I} \Exp \left[ f_{\word(\mathbf{X}_{\kpa})\tau} \mid \mathbf{X}_{\mathbf{r}}, I \right] \\
\;=&\;  \sum_{I} \PR\left(I \mid \mathbf{X}_{\mathbf{r}} \right) \; \frac{1}{(k-r)!} \sum_{\tau \in \stab I} \Exp \left[ (f\tau)_{\word(\mathbf{X}_{\kpa})} \mid \mathbf{X}_{\mathbf{r}}, I \right] \\
\;=&\;  \sum_{I} \PR\left(I \mid \mathbf{X}_{\mathbf{r}} \right) \; \frac{1}{(k-r)!} \Exp \left[ \left(f a_I\right)_{\word(\mathbf{X}_{\kpa})} \mid \mathbf{X}_{\mathbf{r}}, I \right]
\end{align*}

By the assumption of the lemma, the function $f \in W_{\kpa \ell}$ for $\ell > r = |I|$. In this case $f a_I = 0$ by Lemma~\ref{stabsum}(a) stated and proven below. Hence, the argument of the conditional expectation is zero for almost every~$\mathbf{X}_{\kpa}$ given a generic~$\mathbf{X}_{\mathbf{r}}$, and thus $f^{(\mathbf{r})}=0$ almost everywhere.
\end{proof}

\begin{proof}[Proof of Lemma~\ref{hard}]
We continue the analysis of $f^{\mathbf{(r)}}$ from the previous proof, but this time with $f \in W_{\kpa r}$ where $|\mathbf{r}|=r$.

The $r$ coordinates of $\mathbf{X}_\mathbf{r}$ divide $[0,1]$ into $r+1$ intervals, whose lengths we denote by $\Delta x_0, \Delta x_1, \dots, \Delta x_r$ in the order they appear along $[0,1]$. Note that $\Delta x_0 + \dots + \Delta x_r = 1$. Similarly, given a subset $I = \{I_1,\dots,I_r\} \subseteq \{1,\dots,k\}$, we denote its gaps by $\Delta I_j = I_{j+1} - I_{j} - 1$, where $0 \leq j \leq r$ and by convention $0 = I_0 < I_1 < \dots < I_r < I_{r+1} = k+1$. We compute the probabilities in the expansion of~$f^{(\mathbf{r})}$, in terms of these variables:
$$ \PR_{\mathbf{X}_{\kpa}}\left(I(\mathbf{X}_{\kpa})=I \mid \mathbf{X}_{\mathbf{r}} \right) \;=\; \binom{k-r}{\Delta I_0\;\Delta I_1\;\cdots\;\Delta I_r} \Delta x_0^{\Delta I_0} \Delta x_1^{\Delta I_1} \cdots \Delta x_r^{\Delta I_r} $$

Next, we examine the conditional expectations in that expansion. Given $u \in \tbinom{\mathsmaller{\Sigma}}{\kpa}$, the averaging $ua_I$ is the formal sum of all words in $\tbinom{\mathsmaller{\Sigma}}{\kpa}$ whose restriction to $I$ is the same as~$u$. Hence, given $v \in \tbinom{\mathsmaller{\Sigma}}{\mathbf{r}}$, the coefficient $(fa_I)_u$ is the same for all $u$ whose restriction to $I$ is~$v$. Denote $F_{I,v} := (fa_I)_u \in \mathbb{R}$ for any such~$u$. It follows that for each $I$ and~$v$, 
$$ \Exp \left[ \left(f a_I\right)_{\word(\mathbf{X}_{\kpa})} \mid \mathbf{X}_{\mathbf{r}}, I \right]
\;\equiv\; F_{I,v} $$
as a function of $\mathbf{X}_{\mathbf{r}}$ on the domain $D_v$, where $D_v = \word^{-1}(v) \subseteq [0,1]^r$ as in the previous proof.

In conclusion, the function $f^{(\mathbf{r})}$ is piecewise polynomial. For each one of the~$D_v$, where $v \in \tbinom{\mathsmaller{\Sigma}}{\mathbf{r}}$, it has the form
$$ f^{(\mathbf{r})}\left(\mathbf{X}_\mathbf{r}\right) \;=\; \sum_{I\subseteq \tbinom{[k]}{r}} \frac{F_{I,v}}{\Delta I_0!\Delta I_1!\cdots\Delta I_r!} \;\Delta x_0^{\Delta I_0} \Delta x_1^{\Delta I_1} \cdots \Delta x_r^{\Delta I_r} $$

By the assumption of the lemma, $f \in W_{\kpa r}$. By Lemma~\ref{stabsum}(b) below, there exists $I \subseteq \{1,\dots,k\}$ of size $|I|=r$ with $fa_I \neq 0$. Let $u \in \tbinom{\mathsmaller{\Sigma}}{\kpa}$ be such that $(fa_I)_u \neq 0$. Let $v \in \Sigma^r$ be such that the restriction of $u$ to $I$ equals~$v$. Let $\mathbf{r} = (\#\TA(v),\#\TB(v),\dots) \leq \kpa$, so that $v \in \tbinom{\mathsmaller{\Sigma}}{\mathbf{r}}$. For this $\mathbf{r}$, in the expansion of $f^{(\mathbf{r})}$ on the subdomain $D_v$, the term that corresponds to $I$ has a nonzero coefficient, because $F_{I,v} = (fa_I)_u \neq 0$. 

The polynomial expressing $f^{(\mathbf{r})}$ in the variables $\{\Delta x_j\}$ on $D_v$ is nonzero, since it has a nonzero coefficient and all the monomials appearing in the expansion are clearly distinct for different sets~$I$. Note that this polynomial is homogeneous of total degree~$k-r$. The proof will be completed by showing that $f^{(\mathbf{r})}$ is a nonzero polynomial in the original variables $\mathbf{X}_{\mathbf{r}}$ as well. 

Indeed, we substitute $(\Delta x_0,\dots,\Delta x_r) = (x_1,x_2-x_1,\dots,1-x_r)$ where $x_1<\dots<x_r$ are the coordinates of $\mathbf{X}_{\mathbf{r}}$, appropriately reordered and relabeled. This affine transformation is invertible when applied to homogeneous polynomials of a given degree, so that the resulting polynomial in $x_1,\dots,x_r$ is nonzero as well. Different polynomials may arise depending on the $r!$ ordering types of $\mathbf{X}_{\mathbf{r}}$'s coordinates, but within~$D_v$ they are nonzero. This means that $f^{(\mathbf{r})}(\mathbf{X}_{\mathbf{r}}) \neq 0$ almost everywhere in $D_v$ as required.  
\end{proof}

\begin{lemma}
\label{stabsum}
Consider a nonzero $f \in W_{\kpa r}$.
\vspace{0.25em}
\begin{samepage}
\begin{enumerate}
\itemsep0.5em
\item[(a)] 
$f a_I = 0$ for every $I \subseteq \{1,\dots,k\}$ of size $|I|<r$.
\item[(b)] 
There exists $I \subseteq \{1,\dots,k\}$ of size $|I|=r$ with $f a_I \neq 0$.
\end{enumerate}
\end{samepage}
\end{lemma}

\begin{proof}[Proof of Lemma~\ref{stabsum}(a)]
Let $I \subseteq \{1,\dots,k\}$ be of size $|I|<r$, and $f \in W_{\kpa r}$. It is sufficient to show $f a_I = 0$ for $f$ in every one of the simple $A$-modules in~$W_{\kpa r}$ as given by the direct sum in Definition~\ref{Wkkr}, and the general case follows by linearity. Hence, let $f \in \Theta[T] S^{\lmda}$ with a partition $\lmda \vdash k$ such that $\lambda_\TA = k-r$, where $T$ is a semistandard table with shape~$\lmda$ and composition~$\kpa$, see~\S\ref{repr}. Moreover, since the map $\Theta[T]$ is equivariant to the action of~$A$, it is enough to show $f' a_I = 0$ for $f' \in S^{\lmda}$ such that $f = \Theta[T]f'$. Recalling the definition $S^{\lmda} = \alpha_{\lmda} b_{\lmda} A$, it is left to prove $b_{\lmda} \sigma a_I = 0$ for every $\sigma \in S_k$. Note that $\sigma a_I = a_{(\sigma I)} \sigma$ where $\sigma I = \{\sigma(i) \mid  i \in I\}$ and $|\sigma I|<r$ as well, so the general case would follow from showing~$b_{\lmda}a_I = 0$. 

These elements were defined as $a_I = \sum_{\tau \in \stab I} \tau$ and $b_{\lmda} = \sum_{\tau \in Q_{\lmda}} \mathrm{sign}(\tau) \tau$, and the subgroup~$Q_{\lmda}$ as all the permutations of $\{1,\dots,k\}$ permuting the numbers within certain $\lambda_\TA$ subsets that compose this set. The remaining argument is essentially Lemma 4.23 in~\cite{fulton2013representation}. Since $|I|<r = k-\lambda_\TA$, there exists $i,j \not\in I$ such that the transposition $(ij) \in Q_{\lmda} \cap \stab I$. Therefore, $b_{\lmda} (ij) = -b_{\lmda}$ and $(ij)a_I = a_I$ so that $b_{\lmda}a_I = -b_{\lmda}a_I = 0$ as required.
\end{proof}

\begin{proof}[Proof of Lemma~\ref{stabsum}(b)]
Let $f \in W_{\kpa r}$. First, represent $f = \sum_T f^{(T)}$ according to the direct sum in Definition~\ref{Wkkr}, so that at least one $f^{(T)} \neq 0$. Then, let $\lmda = \lmda(T)$ and $f' \in S^{\lmda}$ be such that $f^{(T)} = \Theta[T]f'$, so clearly $f' \neq 0$. Finally, expand $f' = \sum_u f'_u u$ over $u \in \tbinom{\mathsmaller{\Sigma}}{\lmda}$, and let $f'_u \in \mathbb{R}$ be a nonzero coefficient in this expansion. 

Recall that $\lambda_\TA = k-r$, and let $I$ be all the positions of non-$\TA$ letters in~$u$, so that $|I|=r$. Since each term in $a_I$ fixes those positions, $u a_I = (k-r)! u$. Conversely, the expansion of $v a_I$ for any other $v \in \tbinom{\mathsmaller{\Sigma}}{\lmda}$ does not contain a term with~$u$. From $u a_I \neq 0$ and $f'_u \neq 0$ it follows that $f'a_I \neq 0$, and $f^{(T)}a_I \neq 0$ because $\Theta[T]$ is an embedding of $S^{\lmda}$ in $W_{\kpa r}$. Since $a_I$ acts separately on each $A$-module in the direct sum, $f a_I \neq 0$ as required.
\end{proof}

From Lemmas~\ref{easy}-\ref{hard} it follows that $\rank f = r$ for every nonzero $f \in W_{\kpa r}$. Together with Corollary~\ref{Ef2}, it follows that the second moment $\Exp[\tilde\#f^2]$ has the leading term~$C'_{f,\mathbf{p}}/n^{r}$.
This proves the first part of Theorem~\ref{main3}.

\smallskip

We now consider $\Exp[\tilde\#f\,\tilde\#f']$ for $f \in W_{\kpa r}$ and $f' \in W_{\kpa r'}$ with $r' < r$. The second part of Theorem~\ref{main3} claims that these off-diagonal terms are only $o(n^{-r/2-r'/2})$, which means that the correlation between the statistics $\#f$ and~$\#f'$ tends to zero. This is essentially a consequence of the diagonal case, as we show. Starting from Lemma~\ref{Eff} and simplifying as in Corollary~\ref{Ef2},
\begin{align*}
\Exp_w \left[\tilde\# f(w)\,\tilde\#f'(w)\right] \;=&\; \sum_{\mathbf{r} \leq \kpa}\, m_{\mathbf{r}}(f,f') \,\prod_{\TX\in\Sigma} \frac{\binom{n_{\TX}}{r_{\TX}} \binom{n_{\TX}-r_{\TX}}{k_{\TX}-r_{\TX}} \binom{n_{\TX}-k_{\TX}}{k_{\TX}-r_{\TX}}}{\binom{n_{\TX}}{k_{\TX}}^2} 
 \\[0.5em]
\;=&\; \sum_{\mathbf{r} \leq \kpa}\, \Exp_{\mathbf{X}_{\mathbf{r}}}\left[f^{(\mathbf{r})}(\mathbf{X}_{\mathbf{r}}) f'^{(\mathbf{r})}(\mathbf{X}_{\mathbf{r}}) \right]\; \frac{c(\kpa,\mathbf{r},\mathbf{p})+o(1)}{n^{|\mathbf{r}|}}
\end{align*}
for some nonzero constants $c(\kpa,\mathbf{r},\mathbf{p})$. By Lemmas \ref{easy}-\ref{hard} $\rank f = r$, hence $f^{(\mathbf{r})} = 0$ almost everywhere if $|\mathbf{r}|<r$, and these terms drop. This leaves $O(n^{-r})$, which is $o(n^{-(r+r')/2})$ as required. \qed

\begin{rmk*}
It follows from
Definition~\ref{fr} and Corollary~\ref{Ef2} that the computation of the constant $C'_{f,\mathbf{p}}$ appearing in the theorem only involves the evaluation of elementary integrals.
\end{rmk*}

\subsection{Proof of Theorem \ref{main4}}
\label{proof4}

We now turn to the proof of Theorem~\ref{main4}. The main theme is a thorough investigation of the leading second moment terms appearing in the proof of Theorem~\ref{main3}. We describe and study them using a variety of tools from combinatorics, representation theory, and the algebra of words. Parts of this investigation apply to arbitrary finite alphabets. In the special case of two letters, $\{\TA,\TB\}$, we further refine our analysis of the matrices of the leading terms. 

The plan of the proof is as follows. We expand the second moment matrix in the different orders in powers of $n_\TA$ and $n_\TB$. We develop combinatorial expressions to the matrices of each order, and then translate them to operators in the words algebra. By the methods of Theorem~\ref{main3}, the image of the operator of order $n_\TA^{-r_\TA} n_\TB^{-r_\TB}$ is contained in $M^{(k-r,r)}$, for $r=r_\TA+r_\TB$ and $k=k_\TA+k_\TB$, and moreover, the problem reduces to diagonalizing its orthogonal projection to~$S^{(k-r,r)}$. A~unique property of the two-sample case is that the operators that correspond to $(r_\TA,r_\TB)$ and~$(k_\TA,k_\TB)$ only depend on $r_\TA+r_\TB$ up to explicit scalar factors. This proportionality principle is established in Proposition~\ref{prop:Mprime_depends_only_on_diag}. Thus, it is left to treat the case~$(0,r)$, which is done in Proposition~\ref{prop:diagonalizing_M_prime}. It is thanks to this remarkable proportionality of the operators, that Theorem~\ref{main4} holds in greater generality, regardless of how $n_{\TA},n_{\TB} \to \infty$. 

The following lemma summarizes some properties of the frequently used operators $\sh_\TX$, $\del_\TX$ ,$\Theta_{\TX\TY}$. Note that some results concern the special case $\Sigma=\{\TA,\TB\}$.

\begin{lemma}\label{lem:props_of_basic_ops}
$ $
\begin{enumerate}
\item
\label{it:Theta} $\Theta_{ab}^\lambda:M^\lambda\to M^{\lambda+b-a}$ is a morphism of $S_{|\lambda|}-$representations. In particular, it takes any irreducible representation to either an isomorphic irreducible representation, or to $0$. The dual of $\Theta_{ab}^\lambda$ is $\Theta_{ba}^\lambda$. 
\item
In the special case that $\lambda = (k_a,k_b)$ with $k_a\geq k_b,$
\label{it:decomp_of_M_to_Spechts}
\[M^{(k_a,k_b)}\simeq\bigoplus_{j\geq k_a} S^{(j,k-j)}\] and $\Theta_{ba}^{(k_a,k_b)}$ maps each $S^{(j,k-j)}$ isomorphically on a copy of it in $M^{(k_a+1,k_b-1)},$ except for $S^{(k_a,k_b)}$ which maps to $0.$ $\Theta_{ab}^{(k_a,k_b)}$ acts in a dual manner.
\item\label{it:shuffle}
The image of $\sh_a^\lambda:M^\lambda\to M^{\lambda+e_a}$ restricted to $S^\lambda$ is contained in the sum of irreducible $S_{|\lambda|+1}-$representations $S^\mu\hookrightarrow M^{\lambda+e_a}$ where $\mu$ is of the form $\lambda+e_b,~b\leq a.$
\item\label{it:comms}
Denote $k=|\lambda|,\;\del_b^{(k)}=\del_b^\lambda,\;\sh_a^{(k)}=\sh_a^\lambda,\;\del_b^{(k+1)}=\del_b^{\lambda+e_a},\;\sh_a^{(k-1)}=\sh_a^{\lambda-e_b},\;\Theta_{ab}^{(k)}=\Theta_{ab}^\lambda$, $Id$ is the identity map of $M^\lambda$ and $\delta_{a,b}$ is $1$ if $a=b$ and $0$ otherwise. It holds that 
\begin{equation}\label{eq:comm_del_sh}
\del_b^{(k+1)}\circ\sh_a^{(k)}-\sh_a^{(k-1)}\circ\del_b^{(k)}=\Theta_{ba}^{(k)}+\delta_{a,b}(k+1)\mathrm{Id}
\end{equation}
$\Theta_{ab}$ commutes with all $\sh_c,~c\neq a$ and all $\del_c,~c\neq b.$ In the remaining cases we have
\begin{equation}\label{eq:comm_theta_sh}
\Theta^{(k+1)}_{ab}\circ\sh_a^{(k)}-\sh_a^{(k)}\circ\Theta_{ab}^{(k)}=\sh_b^{(k)}.
\end{equation}
\begin{equation}\label{eq:comm_del_theta}
\del_b^{(k+1)}\circ\Theta_{ab}^{(k+1)}-\Theta_{ab}^{(k)}\circ\del_b^{(k+1)}=\del^{(k+1)}_a.
\end{equation}
\begin{equation}\label{eq:comm_theta}
[\Theta_{ab}^{(k)},\Theta_{cd}^{(k)}]=\delta_{a,d}\Theta_{cb}^{(k)}-\delta_{b,c}\Theta_{ad}^{(k)}
\end{equation}
\item\label{it:comm_del^msh}
In the following identities, we forgo in the notation of $\del^\lambda_a,\;\sh^\lambda_a,\;\sh^\lambda_b,\;\Theta^\lambda_{a,b}$ the index $\lambda$ - in all identities multiplication of operators is to be interpreted as composition, the input of the operators is from $M^{(k_a,k_b)}$, and the $\lambda$-indices of each operator is to be picked so that composition is valid. For example,\[\del_a^l=\del_a^{(k_a-l+1,k_b)}\circ\cdots\circ\del_a^{(k_a-1,k_b)}\circ\del_a^{(k_a,k_b)}.\]
The following relations hold
\begin{equation}\label{eq:del1msh2}
\del_a^l\sh_b=\sh_b\del_a^l+l\Theta_{ab}\del_a^{l-1},
\end{equation}
\begin{equation}\label{eq:del1msh1}
\del_a^l\sh_a^{(k_a,k_b)}=\sh_a\del_a^l+l(2k_a+k_b+2-l)\del_a^{l-1}.
\end{equation}
\end{enumerate}
\end{lemma}
\begin{proof}
Item \ref{it:shuffle} is Lemma 46 from \cite{dieker2018spectral}.
Item \ref{it:comms} is partially Lemma 36 of \cite{dieker2018spectral} and partially a simple direct computation. Item \ref{it:comm_del^msh} is a consequence of Item \ref{it:comms}.
\end{proof}

\bigskip
\subsubsection{Power Series Expansion of the Second Moment Matrix \npt} ~

\smallskip
\noindent
We first examine the combinatorial quantities appearing in the coefficients of the second moments expansion of the two random models $\mathcal{W}$ and $\mathcal{W}'$. Specifically, we prove the following relation between the merging coefficients $m_\ell(f,f')$ from Definition~\ref{mcoefs} in the proof of Theorems~\ref{main1}-\ref{main2}, and the coefficients $m_\mathbf{r}(f,f')$ from Definition~\ref{fr} in the proof of Theorem~\ref{main3}.

\begin{prop}
\label{relmm}
Let $f,f' \in W_{\kpa}$ and let $\mathbf{r} = (r_\TA,r_\TB,\dots) \leq \kpa = (k_\TA,k_\TB,\dots)$. Then
$$ m_{\mathbf{r}}(f,f') \;=\; \frac{\prod_{\TX\in\Sigma}\left(r_\TX!(k_\TX-r_\TX)!^2\right)}{\left(2k-r\right)!}\;m_{2k-r}\left(\left(\prod_{\TX\in\Sigma}\frac{\Theta_{\TX\Ti}^{k_\TX-r_\TX}}{(k_\TX-r_\TX) !}\right)f,\,\left(\prod_{\TX\in\Sigma}\frac{\Theta_{\TX\Ti}^{k_\TX-r_\TX}}{(k_\TX-r_\TX)!}\right)f'\right) $$
where as usual $k = |\kpa| = \sum_\TX k_\TX$ and $r = |\mathbf{r}| = \sum_\TX r_\TX$.
\end{prop}
The proof of Proposition \ref{relmm} appears in Appendix \ref{pf:relmm}.

Let $f,f'\in  \tbinom{\Sigma}{\kpa}$.
Now by combining Proposition~\ref{relmm} and Lemma \ref{Eff}, we obtain the following expression for $\Exp_w \left[\# f(w)\,\#f'(w)\right]$
$$ \sum_{\mathbf{r} \leq \kpa}\,\tfrac{\prod_{\TX}(r_\TX!(k_\TX-r_\TX)!^2)}{(2k-r)!} \,  m_{2k-r}\left(\left(\prod_{\TX\in\Sigma}\tfrac{\Theta_{\TX\Ti}^{k_\TX-r_\TX}}{(k_\TX-r_\TX)!}\right)f,\left(\prod_{\TX\in\Sigma}\tfrac{\Theta_{\TX\Ti}^{k_\TX-r_\TX}}{(k_\TX-r_\TX)!}\right)f'\right) \prod_{\TX\in\Sigma} \tbinom{n_{\TX}}{r_{\TX}} \tbinom{n_{\TX}-r_{\TX}}{k_{\TX}-r_{\TX}} \tbinom{n_{\TX}-k_{\TX}}{k_{\TX}-r_{\TX}} $$ 
where as usual we abbreviate $r = |\mathbf{r}|$ and $k = |\kpa|$. As in Definition~\ref{mkr}, we proceed by evaluating this expression for all words in~$\tbinom{\Sigma}{\kpa}$. This yields a square second moment matrix of size $\tbinom{k}{\kpa} = {k!}/{\prod_{\TX}k_\TX!}$ for all words with composition~$\kpa$. Using the duality of $\Theta_{ab}$ and~$\Theta_{ba}$ from Lemma~\ref{lem:props_of_basic_ops}(\ref{it:decomp_of_M_to_Spechts}), we write this second moment matrix as the linear operator
\begin{equation}
\label{eq:2nd moment series}
\sum_{\mathbf{r} \leq \kpa}\, 
\tfrac{\prod_{\TX}(r_\TX!(k_\TX-r_\TX)!^2)}{(2k-r)!}
(\prod_{\TX\in\Sigma}\tfrac{\Theta_{\Ti\TX}^{k_\TX-r_\TX}}{(k_\TX-r_\TX)!}\circ
M_{kr}\circ\prod_{\TX\in\Sigma}\tfrac{\Theta_{\TX\Ti}^{k_\TX-r_\TX}}{(k_\TX-r_\TX)!}) \prod_{\TX\in\Sigma} \tbinom{n_{\TX}}{r_{\TX}} \tbinom{n_{\TX}-r_{\TX}}{k_{\TX}-r_{\TX}} \tbinom{n_{\TX}-k_{\TX}}{k_{\TX}-r_{\TX}}. \end{equation}
This representation motivates the following definition.

\begin{definition}
\label{NKR}
Let $\mathbf{r} = (r_\TA,r_\TB,\dots) \leq \kpa = (k_\TA,k_\TB,\dots)$, and $k = |\kpa|$, $r = |\mathbf{r}|$.
The \emph{$(\kpa,\mathbf{r})$-merging matrix} is the $\tbinom{k}{\kpa} \times \tbinom{k}{\kpa}$ integer valued matrix 
\[
\Np_{\mathbf{r}}^{\kpa}=
\prod_{\TX\in\Sigma}\frac{\Theta_{\Ti\TX}^{k_\TX-r_\TX}}{(k_\TX-r_\TX)!}\circ
M_{kr}\circ\prod_{\TX\in\Sigma}\frac{\Theta_{\TX\Ti}^{k_\TX-r_\TX}}{(k_\TX-r_\TX)!}
\]
\end{definition}

Using this notation, the second moment in~\eqref{eq:2nd moment series} can be asymptotically described as 
\begin{equation*}
\sum_{\mathbf{r} \leq \kpa}\, 
\frac{1}{(2k-r)!}\,
\Np_{\mathbf{r}}^{\kpa}
\prod_{\TX\in\Sigma} {n_{\TX}}^{2k_\TX-r_\TX}\left(1+\sum_{\TX\in\Sigma}O\left(\frac{1}{n_\TX}\right)\right). 
\end{equation*}
Compare this expression with the first expectation in the statement of Theorem \ref{thm:main3}. If we assume $\mathbf{n}/n\rightarrow\mathbf{p}$ as in the theorem, we get
\begin{equation}
\label{eq:asymptotic as func of N}
\begin{aligned}
n^{\rank f}&\Exp_w \left[\left(\tilde\# f(w)\right)^2\right] 
\;=\;
\frac{n^{\rank f}}{\prod_{\TX}\binom{n_{\TX}}{k_{\TX}}^2}
\sum_{\mathbf{r} \leq \kpa} \frac{\left<\Np_{\mathbf{r}}^{\kpa}f,f\right>}{(2k-r)!}  \prod_{\TX\in\Sigma} {n_{\TX}}^{2k_\TX-r_\TX}\left(1+\sum_{\TX\in\Sigma}O\left(\tfrac{1}{n_\TX}\right)\right)
\\[0.5em] \;=&\; 
\left(\prod_{\TX\in\Sigma}(k_{\TX}!)^2\right) n^{\rank f-2k}
\sum_{\mathbf{r} \leq \kpa} \frac{\left<\Np_{\mathbf{r}}^{\kpa}f,f\right>}{(2k-r)!}  \left(\prod_{\TX\in\Sigma} {p_{\TX}}^{2k_\TX-r_\TX}\right) n^{2k-r}
(1+o_n(1))
\\[0.5em] \;=&\;
\sum_{\mathbf{r} \leq \kpa}\, \left<\Np_{\mathbf{r}}^{\kpa}f,f\right>\; (c'(\kpa,\mathbf{r},\mathbf{p}) + o_n(1)) n^{\rank f -r}
\end{aligned}
\end{equation}
where $c'(\kpa,\mathbf{r},\mathbf{p})$ is some positive constant depending on $\kpa$, $\mathbf{r}$, and $\mathbf{p}$. Theorem \ref{thm:main3} says that this expression approaches a positive value and that $f\in \bigoplus_{r\geq\rank f} W_{\kpa r} \setminus \bigoplus_{r>\rank f} W_{\kpa r}$. Clearly, this term approaches a finite nonnegative value if and only if $\left<\Np_{\mathbf{r}}^{\kpa}f,f\right>=0$ for all $\mathbf{r}\leq\kpa$ such that $|\mathbf{r}|<\rank f$. Moreover, it approaches the value zero if $\left<\Np_{\mathbf{r}}^{\kpa}f,f\right>=0$ also for all $\mathbf{r}\leq\kpa$ such that $|\mathbf{r}|=\rank f$. Note that by 
Lemma \ref{Akr_identities},
$$\left<\Np_{\mathbf{r}}^{\kpa}f,f'\right> \;=\; 
\left< \mathcal{D}_{k-r} \prod_{\TX\in\Sigma}\frac{\Theta_{\TX\Ti}^{k_\TX-r_\TX}}{(k_\TX-r_\TX)!} f,\;
\mathcal{D}_{k-r} \prod_{\TX\in\Sigma}\frac{\Theta_{\TX\Ti}^{k_\TX-r_\TX}}{(k_\TX-r_\TX)!} f' \right>,$$
and therefore $\left<\Np_{\mathbf{r}}^{\kpa}f,f\right>=0$ if and only if $$ f\;\in\;\ker\Np_{\mathbf{r}}^{\kpa} \;=\; \ker \prod_{\TX\in\Sigma}\frac{\Theta_{\TX\Ti}^{k_\TX-r_\TX}}{(k_\TX-r_\TX)!} \;=\; \ker \prod_{\TX\in\Sigma}\Theta_{\TX\Ti}^{k_\TX-r_\TX},$$
where it is an equality because $M_{kr}$ is a nondegenerate bilinear form.
Therefore, by the argument above,
\begin{equation}
\label{eq:Wkr_representation}
W_{\kpa r} = \left(\bigcap_{\substack{\mathbf{r}\leq\kpa\\|\mathbf{r}|=r}} \ker \prod_{\TX\in\Sigma}\Theta_{\TX\Ti}^{k_\TX-r_\TX} \right)^{\displaystyle\perp} \bigcap\;\,
\left(\bigcap_{\substack{\mathbf{r}\leq\kpa\\|\mathbf{r}|<r}} \ker \prod_{\TX\in\Sigma}\Theta_{\TX\Ti}^{k_\TX-r_\TX}\right)
\end{equation}

Now, returning to \eqref{eq:asymptotic as func of N}, we can see that
\[
n^{\rank f}\Exp_w \left[\left(\tilde\# f(w)\right)^2\right] = \sum_{\substack{\mathbf{r} \leq \kpa\\|\mathbf{r}|=\rank f}}\,c'(\kpa,\mathbf{r},\mathbf{p}) \left<\Np_{\mathbf{r}}^{\kpa}f,f\right> +o_n(1).
\]
Note that while $f\in \bigoplus_{r\geq\rank f} W_{\kpa,r} \setminus \bigoplus_{r>\rank f} W_{\kpa,r}$, for any $e\in \bigoplus_{r>\rank f} W_{\kpa,r}$ it holds that $\left<\Np_{\mathbf{r}}^{\kpa}f,f\right> = \left<\Np_{\mathbf{r}}^{\kpa}(f+e),f+e\right>.$ Therefore, if we take $\proj = \proj_{\rank f}$ from Section \ref{moreoperators} to be the projection to $W_{\kpa,\rank f}$ by $\proj$, we get that 
\begin{align*}
n^{\rank f}\Exp_w \left[\left(\tilde\# f(w)\right)^2\right]
=& \sum_{\substack{\mathbf{r} \leq \kpa\\|\mathbf{r}|=\rank f}}\,c'(\kpa,\mathbf{r},\mathbf{p}) \left<\Np_{\mathbf{r}}^{\kpa}\circ\proj f,\proj f\right> +o_n(1).
\\=& \sum_{\substack{\mathbf{r} \leq \kpa\\|\mathbf{r}|=\rank f}}\,c'(\kpa,\mathbf{r},\mathbf{p}) \left<(\proj\circ\Np_{\mathbf{r}}^{\kpa}\circ\proj) f, f\right> +o_n(1).
\end{align*}
Denote the matrix $\proj\circ\Np_{\mathbf{r}}^{\kpa}\circ\proj$ by $\Mp_{\mathbf{r}}^{\kpa}$.

By the same arguments that we used, it can be seen that for $f,f'\in W_{\kpa}$ with $\rank f = \rank f' = r$ it holds that 
$$
n^r\Exp_w \left[\left(\tilde\# f(w)\right)\left(\tilde\# f'(w)\right)\right] = 
\sum_{\substack{\mathbf{r} \leq \kpa\\|\mathbf{r}|=\rank f}}\,c'(\kpa,\mathbf{r},\mathbf{p}) \left<\Mp_{\mathbf{r}}^{\kpa} f, f'\right> +o_n(1)
$$
and
\begin{equation}
\label{eq:asymptotic as func of n_x}
\begin{aligned}
\Exp_w \left[\left(\tilde\# f(w)\right)\left(\tilde\# f'(w)\right)\right] &=
\sum_{\mathbf{r} \leq \kpa} \frac{ \prod_{\TX\in\Sigma}(k_{\TX}!)^2 }{(2|\kpa|-|\mathbf{r}|)!} \left<\Np_{\mathbf{r}}^{\kpa}f,f'\right> \prod_{\TX\in\Sigma} {n_{\TX}}^{-r_\TX}(1+\sum_{\TX\in\Sigma}O(\tfrac{1}{n_\TX})) \\&=
\sum_{\substack{\mathbf{r} \leq \kpa \\|\mathbf{r}|\ge\rank f}} \frac{ \prod_{\TX\in\Sigma}(k_{\TX}!)^2 }{(2|\kpa|-|\mathbf{r}|)!} \left<\Np_{\mathbf{r}}^{\kpa}f,f'\right> \prod_{\TX\in\Sigma} {n_{\TX}}^{-r_\TX}(1+\sum_{\TX\in\Sigma}O(\tfrac{1}{n_\TX})).
\end{aligned}
\end{equation}

\bigskip
\subsubsection{The Two Sample Case \npt} ~

\smallskip
\noindent
We make some preparations to the next steps of the proof of Theorem~\ref{main4}. 
For the remainder of the proof of Theorem \ref{main4}, we let the alphabet be $\Sigma=\{\TA,\TB\}$, so that $k=k_\TA+k_\TB$ and $r=r_\TA+r_\TB$ unless otherwise indicated. Occasionally, when another letter is needed, such as $\Ti$ in Proposition~\ref{relmm}, we use $\Sigma=\{\TA,\TB,\Ti\}$ and it will be clear from the context.

The next observation is an immediate consequence of Definition \ref{Wkkr}.
\begin{obs} For any $\mathbf{r}=(r_\TA,r_\TB)\leq\kpa=(k_\TA,k_\TB)$, the submodule $W_{(k_a,k_b),r}$ of $W_{(k_a,k_b)}$ is precisely the isomorphic copy of $S^{(k-r,r)}$ inside $W_{(k_a,k_b)} = M^{(k_a,k_b)}$. By abuse of notation we write this as
$W_{(k_a,k_b),r} = S^{(k-r,r)} \hookrightarrow M^{(k_a,k_b)}$.
\end{obs}

For the special case $\kpa=(k_a,k_b)$, we get from \eqref{eq:asymptotic as func of n_x}:
\begin{align*}
\Exp_w \left[\tilde\# f(w)\,\tilde\# f'(w)\right] =
\sum_r \sum_{r_a+r_b=r} \tfrac{(k_\TA!)^2(k_\TB!)^2}{(2k-r)!} \left<\Np_{(r_a,r_b)}^{(k_a,k_b)}f,f'\right> n_a^{-r_a} n_b^{-r_b} (1+O(n_a^{-1}+n_b^{-1})).
\end{align*}
Denote by $\proj_r$ the projection to $W_{\kpa r}\simeq S^{(k-r,r)}$, and for each $r$ let $f=\proj_r f +e_r$ and $f'=\proj_r f' +e_r'$. Then
$$\left<\Np_{(r_a,r_b)}^{(k_a,k_b)}f,f'\right> = \left<\Mp_{(r_a,r_b)}^{(k_a,k_b)}f,f'\right> + \left<\Np_{(r_a,r_b)}^{(k_a,k_b)}e_r,\proj_r f'\right> + \left<\proj_r f,\Np_{(r_a,r_b)}^{(k_a,k_b)}e_r'\right> +\left<\Np_{(r_a,r_b)}^{(k_a,k_b)}e_r,e_r'\right>.$$
If $r<\rank f$ or $r<\rank f'$ then $\left<\Np_{(r_a,r_b)}^{(k_a,k_b)}f,f'\right>=0$. If $r=\rank f = \rank f'$ then $\left<\Np_{(r_a,r_b)}^{(k_a,k_b)}f,f'\right> = \left<\Mp_{(r_a,r_b)}^{(k_a,k_b)}f,f'\right>$. So assuming $\rank f = \rank f'$,
\begin{equation}
\label{eq:Wkr_cov_matrix}
\Exp_w \left[\tilde\# f(w)\,\tilde\# f'(w)\right] =
\frac{(k_\TA!)^2(k_\TB!)^2}{(2k-r)!} \sum_{r_a+r_b=\rank f} \frac{\left<\Mp_{(r_a,r_b)}^{(k_a,k_b)}f,f'\right> +O\left(\tfrac{1}{n_a}+\tfrac{1}{n_b}\right)}{n_a^{r_a} n_b^{r_b}}
\end{equation}
If $\rank f \neq \rank f'$, without loss of generality let $\rank f > \frac{\rank f + \rank f'}{2} > \rank f'$. Then 
\begin{equation}
\label{eq:multisample different r case}
\begin{aligned}
\Exp_w &\left[\tilde\# f(w)\,\tilde\# f'(w)\right] =
\sum_{r\geq\rank f}  \tfrac{(k_\TA!)^2(k_\TB!)^2}{(2k-r)!} \sum_{r_a+r_b=r} \frac{\left<\Np_{(r_a,r_b)}^{(k_a,k_b)}f,f'\right> +O\left(\tfrac{1}{n_a}+\tfrac{1}{n_b}\right)}{n_a^{r_a} n_b^{r_b}} \\ =&
\sum_{r_a+r_b=\rank f} O\left(\tfrac{1}{n_a^{r_a} n_b^{r_b}}\right)
\;=\; O\left(\left(\tfrac{1}{n_a}+\tfrac{1}{n_b}\right)^{\rank f}\right) \;=\; o\left(\left(\tfrac{1}{n_a}+\tfrac{1}{n_b}\right)^{\frac{\rank f + \rank f'}{2}}\right)
\end{aligned}
\end{equation}

\bigskip
\subsubsection{Spectral Decomposition of $\Mp_{k_b}^{(k_a,k_b)}$ \npt} ~

\smallskip
\noindent
Let us now recall a few results from \cite{dieker2018spectral}, stated with our notations. We fix $k_a\geq k_b$ and their sum $k.$
Write 
$$\proj=\proj^{(k_a,k_b)}:M^{(k_a,k_b)}\;\to\; S^{(k_a,k_b)}$$ 
the projection on the Specht module.
Denote by $\RTR^{(k_a,k_b)}$ the map
$$ (\sh_b\del_b+\sh_a\del_a)\circ\proj:M^{(k_a,k_b)}\;\to\; S^{(k_a,k_b)}\;\hookrightarrow\; M^{(k_a,k_b)}$$ 
Its kernel contains $(S^{(k_a,k_b)})^\perp$ by definition. Its image is in $S^{(k_a,k_b)}\hookrightarrow M^{(k_a,k_b)},$ since the map $\sh_b\del_b+\sh_a\del_a\in End(M^{(k_a,k_b)})$ can be written as the action of an element from the group algebra of $S_k$ acting on the module $M^{(k_a,k_b)}$ (\cite[Definition 33]{dieker2018spectral}; the equivalence of definitions is Proposition 35 there), hence it respects the decomposition of $M^{(k_a,k_b)}$ to $S_{k}$ submodules, such as the Specht modules.

Returning to our problem, in the two-sample case, $\Mp_{k_b}^{(k_a,k_b)}=\Mp_{(0,k_b)}^{(k_a,k_b)}$ has the form 
\begin{equation}\label{eq:Mp_2_letters}
\Mp_{k_b}^{(k_a,k_b)}=\proj^{(k_a,k_b)}\circ \A_{k,k_b}\circ \proj^{(k_a,k_b)}=\proj^{(k_a,k_b)}\circ \sum_{m=0}^{k_a}\frac{\sh_a^m\del_a^m}{(m!)^2}\circ \proj^{(k_a,k_b)}.
\end{equation}

As we will soon see in Proposition \ref{prop:diagonalizing_M_prime}, the following operator will play an important role in our analysis: \[\shp_b^{(k_a,k_b)}(v) \;=\; \sh_b^{(k_a,k_b)}(v)+\frac{1}{k_b-k_a-1}\Theta_{ab}^{(k_a+1,k_b)}(\sh_a^{(k_a,k_b)}(v)).\]

\begin{prop}\label{prop:ops_from_Sal}
The linear operator $\sh_a^{(k_a,k_b)}$ maps $S^{(k_a,k_b)}\hookrightarrow M^{(k_a,k_b)}$ to $S^{(k_a+1,k_b)}\hookrightarrow M^{(k_a+1,k_b)}.$
The operator $\shp_2^{(k_a,k_b)}$ maps $S^{(k_a,k_b)}\hookrightarrow M^{(k_a,k_b)}$ to $S^{(k_a,k_b+1)}\hookrightarrow M^{(k_a,k_b+1)},$ and moreover $\shp_b^{(k_a,k_b)}=\proj^{(k_a,k_b+1)}\circ\sh_b^{(k_a,k_b)}.$ Additionally,
\begin{equation}
\label{eq:sh1sh'2_comm}
\sh^{(k_a,k_b+1)}_a\shp^{(k_a,k_b)}_b=\frac{k_a+2-k_b}{k_a+1-k_b}\shp^{(k_a+1,k_b)}_b\sh^{(k_a,k_b)}_a.
\end{equation}
\end{prop}
\begin{proof}
The first statement is a consequence of \cite[Proposition 15]{dieker2018spectral}. The second statement is a consequence of \cite[Proposition 19]{dieker2018spectral}. From Lemma~\ref{lem:props_of_basic_ops}(\ref{it:Theta}) $\frac{1}{k_b-k_a-1}\Theta_{ab}(\sh_a(v))$ maps $S^{(k_a,k_b)}$ to
\[\bigoplus_{j> k_a} S^{(j,k+1-j)}\hookrightarrow M^{(k_a,k_b+1)}.\] Since $\shp_b$ maps $S^{(k_a,k_b)}$ to $S^{(k_a,k_b+1)},$ it must be the orthogonal projection on the latter.
Finally, \eqref{eq:sh1sh'2_comm} is a direct consequence of \eqref{eq:comm_theta_sh} and the fact that different shuffle operators commute. 
\end{proof}

\begin{rmk*}
For the sake of brevity, we use the following abuse of notation in the following pages.
When using the following operators of the $S_k$-module $M^\lambda$ 
$$\Mp_{k_b}^{(k_a,k_b)},\; \shp_b^{(k_a,k_b)},\;\proj^{(k_a,k_b)},\;\Theta_{ab}^{k},\; \sh_b^k \text{ or }\del_b^{k+1}$$ 
we avoid writing the indices that denote the space from which the input of the operator comes from when it is clearly implied by the context. For example, instead of writing $\Mp_{k_b+1}^{\scriptscriptstyle (k_a,k_b+1)}\circ\shp_b^{\scriptscriptstyle (k_a,k_b)}$ we may write $\Mp\circ\shp_b^{\scriptscriptstyle (k_a,k_b)}$, since no other $\Mp$ is valid.
\end{rmk*}

Theorem 26 \cite{dieker2018spectral}, specialized to partitions with two parts $(k_a,k_b),$ says the following.

\begin{alsotheorem}
\label{thm:DS}
Let $K^{(m_a,m_b)}\subseteq S^{(m_a,m_b)}\hookrightarrow M^{(m_a,m_b)}$ be the kernel of $\RTR^{(m_a,m_b)}.$
Then when {$k_b\neq 0,$} for each $0\leq i\leq k_a-k_b,~0\leq j\leq k_b-1,$
\[\shp_b^j(\sh_a^i(K^{(k_a-i,k_b-j)}))\] is an eigenspace for $\RTR^{(k_a,k_b)},$ of dimension $\binom{k_a+k_b-i-j-2}{k_a-i-1}-\binom{k_a+k_b-i-j-2}{k_a-i}.$ Moreover, $S^{(k_a,k_b)}$ decomposes as a direct sum of these subspaces. 
When $k_b=0$, the only eigenspace of $\RTR^{(k_a)}$ is the one-dimensional space of vectors with constant entries, which can be regarded as $\sh_a^{k_a}(K^{(0,0)}).$
\end{alsotheorem}

\begin{rmk*}
In~\cite{dieker2018spectral}, the summands are indexed by partitions $(m_a,m_b)=(k_a-i,k_b-j)$ such that the relative Young diagram $(k_a,k_b)/(m_a,m_b)$ is a \emph{horizontal strip} and the dimension of the corresponding space is the number of \emph{desarrangement tableaux} for the diagram $(i,j)$ \cite[\S3.1 for definitions]{dieker2018spectral}.  
The condition that $(m_a,m_b)$ is a horizontal strip bounds $i$ by $k_a-k_b.$ Supposing that the dimension of the corresponding space is nonzero is equivalent to requiring that $j\neq k_b,$ whenever $k_b>0$ or that $i=k_a$ when $k_b=0,$ as we shall now elaborate; moreover, the dimension of the eigenspace is exactly as stated in the theorem above. 
The dimension when $k_b=0$ is $1,$ while for $k_b>0,$ for standard Young tableaux with two rows, the desarrangement condition amounts to requiring that the $(2,1)$-box is filled with $2,$ and a simple calculation shows that the number of such standard tableaux for a fixed Young diagram $(m_a,m_b)$ is precisely $\binom{m_a+m_b-2}{m_a-1}-\binom{m_a+m_b-1}{m_a}$.
In fact, $\sh_a$ takes the $(i,j)$ eigenvalue for $(k_a,k_b)$ to the $(i+1,j)-$th one for $(k_a+1,k_b).$ $\shp_b$ takes the $(i,j)$ eigenvalue for $(k_a,k_b)$ to the $(i,j+1)-$th one for $(k_a,k_b+1),$ if $j<k_b,$ and to $0$ otherwise.
\end{rmk*}

\begin{prop}\label{prop:diagonalizing_M_prime}
The eigenspaces for $\Mp_{k_b}^{(k_a,k_b)}$ are 
\[\shp_b^j\;\sh_a^i\;K^{(k_a-i,k_b-j)}\] 
for each $0\leq i\leq k_a-k_b,~0\leq j\leq k_b-1$, or for $(i,j)=(k_a,0)$ if $k_b=0$. \\
The eigenvalue for the $(i,j)$ eigenspace is
\[\frac{(2k_a+k_b)!(k_a-k_b+1)!}{i!(2k_a+k_b-i-j)!(k_a-k_b+1+j)!}.\] 
The dimension of the $(i,j)$ eigenspace is \[\binom{k_a+k_b-i-j-2}{k_a-i-1}-\binom{k_a+k_b-i-j-2}{k_a-i}\]
unless $k_b=0$ and $(i,j)=(k_a,0)$ where the dimension is~1.
\end{prop}

In \cite{dieker2018spectral} the commutators between the operators $\RTR$ and $\sh_a,\shp_b$ are scalar operators. This fact, combined with the summation of the dimensions of the subspaces and the characterization of the kernel of $\RTR$ obtained in \cite{reiner2014spectra}, yields the decomposition of Theorem~\ref{thm:DS}.
Here we will use the existence of this decomposition, but our operators $\Mp$ and their commutation relations with $\sh_a$ and $\shp_b$ will be substantially more complicated.
Still, we will be able to characterize the { $(i,0),~i\leq k_a-k_b$} eigenspaces, as well as a tricky relation involving $\shp_b$ and $\Mp,$ in order to obtain the proposition.
\begin{proof}
The proof is an immediate consequence of Theorem~\ref{thm:DS}, and the following two lemmas.
\begin{lemma}\label{lem:0j}
The spaces $\sh_a^i(K^{(k_a-i,k_b)}),~0\leq i\leq k_a-k_b$ are eigenspaces of $\Mp^{(k_a,k_b)}$ for the eigenvalues $\binom{2k_a+k_b}{i}$ respectively.
\end{lemma}
\begin{lemma}\label{lem:ij}
\begin{equation}\label{eq:shp2_relation}\Mp\circ\shp_b^{(k_a,k_b)}\;=\; \frac{2k_a+k_b+1}{k_a-k_b+1}\;\shp_b\circ\Mp^{(k_a,k_b)}
\end{equation}
\end{lemma}
Indeed, Lemma \ref{lem:0j} shows that the $(i,0)$ eigenspaces in Proposition \ref{prop:diagonalizing_M_prime} are correct, with the correct eigenvalues. 
The application of $\shp_b$ has the effect of increasing $k_b$ to $k_b+1,$ and increasing $j$ to $j+1.$
Lemma \ref{lem:ij} then shows that $\Mp^{(k_a,k_b+1)}_{k_b+1}$ has an eigenspace, naturally indexed by $(i,j+1)$ obtained by applying $\shp_b$ to the $(i.j)$ eigenspace of $\Mp^{(k_a,k_b)}_{k_b},$ and the eigenvalue corresponding to the former space is
$\frac{2k_a+k_b+1}{k_a-k_b+1}$ times the eigenvalue for the later eigenspace. 
Assuming, inductively, that the proposition holds for smaller $k=k_a+k_b,$ then the
$(i,j+1)$ eigenvalue of $\Mp^{(k_a,k_b+1)}_{k_b+1}$ equals
\[\tfrac{2k_a+k_b+1}{k_a-k_b+1}\,\cdot\,\tfrac{(2k_a+k_b)!(k_a-k_b+1)!}{i!(2k_a+k_b-i-j)!(k_a-k_b+1+j)!}
\;=\;\tfrac{(2k_a+k_b+1)!(k_a-k_b)!}{i!(2k_a+k_b-i-j)!(k_a-k_b+1+j)!}\] as claimed.
Since by Theorem~\ref{thm:DS} \[S^{(k_a,k_b)}=\bigoplus_{\substack{0\leq i\leq k_a-k_b\\
0\leq j\leq k_b-1}}\shp_b^j(\sh_a^i(K^{(k_a-i,k_b-j)})),\]
this gives a complete decomposition.

\begin{proof}[Proof of Lemma \ref{lem:0j}]
We first observe that the kernel of the operator $\RTR^{(m_a,m_b)}$ is exactly the kernel of $\del_a^{(m_a,m_b)}$ restricted to $S^{(m_a,m_b)}$.
Indeed, for $v\in S^{(m_a,m_b)}$
\[v\in \ker(\RTR^{(m_a,m_b)})\;\Leftrightarrow\; \langle v,\RTR(v)\rangle=0 \;\Leftrightarrow\;
\langle v,(\sh_a\del_a+\sh_b\del_b)(v)\rangle=0
\;\Leftrightarrow\; \]
\[\;\Leftrightarrow\;\langle \del_a v,\del_a v\rangle+\langle \del_b v,\del_b v\rangle=0
\;\Leftrightarrow\;
\langle \del_a v,\del_a v\rangle=0,~\langle \del_b v,\del_b v\rangle=0\]
where $\langle-,-\rangle$ is the canonical positive definite bilinear pairing and we have used the duality of $\sh$ and $\del$, and the fact that $\langle u,u\rangle\geq 0.$
Thus, for $v$ as above, if $\RTR(v)=0$ then also $\del_a(v)=0.$ For the opposite direction, assume $\del_a(v)=0.$ Note that from Lemma~\ref{lem:props_of_basic_ops}(\ref{it:comms})
\[\del_b=\del_a\circ \Theta_{ba}-\Theta_{ba}\circ\del_a.\] Now, from Lemma \ref{lem:props_of_basic_ops}(\ref{it:Theta}),(\ref{it:decomp_of_M_to_Spechts}), for $v\in S^{(k_a,k_b)},~\Theta_{ba}(v)=0.$ Thus, \[\del_b(v)=\Theta_{ba}(\del_a(v))=0.\]

Returning to the proof, suppose that $v\in S^{(k_a-i,k_b)}$ satisfies $\del_a(v)=0.$ We want to show that $\Mp^{(k_a,k_b)}(\sh_a^i(v))=\binom{2k_a+k_b}{i}\sh_a^i(v).$
From Proposition \ref{prop:ops_from_Sal}, $\sh_a^i(v)\in S^{{(k_a,k_b)}}$. Hence
\begin{equation}
\label{eq:M_k2,0_of_shuff}
\Mp^{(k_a,k_b)}\, \sh_a^i \,v = \proj\, \A_{k,k_b} \,\proj\, \sh_a^i \,v = \proj\,  \A_{k,k_b} \,\sh_a^i\, v = \proj \sum_{m=0}^{k_a}\frac{\sh_a^m\del_a^m}{(m!)^2}\sh_a^i \,v
\end{equation}
\begin{obs}
\label{obs:delm_shi}
For $v\in \ker\del_a$, if $m>i$ then $\del_a^m\sh_a^i(v)$ vanishes, and otherwise it equals
$$ \frac{(2k_a+k_b-i)!}{(2k_a+k_b-m-i)!}\;\frac{i!}{(i-m)!}\;\sh_a^{i-m}\,v$$
\end{obs}
The proof of Observation \ref{obs:delm_shi} appears in Appendix \ref{pf:delm_shi}. Using Observation \ref{obs:delm_shi} in \eqref{eq:M_k2,0_of_shuff}, we obtain that 
\[\A_{k,k_b}\sh_a^i \;=\; \sum_{m=0}^i\binom{2k_a+k_b-i}{m}\binom{i}{i-m}\,\sh_a^i\;=\;\binom{2k_a+k_b}{i}\,\sh_a^i\]
where the last equality is the Vandermonde identity. Since $\proj\circ\sh_a^i(v)=\sh_a^i(v),$ as we saw right before \eqref{eq:M_k2,0_of_shuff}, the lemma follows.
\end{proof}
\begin{proof}[Proof of Lemma \ref{lem:ij}]
Using Proposition \ref{prop:ops_from_Sal}, $\shp_b=\proj\circ\shp_b,$
\begin{equation}
\label{eq:comm_MP_SHP}
\begin{aligned}
\Mp\circ\shp_b^{(k_a,k_b)}&=\proj\circ\A_{k+1,k_b+1}\circ\proj\circ\shp_b=\proj\circ\A_{k+1,k_b+1}\circ\shp_b
\\
&=\proj\circ\A_{k+1,k_b+1}\circ(\sh_b+\frac{1}{k_b-k_a-1}\Theta_{ab}\sh_a)\\
&=
\proj\left(\sum_{m=0}^{k_a}\frac{\sh_a^m\del_a^m}{(m!)^2}\right)(\sh_b+\frac{1}{k_b-k_a-1}\Theta_{ab}\sh_a).\end{aligned}
\end{equation}
We now show
\begin{equation}\label{eq:1_for_ij}
\sh_a^m\del_a^m\sh_b=\sh_b\sh_a^m\del_a^m+m\Theta_{ab}\sh_a^m\del_a^{m-1}-m^2\sh_b\sh_a^{m-1}\del_a^{m-1}
\end{equation}
and
\begin{align}\label{eq:2_for_ij}
&\sh_a^m\del_a^m\Theta_{ab}\sh_a=\Theta_{ab}\sh_a\left(\sh_a^m\del_a^m+m(2k_a+k_b+2-m)\sh_a^{m-1}\del_a^{m-1}\right)-\\\notag&\quad\quad
-m\sh_b\left(\sh_a^m\del_a^m+m(2k_a+k_b+2-m)\sh_a^{m-1}\del_a^{m-1}\right).
\end{align}
For \eqref{eq:1_for_ij}, first observe that, using \eqref{eq:del1msh2}
\[\del_a^m\sh_b=\sh_b\del_a^m+m\Theta_{ab}\del_a^{m-1}\] 
Using \eqref{eq:comm_theta_sh} we have
\begin{equation}
\label{eq:coom_sh^m_theta}
\sh_a^m\Theta_{ab}=\Theta_{ab}\sh_a^m-m\sh_b\sh_a^{m-1}
\end{equation}
\eqref{eq:1_for_ij} is a direct consequence of these two equations:
\[\sh_a^m\del_a^m\sh_b=\sh_a^m(\sh_b\del_a^m+m\Theta_{ab}\del_a^{m-1})=\sh_b\sh_a^m\del_a^m+m(\Theta_{ab}\sh_a^m-m\sh_b\sh_a^{m-1})\del_a^{m-1}.\]
Similarly, using \eqref{eq:del1msh1}, 
\[\del_a^m\Theta_{ab}\sh_a=\Theta_{ab}\del_a^m\sh_a=\Theta_{ab}(\sh_a\del_a^m+m(2k_a+k_b+2-m)\del_a^{m-1}).\]
This, together with \eqref{eq:coom_sh^m_theta} again gives \eqref{eq:2_for_ij}:

\begin{align*}
\sh_a^m\del_a^m\Theta_{ab}\sh_a \;=\; &\sh_a^m\Theta_{ab}(\sh_a\del_a^m+m(2k_a+k_b+2-m)\del_a^{m-1})\\
\;=\;&
\Theta_{ab}\sh_a\left(\sh_a^m\del_a^m+m(2k_a+k_b+2-m)\sh_a^{m-1}\del_a^{m-1}\right)\\
&-m\sh_b\left(\sh_a^m\del_a^m+m(2k_a+k_b+2-m)\sh_a^{m-1}\del_a^{m-1}\right).\end{align*}

Using \eqref{eq:1_for_ij},~\eqref{eq:2_for_ij} we have
\begin{equation}\label{eq:1'_for_ij}
\A_{k+1,k_b+1}\sh_b=\sh_b\A_{k,k_b}+ \Theta_{ab}\sh_a\sum_{m=1}^{k_a}\frac{\sh_a^{m-1}\del_a^{m-1}}{m!(m-1)!}-\sh_b\sum_{m=0}^{k_a-1}\frac{\sh_a^{m}\del_a^{m}}{(m!)^2}.
\end{equation}
and
\begin{align}\label{eq:2'_for_ij}
&\A_{k+1,k_b+1}\Theta_{ab}\sh_a=\Theta_{ab}\sh_a\left(\A_{k,k_b}+\sum_{m=1}^{k_a}\frac{(2k_a+k_b+2-m)\sh_a^{m-1}\del_a^{m-1}}{m!(m-1)!}\right)\\
&\notag\quad\quad-\sh_b\left(\sum_{m=1}^{k_a}\frac{\sh_a^m\del_a^m}{m!(m-1)!}+\sum_{m=1}^{k_a}\frac{(2k_a+k_b+2-m)\sh_a^{m-1}\del_a^{m-1}}{((m-1)!)^2}\right).
\end{align}
From Lemma \ref{lem:props_of_basic_ops}(\ref{it:Theta}) the image of $\Theta_{ab}^{(k_a,k_b)}$ is in the kernel of $\proj^{(k_a,k_b+1)}.$ Using this fact together with \eqref{eq:comm_MP_SHP}, and summing \eqref{eq:1'_for_ij} and \eqref{eq:2'_for_ij}, we obtain
\[\proj\circ\A_{k+1,k_b+1}\circ\shp_b=\]\[
\proj\circ\sh_b\left(\A_{k,k_b}-\sum_{m=0}^{k_a-1}\tfrac{\sh_a^m\del_a^m}{(m!)^2}+\tfrac{1}{k_a+1-k_b}\left(
\sum_{m=1}^{k_a}\tfrac{\sh_a^m\del_a^m}{m!(m-1)!}+\sum_{m=0}^{k_a-1}\tfrac{(2k_a+k_b+1-m)\sh_a^m\del_a^m}{(m!)^2}
\right)\right).\]
Simplifying, we get using Lemma \ref{Akr_identities},\ref{Akr_identity} that \[
\proj\circ\A_{k+1,k_b+1}\circ\shp_b 
=\frac{2k_a+k_b+1}{k_a+1-k_b}\proj\circ\sh_b\circ\left(\A_{k,k_b}-\frac{2k_b}{2k_a+k_b+1}\frac{\sh_a^{k_a}\del_a^{k_a}}{(k_a!)^2}\right).
\]
We claim that
\[\frac{2k_b}{2k_a+k_b+1}\frac{\sh_a^{k_a}\del_a^{k_a}}{(k_a!)^2}=0.\] This is obvious when $k_b=0.$ When $k_b>0,~\del_a^{k_a}(v)=0$ for $v\in S^{(k_a,k_b)},$ since $\del_a^{k_a}(v)\in M^{(0,k_b)}$ is a vector proportional to the constant vector, and the proportionality constant is a multiple of the sum of coordinates of $v.$ This sum is $0$ since $v$ is orthogonal to $S^{(k_a+k_b,0)}\hookrightarrow M^{(k_a,k_b)},$ which is a nonzero constant vector (since we can get it by taking a word of $k_a+k_b$ $a$-s and applying $k_b$ times $\Theta_{ab}$).
Thus, 
\[\proj\circ\A_{k+1,k_b+1}\circ\shp_b =
\frac{2k_a+k_b+1}{k_a+1-k_b}\proj\circ\sh_b\circ\A_{k,k_b}.\]
We also must show that
\[\proj\circ\sh_b=\proj\circ\sh_b\circ\proj.\]
The domain of both maps is \[M^{(k_a,k_b)}\simeq\bigoplus_{i\geq k_a} S^{(i,k-i)}.\] Now, \[\proj\circ\sh_b\circ\proj(M^{(k_a,k_b)})=
\proj(\sh_b(S^{(k_a,k_b)}))\subseteq S^{(k_a,k_b+1)}\hookrightarrow M^{(k_a,k_b+1)},\]
by Proposition \ref{prop:ops_from_Sal}. We have to show that $\proj\circ\sh_b$ restricts to $0$ on $\bigoplus_{i> k_a} S^{i,k-i}\subset M^{(k_a,k_b)},$ or, in other words, that $\sh_b$ maps $S^{(i,k-i)},$ to $\bigoplus_{i> k_a} S^{i,k+1-i}\subset M^{(k_a,k_b+1)}$ for $i>k_a$. From Lemma \ref{lem:props_of_basic_ops}(\ref{it:Theta}) $S^{(i,k-i)}\hookrightarrow M^{(k_a,k_b)}$ is the image of $S^{(i,k-i)}\hookrightarrow M^{(i,k-i)}$ under $\Theta_{ab}^{i-k_1}.$ By Item \ref{it:comms} of the same lemma, $\sh_b\circ\Theta_{ab}^{i-k_1}=\Theta_{ab}^{i-k_1}\circ \sh_b.$ Using Lemma \ref{lem:props_of_basic_ops}(\ref{it:shuffle}) $\sh_b(S^{i,k-i})$ is contained in $S^{(i,k-i+1)}\oplus S^{(i+1,k-i)}\subseteq M^{(i,k-i+1)}.$ Since $\Theta_{ab}$ is a module morphism, $\Theta_{ab}^{i-k_1}\circ \sh_b(S^{(i,k-i)})\subseteq S^{(i,k-i+1)}\oplus S^{(i+1,k-i)}\hookrightarrow M^{(k_a,k_b+1)}.$ Thus, for $i>k_a,~\sh_b(S^{i,k-i})$ does not intersect $S^{(k_a,k_b+1)}\hookrightarrow M^{(k_a,k_b+1)}.$ The lemma follows.
\end{proof}
Proposition~\ref{prop:diagonalizing_M_prime} is now proven.
\end{proof}

\bigskip
\subsubsection{Spectral Decomposition of $\Mp_{(r_a,r_b)}^{(k_a,k_b)}$ \npt} ~

\smallskip
\noindent
In the previous section we found the spectral decomposition of $\Mp_{(0,k_b)}^{(k_a,k_b)}$, and in this section we will show how that decomposition can be used to decompose other matrices $\Mp_{(r_a,r_b)}^{(k_a,k_b)}$, by showing that every matrix of the latter type is proportional to a matrix of the former type, up to conjugation by $\Theta$ operations.

\begin{lemma}\label{lem:Thetas_relation}
On the Specht module $S^{(k_a,k_b)}\hookrightarrow M^{(k_a,k_b)}$ the following identity holds for any $i,j$
\begin{equation}\label{eq:Thetas_relation}
    \frac{\Theta_{bc}^{i+j}}{(i+j)!}\frac{\Theta_{ac}^{k_a-i-j}}{(k_a-i-j)!}\frac{\Theta_{ab}^{i}}{i!}=
    (-1)^{j}\binom{k_a-k_b}{i}\frac{\Theta_{ba}^{j}}{j!}\frac{\Theta_{ac}^{k_a}}{k_a!}=
    \binom{k_a-k_b}{i}\frac{\Theta_{bc}^{j}}{j!}\frac{\Theta_{ac}^{k_a-j}}{(k_a-j)!}.
\end{equation}
\end{lemma}
The proof of Lemma \ref{lem:Thetas_relation} appears in Appendix \ref{pf:Thetas_relation}.

Recall that by Definition \ref{NKR},\begin{equation}
\label{eq:Mp_2_dgimot_identity}
\Mp_{j,k_b-j}^{(k_a-i,k_b+i)}=
\proj\circ 
\frac{\Theta_{ca}^{k_a-i-j}}{(k_a-i-j)!}
\frac{\Theta_{cb}^{i+j}}{(i+j)!}
\A_{k,k_b}\frac{\Theta_{bc}^{i+j}}{(i+j)!}
\frac{\Theta_{ac}^{k_a-i-j}}{(k_a-i-j)!}
\circ\proj,
\end{equation}
where $\proj$ is the projection to $S^{(k_a,k_b)}$ and $\A_{k,k_b}=\sum_{m=0}^{k_a}\frac{\sh_c^m\del_c^m}{(m!)^2}.$ We routinely use the basic facts that the conjugate transpose of $\Theta_{xy}$ is $\Theta_{yx}$, that $\sha_x$ and $\del_x$ are conjugate transpose of each other, and that $\proj$ is conjugate transpose to itself, as it is an orthogonal projection.

\begin{prop}
\label{prop:Mprime_depends_only_on_diag} 
For any $i,j$ it holds that 
\begin{equation}\label{eq:compatibility_of_MP}\Mp_{j,k_b-j}^{(k_a-i,k_b+i)}
=
\left(\binom{k_a-k_b}{i}i!\right)^2
\binom{k_b}{j}
{\Theta^{-i}_{ba}}
\Mp_{0,k_b}^{(k_a,k_b)}
{\Theta_{ab}^{-i}}.
\end{equation}
\end{prop}
\begin{proof}
We begin with the case $i=0$. 
By Lemma \ref{lem:Thetas_relation}, 
\begin{equation}
\label{eq:Akr_diagonal_identity1}
\begin{aligned}\Mp_{(j,k_b-j)}^{(k_a,k_b)}
\;&=\;\proj\circ \frac{\Theta_{ca}^{k_a-j}}{(k_a-j)!}\frac{\Theta_{cb}^{j}}{j!} \A_{k,k_b}\frac{\Theta_{bc}^j}{j!}\frac{\Theta_{ac}^{k_a-j}}{(k_a-j)!}\circ\proj\\
\;&=\;\proj\circ \frac{\Theta_{ca}^{k_a}}{k_a!}\frac{\Theta_{ab}^{j}}{j!} \A_{k,k_b}\frac{\Theta_{ba}^j}{j!}\frac{\Theta_{ac}^{k_a}}{k_a!}\circ\proj
\end{aligned}\end{equation}

Now, the map $\frac{\Theta_{ac}^{k_a}}{k_a!}:M^{(k_a,k_b,0)}\to M^{(0,k_b,k_a)}$ is a (trivial) isometry.
On the other hand, the map $\frac{\Theta_{ba}^{j}}{j!}:M^{(0,k_b,k_a)}\to M^{(j,k_b-j,k_a)}$ is a dilatation by a factor of $\sqrt{\tbinom{k_b}{j}}$. Indeed, it sends a basis element $e_I,$ where $I\in\binom{[k]}{k_c=k_a,k_b}$ (thought of as a map $[k]\to\{c,b\} $ which gives the value $b$ to exactly $k_b$ elements) to $\hat{e}_I:=\sum_J e_J,$ where $J\in\binom{[k]}{k_a,k_b-j,j}$ runs over all possible ways to assign $a$ to $j$ of the elements for which $I$ assigns $b.$ Clearly for different $I,I'$ the elements $\hat{e}_{I},\hat{e}_{I'}$ are orthogonal and $\langle\hat{e}_I,\hat{e}_I\rangle=\binom{k_b}{j}.$
Thus, for $\sum \lambda_I e_I\in S^{(k_a,k_b)},$\[\frac{\Theta_{ba}^j}{j!}\frac{\Theta_{ac}^{k_a}}{k_a!}\circ\proj(\sum \lambda_I e_I)=\sum \lambda_I \hat{e}_I.\]
By \eqref{eq:Akr_diagonal_identity1}, if $\sum \lambda_I e_I\in S^{(k_a,k_b)}$ is an eigenvector for the eigenvalue $\mu$ of $\Mp^{(k_a,k_b)}_{(0,k_b)},$ then $\sum \lambda_I \hat{e}_I$ is an eigenvector of $\proj\circ\A_{k,k_b}^{M^{(k_a,k_b-j,j)}}\circ\proj$ for the same eigenvalue $\mu$.
Thus, for any eigenvectors $\sum \lambda_I e_I,\sum \lambda'_I e_I\in S^{(k_a,k_b)}$ for the eigenvalues $\mu,\mu'$ of $\Mp^{(k_a,k_b)}_{(0,k_b)}$, we get from \eqref{eq:Akr_diagonal_identity1} that
\begin{align*}
\langle\sum \lambda'_I e_I,\Mp_{(k,k_b-j)}^{(k_a,k_b)} \sum \lambda_I e_I\rangle\,=\,&\langle \tfrac{\Theta_{ba}^j}{j!}\tfrac{\Theta_{ac}^{k_a}}{k_a!}\circ\proj(\sum \lambda'_I e_I),
\A_{k,k_b}^{M^{(k_a,k_b-j,j)}}\tfrac{\Theta_{ba}^j}{j!}\tfrac{\Theta_{ac}^{k_a}}{k_a!}\circ\proj(\sum \lambda_I e_I)\rangle \\
\,=\,&\langle \sum \lambda'_I \hat{e}_I, \proj\circ\A_{k,k_b}^{M^{(k_a,k_b-j,j)}}\circ\proj(\sum \lambda_I \hat{e}_I)\rangle \\
\,=\,& \bar{\mu}\langle \sum \lambda'_I \hat{e}_I,
\sum \lambda_I \hat{e}_I\rangle=\bar{\mu}\binom{k_b}{j}\langle \sum \lambda'_I {e}_I, \sum \lambda_I {e}_I\rangle \\
\,=\,&\langle \sum \lambda'_I {e}_I, \binom{k_b}{j}\Mp^{(k_a,k_b)}_{(0,k_b)} \sum \lambda_I {e}_I\rangle
\end{align*}
Since the eigenvectors of a real symmetric matrix are a basis for the space, this settles the case $i=0$.

The general case now follows from the $i=0$ case by applying Lemma \ref{lem:Thetas_relation} to the expression \eqref{eq:Mp_2_dgimot_identity} in the following way:
\begin{multline*}
\Mp_{j,k_b-j}^{(k_a-i,k_b+i)}=
\proj\circ 
\frac{\Theta_{ca}^{k_a-i-j}}{(k_a-i-j)!}
\frac{\Theta_{cb}^{i+j}}{(i+j)!}
\A_{k,k_b}
\frac{\Theta_{bc}^{i+j}}{(i+j)!}
\frac{\Theta_{ac}^{k_a-i-j}}{(k_a-i-j)!}
\circ\proj
\\=
{i!\Theta^{-i}_{ba}}
\circ\proj\circ
\frac{\Theta^i_{ba}}{i!}
\frac{\Theta_{ca}^{k_a-i-j}}{(k_a-i-j)!}
\frac{\Theta_{cb}^{i+j}}{(i+j)!}
\A_{k,k_b}
\frac{\Theta_{bc}^{i+j}}{(i+j)!}
\frac{\Theta_{ac}^{k_a-i-j}}{(k_a-i-j)!}
\frac{\Theta_{ab}^i}{i!}
\circ\proj\circ
{i! \Theta_{ab}^{-i}}
\\=
\left(\binom{k_a-k_b}{i}i!\right)^2
{\Theta^{-i}_{ba}}
\circ\proj\circ
\frac{\Theta_{ca}^{k_a-j}}{(k_a-j)!}
\frac{\Theta_{cb}^{j}}{j!}
\A_{k,k_b}
\frac{\Theta_{bc}^{j}}{j!}
\frac{\Theta_{ac}^{k_a-j}}{(k_a-j)!}
\circ\proj\circ
{\Theta_{ab}^{-i}}
\\=
\left(\binom{k_a-k_b}{i}i!\right)^2
{\Theta^{-i}_{ba}}
\Mp_{j,k_b-j}^{(k_a,k_b)}
{\Theta_{ab}^{-i}}
=
\left(\binom{k_a-k_b}{i}i!\right)^2
\binom{k_b}{j}
{\Theta^{-i}_{ba}}
\Mp_{0,k_b}^{(k_a,k_b)}
{\Theta_{ab}^{-i}}.
\end{multline*}
\end{proof}
Now, by combining Proposition \ref{prop:Mprime_depends_only_on_diag} with Proposition \ref{prop:diagonalizing_M_prime}, and using the fact that the kernel of $\RTR^{(m_1,m_2)}$ is exactly the kernel of $\del_a^{(m_1,m_2)}|_{S^{(m_1,m_2)}}$ shown in the proof of Lemma \ref{lem:0j}, we get:

\begin{corollary}
\label{cor:eigenspaces and eigenvalues for Mp - uncomfortable params}
The eigenspaces for
$\Mp_{r,k_b-r}^{(k_a-l,k_b+l)}$ are \[\Theta_{ab}^l(\shp_b^j(\sh_a^i(S^{(k_a-i,k_b-j)}\cap\ker\del_a^{(k_a-i,k_b-j)}))), \]
for $0\leq i\leq k_a-k_b,~0\leq j\leq k_b-1.$
The eigenvalue for the $(i,j)-$th eigenspace is
\[
\binom{k_a-k_b}{l}
\binom{k_b}{r}
\cdot
\frac{(2k_a+k_b)!(k_a-k_b+1)!}{i!(2k_a+k_b-i-j)!(k_a-k_b+1+j)!},
\]
and for the eigenvalue $0$, the eigenspace is $\left(S^{(k_a,k_b)}\right)^{\perp}=\bigoplus_{j> k_a} S^{(j,k_a+k_b-j)}$.
The dimension of the $(i,j)$-eigenspace is
\[\binom{k_a+k_b-i-j-2}{k_a-i-1}-\binom{k_a+k_b-i-j-2}{k_a-i}.\]
\end{corollary}
\begin{proof}
To prove the Corollary, we require the following lemma, whose proof will be given in Appendix \ref{pf:eigen_of_THETA}.

\begin{lemma}
\label{lem:eigen_of_THETA} ~
\begin{enumerate}
\item The map $\Theta_{ba}\Theta_{ab}$ when applied to the irreducible copy $S^{(k_a,k_b)}\hookrightarrow M^{(k_a-i,k_b+i)},$ where $k_b\leq k_a$ and $0\leq i\leq k_a-k_b,$ acts as the scalar $(i+1)(k_a-k_b-i).$
\item The map $\Theta_{ab}\Theta_{ba}$ applied to the same $S^{(k_a,k_b)}\hookrightarrow M^{(k_a-i,k_b+i)},$ acts as the scalar $i(k_a-k_b-i+1)$ if $1\leq i$ and acts as $0$ if $i=0$.
\item The map $\Theta_{ba}^r\Theta_{ab}^r$ applied to the same $S^{(k_a,k_b)}\hookrightarrow M^{(k_a-i,k_b+i)},$ acts as the scalar $\frac{(i+r)!(k_a-k_b-i)!}{i!(k_a-k_b-i-r)!}$.
\item The map $\Theta_{ab}^l\Theta_{ba}^l$ applied to the same $S^{(k_a,k_b)}\hookrightarrow M^{(k_a-i,k_b+i)},$ acts as the scalar $\frac{i!(k_a-k_b-i+l)!}{(i-l)!(k_a-k_b-i)!}$ if $l\leq i$ and acts as $0$ if $l>i$.
\end{enumerate}
\end{lemma}

Let $v$ be an eigenvector of $\Mp_{0,k_b}^{(k_a,k_b)}$ in the $(i,j)$ eigenspace 
$$\shp_b^j(\sh_a^i(S^{(k_a-i,k_b-j)}\cap\ker\del_a^{(k_a-i,k_b-j)}))$$ 
of the eigenvalue 
$$\lambda_{ij}=\frac{(2k_a+k_b)!(k_a-k_b+1)!}{i!(2k_a+k_b-i-j)!(k_a-k_b+1+j)!}.$$
By Theorem \ref{prop:diagonalizing_M_prime}, this is the form of any eigenvector outside the kernel. Now we apply the map $\Theta_{ab}^l$ on these eigenspaces and we claim that we get the eigenspaces of $\Mp_{r,k_b-r}^{(k_a-l,k_b+l)}$. Since the nonkernel eigenspaces of $\Mp_{0,k_b}^{(k_a,k_b)}$ span $S^{(k_a,k_b)}\hookrightarrow M^{(k_a,k_b)}$, their images through $\Theta_{ab}^l$ spans $S^{(k_a,k_b)}\hookrightarrow M^{(k_a-l,k_b+l)}$, which is the orthogonal complement of the kernel of $\Mp_{r,k_b-r}^{(k_a-l,k_b+l)}$. Thus it is enough to show that $\Theta_{ab}^l v$ is an eigenvector of $\Mp_{r,k_b-r}^{(k_a-l,k_b+l)}$ with eigenvalue depending only on $(i,j)$ to show that the eigenspaces of $\Mp_{r,k_b-r}^{(k_a-l,k_b+l)}$ are precisely of the form we claim in this corollary.

So we may calculate, using Proposition \ref{prop:Mprime_depends_only_on_diag}, that
\begin{equation*}
\begin{aligned}
\Mp_{r,k_b-r}^{(k_a-l,k_b+l)} \Theta_{ab}^l v &=
\left(\tbinom{k_a-k_b}{l}l!\right)^2
\tbinom{k_b}{r}
\Theta^{-l}_{ba}
\Mp_{0,k_b}^{(k_a,k_b)}
\Theta_{ab}^{-l}
\Theta_{ab}^l v \\&=
\left(\tbinom{k_a-k_b}{l}l!\right)^2
\tbinom{k_b}{r}
\Theta^{-l}_{ba}
\Mp_{0,k_b}^{(k_a,k_b)}
v \\&=
\left(\tbinom{k_a-k_b}{l}l!\right)^2
\tbinom{k_b}{r}
\lambda_{ij}
\Theta^{-l}_{ba}
v \\&=
\left(\tbinom{k_a-k_b}{l}l!\right)^2
\tbinom{k_b}{r}
\lambda_{ij}
\left(\Theta_{ab}^l \Theta^{l}_{ba}\right)^{-1}
\Theta_{ab}^l v
\end{aligned}
\end{equation*}
Since $\Theta_{ab}^l v \in S^{(k_a,k_b)}\hookrightarrow M^{(k_a-l,k_b+l)}$ we apply Lemma \ref{lem:eigen_of_THETA} with $i=l$, and continue
\begin{align*}
\;\;\;\;\;\;\;\;\;\;\;\;\;\;\;\;\;\;\;\;\;\;\;\;\;\;\;\;\; &= 
\left(\tbinom{k_a-k_b}{l}l!\right)^2
\tbinom{k_b}{r}
\lambda_{ij}
\left(\tfrac{l!(k_a-k_b)!}{(k_a-k_b-l)!}\right)^{-1}
\Theta_{ab}^l v \\&=
\tbinom{k_a-k_b}{l}^2
\tbinom{k_b}{r}
\lambda_{ij}
\tbinom{k_a-k_b}{l}^{-1}
\Theta_{ab}^l v \\&=
\tbinom{k_a-k_b}{l}
\tbinom{k_b}{r}
\lambda_{ij}
\Theta_{ab}^l v.
\end{align*}
This shows that $v$ is indeed an eigenvector with an eigenvalue depending only on the eigenspace index $(i,j)$. Putting the explicit value of $\lambda_{ij}$ in the above equation, we obtain the required eigenvalue. The dimensions of the eigenspaces follow from Theorem \ref{prop:diagonalizing_M_prime}, since $\Theta_{ab}$ is injective on the spaces to which it is applied.
\end{proof}

\begin{proof}[Proof of Theorem \ref{thm:main4}]
In Corollary \ref{cor:eigenspaces and eigenvalues for Mp - uncomfortable params} assign $k_b\coloneqq (r_a+r_b)$, $k_a\coloneqq ((k_a-r_a)+(k_b-r_b))$, $l\coloneqq ((k_b-r_b)-r_a)$, $r\coloneqq r_a$, and then replace when possible $k_a+k_b=k,\;r_a+r_b=r'$. 

This replaces the matrix $\Mp_{r,k_b-r}^{(k_a-l,k_b+l)}$ with the more convenient $\Mp_{r_a,r_b}^{(k_a,k_b)}$. Note that the following expressions will be valid because $r'=r_a+r_b\leq k_b \leq k_a$. 
The ranges for $(i,j)$ are now $0\leq i\leq k_a+k_b-2r_a-2r_b=k-2r',~0\leq j\leq r_a+r_b-1=r'-1$, and the $(i,j)$ eigenspace is by Definition \ref{wkrij}
\begin{multline*}
    W_{\kpa r'}^{(i,j)}=\Theta_{ab}^{k_b-r'}(\shp_b^j(\sh_a^i(S^{(k-r'-i,r-j)}\cap\ker\del_a^{(k-r'-i,r'-j)})))\\=
    \Theta_{ab}^{k_b-r_b-r_a}(\shp_b^j(\sh_a^i(S^{(k_a-r_a+k_b-r_b-i,r_a+r_b-j)}\cap\ker\del_a^{(k_a+k_b-r_a-r_b-i,r_a+r_b-j)}))).
\end{multline*}
The eigenvalue for the $(i,j)-$th eigenspace is
\begin{multline*}
\frac{\binom{k_a-2r_a+k_b-2r_b}{k_b-r_b-r_a}
\binom{r_a+r_b}{r_a}
\cdot(2k_a-r_a+2k_b-r_b)!(k_a-2r_a+k_b-2r_b+1)!}{i!(2k_a-r_a+2k_b-r_b-i-j)!(k_a-2r_a+k_b-2r_b+1+j)!} \\[0.25em] \;=\;
\binom{k-2r'}{k_b-r'}
\binom{r'}{r_a}
\cdot
\frac{(2k-r')!(k-2r'+1)!}{i!(2k-r'-i-j)!(k-2r'+1+j)!}
\end{multline*}
The dimension of the $(i,j)$-eigenspace is
\begin{multline*}
\binom{k_a+k_b-i-j-2}{k_a+k_b-r_a-r_b-i-1}-\binom{k_a+k_b-i-j-2}{k_a+k_b-r_a-r_b-i} \\[0.25em] \;=\;
\binom{k-i-j-2}{k-r'-i-1}-\binom{k-i-j-2}{k-r'-i} \;=\; \frac{(k-2r'-i+j+1)\,(k-i-j-2)!}{(k-i-r')!\,(r'-j-1)!}    
\end{multline*}
Finally, the eigenspace of the eigenvalue $0$ of $\Mp_{r,k_b-r}^{(k_a-l,k_b+l)}$ can be found from either Corollary \ref{cor:eigenspaces and eigenvalues for Mp - uncomfortable params} by the assignment of variables, or directly since it is the space orthogonal to the projection $\proj$ to $S^{(k-r,r)}$.

Let us summarize the information about the eigenspaces in the following lemma. Since we no longer use $r$ from Corollary \ref{cor:eigenspaces and eigenvalues for Mp - uncomfortable params}, replace $r'$ with $r$.

\begin{lem}
The eigenspaces for
$\Mp_{r_a,r_b}^{(k_a,k_b)}$ are \[W_{\kpa rij} = \Theta_{ab}^{k_b-r}(\shp_b^j(\sh_a^i(S^{(k-r-i,r-j)}\cap\ker\del_a^{(k-r-i,r-j)}))), \]
for $0\leq i\leq k-2r,~0\leq j\leq r-1$. \\
The eigenvalue for the $(i,j)$ eigenspace is
\[
\binom{k-2r}{k_b-r}
\binom{r}{r_a}
\lambda_{\kpa rij}
\;\;\;\;\;\text{where}\;\;\;\;\;
\lambda_{\kpa rij} :=
\frac{(2k-r)!(k-2r+1)!}{i!(2k-r-i-j)!(k-2r+1+j)!},
\]
and for the eigenvalue $0$, the eigenspace is \[\left(S^{(k-r,r)}\right)^{\perp}=\bigoplus_{j> k-r} S^{(j,k-j)}.\]
The dimension of the $(i,j)$ eigenspace is
\[ \frac{(k-2r-i+j+1)\,(k-i-j-2)!}{(k-i-r)!\,(r-j-1)!}\;.\]
\end{lem}

\medskip
Let $f \in W_{\kpa rij}$ and $f' \in W_{\kpa r'i'j'}$ be as in the statement of the theorem. The case $r \neq r'$ was settled in \eqref{eq:multisample different r case}, so suppose $f,f' \in W_{\kpa r}$. Then by \eqref{eq:Wkr_cov_matrix} their mixed second moment is
\begin{align*}
\Exp_w \left[\tilde\# f(w)\,\tilde\# f'(w)\right] \;&=\; \frac{(k_\TA!)^2(k_\TB!)^2}{(2k-r)!} \sum_{r_a+r_b=r} \frac{\left<\Mp_{(r_a,r_b)}^{(k_a,k_b)}f,f'\right> +O\left(\tfrac{1}{n_a}+\tfrac{1}{n_b}\right)}{n_a^{r_a} n_b^{r_b}} \\ 
\;&=\; \frac{(k_\TA!)^2(k_\TB!)^2}{(2k-r)!} \sum_{r_a+r_b=r} \frac{\left<\binom{k-2r}{k_b-r} \binom{r}{r_a} \lambda_{\kpa rij} f,f'\right> +O\left(\tfrac{1}{n_a}+\tfrac{1}{n_b}\right)}{n_a^{r_a} n_b^{r_b}}  \\
\;&=\; \frac{(k_\TA!)^2(k_\TB!)^2}{(2k-r)!}  \tbinom{k-2r}{k_b-r} \lambda_{\kpa rij}\left(\sum_{r_a+r_b=r} \frac{\binom{r}{r_a}}{n_a^{r_a} n_b^{r_b}} \right) \left(\left<f,f'\right> +O\left(\tfrac{1}{n_a}+\tfrac{1}{n_b}\right)\right) \\
\;&=\; \frac{(k_\TA!)^2(k_\TB!)^2(k-2r)!\, \lambda_{\kpa rij}}{(k_\TA-r)!(k_\TB-r)!(2k-r)!} \left(\tfrac{1}{n_a}+\tfrac{1}{n_b}\right)^r \left(\left<f,f'\right> +O\left(\tfrac{1}{n_a}+\tfrac{1}{n_b}\right)\right)
\end{align*}
Now multiply by $(n_\TA n_\TB / n)^r = (1/n_\TA + 1/n_\TB)^{-r}$ as in the statement of the theorem. The error term becomes $O(1/n_a+1/n_b)$ and goes to zero if $\min(n_\TA,n_\TB) \to \infty$. The remaining constant is as stated in the theorem. This completes the proof of Theorem \ref{main4}.
\end{proof}

\appendix
\section{Properties of the Algebra of Words}
\label{appendixsection}

This appendix contains the proofs of several properties of the algebra of words that are used in the derivation of Theorem~\ref{main4}.

\subsection{Proof of Lemma \ref{lem:eigen_of_THETA}}
\label{pf:eigen_of_THETA}
\begin{proof}[\unskip\nopunct]
Since $\Theta_{xy}$ is always a module map, it takes an irreducible module either to $0$ or to an isomorphic module, and a composition of $\Theta$s which takes a module to itself must act as a scalar.
Using Lemma \ref{lem:props_of_basic_ops}\eqref{it:decomp_of_M_to_Spechts} we see that if $k_a\geq k_b$ then
\[\Theta_{ba}|_{S^{(k_a,k_b)}\hookrightarrow M^{(k_a,k_b)}}=0,\] hence also
\[\Theta_{ab}\Theta_{ba}|_{S^{(k_a,k_b)}\hookrightarrow M^{(k_a,k_b)}}=0.\]
Using \eqref{eq:comm_theta} we see that
\[\Theta_{ba}\Theta_{ab}|_{S^{(k_a,k_b)}\hookrightarrow M^{(k_a,k_b)}}=(k_a-k_b)Id_{S^{(k_a,k_b)}\hookrightarrow M^{(k_a,k_b)}}.\]
This establishes the $i=0$ case of the lemma.
We shall use induction. For $0\leq i\leq k_a-k_b,$ using Lemma~\ref{lem:props_of_basic_ops}\eqref{it:Theta}, we see that
\[\Theta_{ab}^i|_{S^{(k_a,k_b)}\hookrightarrow M^{(k_a,k_b)}},\]maps the Specht module isomorphically on ${S^{(k_a,k_b)}\hookrightarrow M^{(k_a-i,k_b+i)}}.$
Thus, in order to calculate the scalar action of $\Theta_{ba}\Theta_{ab}|_{S^{(k_a,k_b)}\hookrightarrow M^{(k_a-i,k_b+i)}},$ it is enough to calculate it on one non zero element of ${S^{(k_a,k_b)}\hookrightarrow M^{(k_a-i,k_b+i)}},$ or equivalently on
$\Theta_{ab}^iv,$ for a non zero $v\in{S^{(k_a,k_b)}\hookrightarrow M^{(k_a,k_b)}}.$
Suppose we have shown that for a non zero $v\in{S^{(k_a,k_b)}\hookrightarrow M^{(k_a,k_b)}},$
\[\Theta_{ba}\Theta_{ab}(\Theta_{ab}^{i-1} (v))=\left(i(k_a-k_b)-i(i-1)\right)\Theta_{ab}^{i-1}(v),\]
then
\begin{align*}\Theta_{ba}\Theta_{ab}(\Theta_{ab}^{i} (v))&=\Theta_{ab}\Theta_{ba}\Theta_{ab}(\Theta_{ab}^{i-1} (v))+
[\Theta_{ba},\Theta_{ab}](\Theta_{ab}^{i} (v))\\
&=\left(i(k_a-k_b)-i(i-1)\right)\Theta_{ab}^{i}(v)+(k_a-k_b-2i)\Theta_{ab}^{i}(v)\\
&=\left((i+1)(k_a-k_b)-i(i+1)\right)\Theta_{ab}^{i}(v),\end{align*}
where the second passage used the induction hypothesis and the commutation relation \eqref{eq:comm_theta}. The induction follows.

Using $\Theta_{ba}\Theta_{ab}=(i+1)(k_a-k_b-i)\mathrm{Id}$ and the commutation relation \eqref{eq:comm_theta} $$[\Theta_{ab},\Theta_{ba}]=-((k_a-i)-(k_b+i))=-(k_a-k_b-2i)\mathrm{Id},$$ we get $\Theta_{ab}\Theta_{ba}=i(k_a-k_b-i+1)\mathrm{Id}$.

To show that $\Theta_{ba}^r\Theta_{ab}^r = \frac{(i+r)!(k_a-k_b-i)!}{i!(k_a-k_b-i-r)!} \mathrm{Id}$ and $\Theta_{ab}^l\Theta_{ba}^l = \frac{i!(k_a-k_b-i+l)!}{(i-l)!(k_a-k_b-i)!} \mathrm{Id}$ if $l\leq i$ and $0$ if $l>i$, we use induction on $r,l$ respectively. The cases $l=0,\;r=0$ are the two cases previously discussed. Note that $\Theta_{ba}^r\Theta_{ab}^r(v)$ is 
\begin{multline*}
    \Theta_{ba}^{(k_a-i-1,k_b+i+1)} \cdots \Theta_{ba}^{(k_a-i-(r-1),k_b+i+(r-1))} \Theta_{ba}^{(k_a-i-r,k_b+i+r)} \\
\Theta_{ab}^{(k_a-i-(r-1),k_b+i+(r-1))} \Theta_{ab}^{(k_a-i-(r-2),k_b+i+(r-2))} \cdots \Theta_{ab}^{(k_a-i,k_b+i)} (v)
\end{multline*}
and since by the case $r=1$, 
$$\Theta_{ba}^{(k_a-i-r,k_b+i+r)} \Theta_{ab}^{(k_a-i-(r-1),k_b+i+(r-1))} = (i+(r-1)+1)(k_a-k_b-i-(r-1))\mathrm{Id},$$ we get that $\Theta_{ba}^r\Theta_{ab}^r(v)$ is 
\begin{multline*}
    (i+r)(k_a-k_b-i-r+1) (\Theta_{ba}^{(k_a-i-1,k_b+i+1)} \cdots \Theta_{ba}^{(k_a-i-(r-1),k_b+i+(r-1))}\\ \Theta_{ab}^{(k_a-i-(r-2),k_b+i+(r-2))} \cdots \Theta_{ab}^{(k_a-i,k_b+i)}) (v).
\end{multline*} Assuming by induction that the statement is true for $r-1,$ we get $$\Theta_{ba}^r\Theta_{ab}^r(v) = (i+r)(k_a-k_b-i-r+1)\frac{(i+r-1)!(k_a-k_b-i)!}{i!(k_a-k_b-i-r+1)!} \mathrm{Id} = \frac{(i+r)!(k_a-k_b-i)!}{i!(k_a-k_b-i-r)!} \mathrm{Id},$$ giving the desired result. The remaining case follows similarly.
\end{proof}

\subsection{Proof of Observation \ref{obs:delm_shi}}
\label{pf:delm_shi}
\begin{proof}[\unskip\nopunct]
If $m>i$ then by the commutation relations \eqref{eq:comm_del_sh} (and using $\Theta_{aa}^{(k_a,k_b)}=k_a$) $\del_a^m\sh_a^i = O\del_a^{m-i},$ where $O\in \mathbb{Z}[\sh_a\del_a].$ Thus, $\del_a^m\sh_a^i(v)=0.$

Suppose $m\leq i.$ Using the commutation relations \eqref{eq:comm_del_sh} again we obtain
\[\del_a^m\sh_a^i = \sh_a^{i-m}O,~O\in \mathbb{Z}[\sh_a\del_a],\] thus $\del_a^m\sh_a^i(v)=O_0 \sh_a^{i-m}(v)$ where $O_0\in\mathbb{Z}$ is the constant coefficient in $O.$
Denote this coefficient $O_0$ by $c(k_a,k_b,m,i).$ Clearly $c(k_a,k_b,0,i)=1.$ Using the commutation relations in \eqref{eq:comm_del_sh} we obtain a recursion from commuting one $\del_a$ to the rightmost position, obtaining a scalar whenever we pass each $\sh_a$ operator:
\[\del_a\sh_a^i(v)=\sh_a\del_a\sh_a^{i-1}(v)+(k_b+1+2(k_a-1))\sh_a^{i-1}v=\]
\[=\sh^2_a\del_a\sh_a^{i-2}(v)+(k_b+1+2(k_a-1)+k_b+1+2(k_a-2))\sh_a^{i-1}v=\cdots=\]
\[=(k_b+1+2(k_a-1)+k_b+1+2(k_a-2)+\ldots+k_b+1+2(k_a-i))\sh_a^{i-1}v+\sh_a^l \del_a v=\]
\[=(k_b+1+2(k_a-1)+k_b+1+2(k_a-2)+\ldots+k_b+1+2(k_a-i))\sh_a^{i-1}v.\]
Thus, $c(k_a,k_b,m,i)=$
\[c(k_a-1,k_b,m-1,i-1)(k_b+1+2(k_a-1)+k_b+1+2(k_a-2)+\ldots+k_b+1+2(k_a-i))=\]\[=i(2k_a+k_b-i)c(k_a-1,k_b,m-1,i-1).\]
Repeating we obtain $c(k_a,k_b,m,i)=\frac{(2k_a+k_b-i)!}{(2k_a+k_b-m-i)!}\frac{i!}{(i-m)!}\sh_a^{i-m}(v).$
\end{proof}

\subsection{Proof of Lemma \ref{lem:Thetas_relation}}
\label{pf:Thetas_relation}
\begin{proof}[\unskip\nopunct]
The second equality follows immediately from commuting $\frac{\Theta_{ba}^{j}}{j!}, ~\frac{\Theta_{ac}^{k_a}}{k_a!},$ in the middle expression in \eqref{eq:Thetas_relation}, using the commutations relations \eqref{eq:comm_theta} and the fact that $\Theta_{ba}$ restricted to the copy of $S^{(k_a,k_b)}\hookrightarrow M^{(k_a,k_b)}$ is zero.

Regarding the main part, first note that both sides of \eqref{eq:Thetas_relation} map $M^{(k_a,k_b)}$ to $M^{(k_a,j,k_b-j)}.$ The hook length formula \cite[page 50]{fulton2013representation} guarantees that $S^{(k_a,k_b)}$ appears in the latter $S_{k_a+k_b}$-module with multiplicity $1.$ Since $\Theta_{xy}$ are morphisms of $S_{k_a+k_b}$-modules by Lemma~\ref{lem:props_of_basic_ops}(\ref{it:Theta}), it must happen that the expressions on the left-hand and right-hand sides of \eqref{eq:Thetas_relation} differ by a multiplicative scalar, that may also be zero.
Since these two expressions are maps of modules, in order to calculate this scalar, it is enough to apply both sides of \eqref{eq:Thetas_relation} to one element of $S^{(k_a,k_b)}$ and compare the results.

In order to describe a nonzero element $e_T$ we shall use the well-known construction of the Specht module. Consider the Young diagram $D_{k_a,k_b}=(k_a,k_b).$ Recall that a Young tableau $T$ corresponds to a basis element $e_T$ of $M^{(k_a,k_b)},$ in which the entries of the first row correspond to the locations of $\TA$ while those of the second row to the locations of $\TB.$ Consider the Standard Young tableau $T$ in which the first row contains the elements $1,2,\ldots, k_b,2k_b+1,2k_b+2,\ldots, k_a.$ Define \[v_T=\sum_{\pi}(-1)^{sgn(\pi)}e_{\pi\cdot T},\] where the summation is taken over $\pi\in S_{k_a+k_b}$ which preserve the columns of $T.$ $v_T$ can be written more explicitly as
\begin{equation}\label{eq:v_T}
\sum_{R \in R_{T}}(-1)^{s(R)}e_R,
\end{equation}
here $R_{T}$ is the set of Young tableau for $D_{k_a,k_b}$ for which the rightmost elements in the first row are $2k_b+1,\ldots,k_a,$ the entries of the $j-$th column for $m\leq k_b$ are $\{m,k_b+m\},$ and
\begin{multline*}
s(r)=|\{k_b<m\leq2k_b|~m~\text{appears in the first row of } R\}|\\=|\{m\leq k_b|~m~\text{appears in the second row of }R\}|.\end{multline*}

Let $D_{k_a,j,k_b-j}$ we the Young diagram with rows of length $k_a,j,k_b-j.$ The corresponding Young tableaux index the standard basis of $M^{(k_a,j,k_b-j)},$ where the first row indicates the locations of $\TC$, the row of length $j$ the locations of $\TA$ and the remaining row the locations of $\TB.$
We first calculate \begin{equation}\label{eq:Thetas_RHS}\frac{\Theta_{ba}^{j}}{j!}\frac{\Theta_{ac}^{k_a}}{k_a!} v_T=
\sum_{Q\in Q_T}(-1)^{s_r(Q)}e_Q,\end{equation}
where $Q$ is a Young tableau of $D_{k_a,j,k_b-j},~e_Q$ the corresponding basis element, $Q_T$ is the set of Young tableaux for which the indices greater than $2k_b$ are in the first row, and for each $m\leq k_b$ exactly one index in $\{m,k_b+m\}$ appears in the first row.
The sign $s_r(Q)$ is\[|\{m\leq k_b|~m~\text{does not appear in the first row of }Q\}|.\]
The proof is straightforward. The application of $\Theta_{ac}^{k_a}/k_a!$ changes all $\TA$ to $\TC$ and does nothing else. Then $\Theta_{ba}^j/j!$ changes exactly $j$ of the $\TB$'s to $\TA.$ Note that any $Q\in Q_T$ comes from a single $R\in R_T,$ obtained by combining the second and third row of $Q$ into one row. From this, the translation of $s(-)$ to $s_r(-)$ is also immediate.

We now calculate the application of the operator on the LHS of \eqref{eq:Thetas_relation} on $v_T.$
We first show
\begin{equation}\label{eq:Thetas_LHS_1}
v'_T:=\frac{\Theta_{ab}^{i}}{i!} v_T=\sum_{X\in X_T}(-1)^{s(X)}e_X,
\end{equation}
where $X_T$ is the collection of Young tableaux for the Young diagram $D_{k_a-i,k_b+i}$ for which in the $m-$th column for $m\leq k_b$ the entries are $m,k_b+m$ in some order, and $s(X)$ equals again
\begin{multline*}|\{k_b<m\leq2k_b|~m~\text{appears in the first row of } X\}|\\=|\{m\leq k_b|~m~\text{appears in the second row of }X\}|.\end{multline*}
For the proof, define, for a given Young tableau $X$ of shape $(k_a-i,k_b+i)$ the set $P(X)$ of Young tableaux $R\in R_T$ such that
$X$ appears with non zero coefficient in the representation of $\frac{\Theta_{ab}^i}{i!} e_R$ according to the standard basis. 
It is easy to see that the coefficient, if non zero, is $1.$
Thus, the coefficient of $e_X$ in $\frac{\Theta_{ab}^i}{i!} v_T$ is
\begin{equation}\label{eq:coeff_Thetas_1}
\sum_{R\in P(X)}(-1)^{s(R)}.
\end{equation}
When $X\in X_T$ it is straightforward to see that $P(X)$ is the singleton $R$ defined by removing the boxes containing the highest $i$ values in the row of $\TB$ and adding them to the row of $\TA$. This argument also explains the sign $s(X)=s(T).$

For any $X\notin X_T,$ with non empty $P(X),$ there is at least one $m\leq k_b$ such that both $m$ and $m+k_b$ appear in the row of $\TB.$
Let $m_\star$ be the least such $m.$ Define an involution $\iota:P(X)\to P(X)$ as follows. For $R\in P(X)$ exactly one of $m_{\star},k_b+m_{\star}$ appears in the row of $\TB.$ Then $\iota(R)$ is the same tableau, except that if $m_{\star}$ is in the row of $\TB$ in $R$ then it is $k_b+m_{\star}$ in $\iota(R)$, and vice versa.
$\iota$ clearly maps $P(X)$ to itself and is an involution. Moreover, by the definition of the sign $s(R),$ it is easily seen to be sign reversing,
\[s(\iota(R))=1-s(R).\]
Thus, 
\begin{align*}
\sum_{R\in P(X)}(-1)^{s(R)}=&
\sum_{R\in \iota(P(X))}(-1)^{s(R)}=
\sum_{R\in P(X)}(-1)^{s(\iota(R))}\\
=&\sum_{R\in P(X)}(-1)^{1-s(R)}=
-\sum_{R\in P(X)}(-1)^{s(R)}
.
\end{align*}
Therefore, the coefficient of $e_X,$ which is sum of \eqref{eq:coeff_Thetas_1}, vanishes, proving \eqref{eq:Thetas_LHS_1} holds, since

The next step is to calculate 
\begin{equation}\label{eq:Thetas_LHS_2}
\frac{\Theta_{bc}^{i+j}}{(i+j)!}\frac{\Theta_{ac}^{k_a-i-j}}{(k_a-i-j)!} v'_T=
\binom{k_a-k_b}{i}\sum_{Q\in Q_T}(-1)^{s_l(Q)}e_Q,
\end{equation}
where $Q_T$ is as above, but \begin{multline*}
    s_l(Q)=|\{m\leq k_b|~m~\text{appears in the row of $\TB$ in $Q$ or $k_b+m$ in the row of $\TA$ in }Q\}|.\end{multline*}
This proves the lemma, since $s_l(Q)s_r(Q)=(-1)^{j}.$ Indeed, modulo $2,$ the sum of 
\begin{multline*}|\{m\leq k_b|~m~\text{appears in the row of $\TB$ in $Q$} \text{ or $k_b+m$ appears in the row of $\TA$ in }Q\}|\end{multline*}and\[
|\{m\leq k_b|~m~\text{does not appear in the first row of }Q\}|
\]
equals \[|\{m\leq k_b|~\text{either $m$ or $k_b+m$ appears in the row of $\TA$}\}|.\]
This can be seen by defining three sets of $m\le k_b$, so that in the first $m$ and $m+k_b$ are in the rows of $\TB$ and $\TC$ respectively, in the second they're in the rows of $\TC$ and $\TA$ respectively, and in the third they're in the rows of $\TA$ and $\TC$ respectively, and noticing that each of the the above three sets whose cardinality we use is a disjoint union of another pair of the new sets.

Since for any $Q\in Q_T$ the row of $\TA$ is made of $j$ elements, all at most $2k_b,$ and for any $m\leq k_b$ at most one of $\{m,k_b+m\}$ belongs to the row of $\TA$ (because by our definition of $Q_T$, exactly one index in $\{m,k_b + m\}$ appears in the first row, that of $\TC$, so either only the other one  appears in the row of $\TA$, or neither of them appear in this row), the last cardinality is exactly the length of the row of $\TA,$ that is $j.$ 

We are left with proving \eqref{eq:Thetas_LHS_2}. First note that $Q_T$ is made of tableaux such that all elements greater than $2k_b$ are in the row of $\TC,$ while for any $m\leq k_b,$ exactly one of $\left\{m,m+k_b\right\}$ is in that row. In addition, exactly $j$ elements from $\{1,2,\ldots, 2k_b\}$ appear in the row of $\TA.$
The proof is similar to the proof of \eqref{eq:Thetas_LHS_1}, and uses an involution argument again. For $Q$ of shape $(k_a,j,k_b-j)$ let $P'(Q)$ be the collection of elements $X\in X_T$ such that $e_Q$ appears in the representation of $\frac{\Theta_{bc}^{i+j}}{(i+j)!}\frac{\Theta_{ac}^{k_a-i-j}}{(k_a-i-j)!}e_X$ in the standard basis. 
The coefficient of $e_Q$ in $\frac{\Theta_{bc}^{i+j}}{(i+j)!}\frac{\Theta_{ac}^{k_a-i-j}}{(k_a-i-j)!} v'_T$
is $\sum_{X\in P'(Q)}(-1)^{s(X)}$.
Suppose $Q\notin Q_T.$ Then, from the pigeonhole principle, there must be $m\le k_b$ such that both $m$ and $k_b+m$ appear in the first row. Let $m_{\star}$ be the minimal such $m$ (for the given $Q$).
Define an involution $\iota':P'(Q)\to P'(Q)$ as follows. For any $X\in P'(Q)$, either $m_\star$ appears in the row of $\TA$ and $k_b+m_\star$ in that of $\TB,$ or the opposite. In both cases $\iota'(X)$ is defined by moving $m_{\star},k_b+m_{\star}$ to the row in which they did not appear.
Again, $s(X)=1-s(\iota'(X)),$ and again we see that the coefficient of $e_X$ in the left-hand side of \eqref{eq:Thetas_LHS_2} vanishes.

For $Q\in Q_T,$ on the other hand, the set $P'(Q)$ is easily described: For any $m\leq k_b$ which appears in the first row of $Q$, if $m+k_b$ is in the row of $\TA,$ then for any $X\in P'(Q)$ it holds that $m$ appears in the row of $\TB,$ and $m+k_b$ appears in the row of $\TA.$ 
Similarly,
\begin{itemize}
    \item if $Q$ has $m\le k_b$ in the first row and $m+k_b$ in the row of $\TB,$ then $X\in P'(Q)$ has $m$ in the row of $\TA$ and $m+k_b$ in the row of $\TB$,
    \item if $Q$ has $k_b<m+k_b\le 2k_b$ in the first row and $m$ in the row of $\TA,$ then $X\in P'(Q)$ has $m$ in the row of $\TA$ and $m+k_b$ in the row of $\TB$, 
    \item if $Q$ has $k_b<m+k_b\le 2k_b$ in the first row and $m$ in the row of $\TB,$ then $X\in P'(Q)$ has $m$ in the row of $\TB$ and $m+k_b$ in the row of $\TA$.
\end{itemize}
These are the only four possibilities for $Q\in Q_T$.

Since $s(X)$ is defined using only the locations of the first $2k_b$ elements, it is the same for all $X \in P'(Q)$, and by the last argument it is equal to $s_l(Q)$.
The first row in $Q$ contains $k_a-k_b$ entries that are greater than $2k_b$, and these could come from either the $\TA$ row or the $\TB$ row in $X\in P'(Q)$, without any restriction. Therefore, we can get each element in $P'(Q)$ uniquely by assigning the entries up to $2k_b$ as determined, and then choosing some $i$ elements greater than $2k_b$ from the first row of Q to put in the $\TB$ row of $X$.
Thus, $|P'(Q)|=\binom{k_a-k_b}{i},$ since they all have the same sign $s_l,~$\eqref{eq:Thetas_LHS_2} follows.
\end{proof}

\section{Proof of Proposition \ref{relmm}}
\label{pf:relmm}
\begin{proof}[\unskip\nopunct]
The proof of Proposition~\ref{relmm} makes use of the following intermediate quantity, which will be related to both sides of the equation. 

\begin{definition}
\label{mcoefs_multisamp}
Let $\mathbf{r} = (r_\TA,r_\TB,\dots) \leq \kpa = (k_\TA,k_\TB,\dots)$. The $\mathbf{r}$th \emph{merging set} of two words $e, e' \in \tbinom{\Sigma}{\kpa}$ is 
\begin{equation*}
\mergeset_{\mathbf{r}}\left(e,e'\right) \;=\; \left\{\left(I,I'\right)\;\middle|\;\;
\begin{aligned}
& I = \{i_1, i_2, \dots, i_{|\kpa|}\} \;\;\; i_1 < i_2 < \dots \;\;\; \\
& I' = \{i_1', i_2', \dots, i_{|\kpa|}'\} \;\;\; i_1' < i_2' < \dots \;\;\; \\
& I \cup I' = \{1,2,\dots,2|\kpa|-|\mathbf{r}|\} \\
& i_j = i_{j'}' \;\implies\; e_j = e_{j'}' \\
& \forall \TX\in\Sigma,\;\left|\left\{i_j \in I \cap I' : e_j = \TX\right\}\right|={r}_\TX  
\end{aligned}
\right\},
\end{equation*}
and the $\mathbf{r}$th \emph{merging coefficient} of $e, e' \in \tbinom{\Sigma}{\kpa}$ is $\mergecnt_{\mathbf{r}}\left(e,e'\right)=\left|\mergeset_{\mathbf{r}}\left(e,e'\right)\right|.$
We extend $\mergecnt_{\mathbf{r}}$ bilinearly to a form on the entire space~$W_{\kpa}$.
\end{definition}

\begin{ex*}
$\mergeset_{(1,1)}(\TA\TA\TB,\TA\TB\TA) = \left\{(\{1,2,3\},\{1,3,4\}),(\{1,2,3\},\{2,3,4\})\right\}$
\end{ex*}

Note the difference between this merging set $\mergeset_{\mathbf{r}}(e,e')$ and the merging set $\mathcal{M}_{\ell}(e,e')$ from Definition~\ref{mcoefs}. 
The latter is indexed by one number $\ell$, while this one is indexed by a vector of numbers~$\mathbf{r}$. The corresponding cardinalities $\mergecnt_\mathbf{r}(f,f')$ and $m_{\ell}(f,f')$ are related by the following lemma.

\begin{lem}
\label{lem:Akr_thetaconj=mergecnt}
For $f,f'\in W_{\kpa}$ and $\mathbf{r} \leq \kpa$ as in Proposition~\ref{relmm},
\[
\mergecnt_{\mathbf{r}}\left(f,f'\right)
\;=\;
m_{2k-r}\left(\left(\prod_{\TX\in\Sigma}\frac{\Theta_{\TX\Ti}^{k_\TX-r_\TX}}{(k_\TX-r_\TX) !}\right)f,\,\left(\prod_{\TX\in\Sigma}\frac{\Theta_{\TX\Ti}^{k_\TX-r_\TX}}{(k_\TX-r_\TX)!}\right)f'\right)
\]
\end{lem}

\begin{proof}
Without loss of generality, we assume that $f$ and $f'$ are both words in $\tbinom{\Sigma}{\kpa}$, since the general case will follow by the bilinearity of $\mergecnt_{\mathbf{r}}$ and~$m_{\ell}$.
Let $A,A'$ be the sets of words that are obtained by 
\[\left(\prod_{\TX\in\Sigma}\frac{\Theta_{\TX\Ti}^{k_\TX-r_\TX}}{(k_\TX-r_\TX)!}\right)f \;=\; \sum_{w\in A} w,\quad\;\; \left(\prod_{\TX\in\Sigma}\frac{\Theta_{\TX\Ti}^{k_\TX-r_\TX}}{(k_\TX-r_\TX)!}\right)f' \;=\; \sum_{w\in A'} w\] 
Indeed this operator produces a sum of different words since $(\prod_{\TX}\Theta_{\TX\Ti}^{k_\TX-r_\TX}/(k_\TX-r_\TX)!)=\Theta(T)$, the replacement operator with respect to a table $T$ of shape~$\kpa$, where $r_{\TX}$ entries are $\TX$s in the row $t_{\TX}$ and the rest are $\Ti$s.

Let $I_{\TX} := \left\{j:f_j=\TX\right\}$. By the definition of $\Theta(T)$, for each $w\in A$ and each $\TX\in\Sigma$ there exist disjoint $I_{\TX\Ti} \cup I_{\TX\TX} = I_{\TX}$ such that if $j\in I_{\TX\Ti}$ then $w_j=\Ti$ and if $j\in I_{\TX\TX}$ then $w_j=\TX$, and such that $|I_{\TX\Ti}|=k_\TX-r_\TX$ and $|I_{\TX\TX}|=r_\TX$. Similarly, we also have disjoint $I_{\TX\Ti}' \cup I_{\TX\TX}' = I'_{\TX} := \{j:f'_j=\TX\}$ with the same properties for any $w'\in A'$.

By bilinearity, it suffices to show that $\mergecnt_{\mathbf{r}}\left(f,f'\right) = \sum_{w\in A}\sum_{w'\in A'} m_{2k-r}\left(w,w'\right)$. Therefore, by the definitions of the merging coefficients, it is enough to construct a bijection between the set $\mergeset_{\mathbf{r}}\left(f,f'\right)$ and
$\left\{(w,w',I_w,I'_{w'})\,\middle|\,w\in A,\;w'\in A',\;(I_w,I'_w)\in\mergeset_{2k-r}\left(w,w'\right)\right\}$.

Since for such sets $|I_w \setminus I'_w|=|I'_w \setminus I_w|=k-r$ and the number of $\Ti$s in each of $w$ and $w'$ is also $k-r$, it is necessary that $w_j = w_{j'}' \neq\Ti$ for $i_j=i'_{j'}\in I_w\cap I'_w$. Hence, the merging set of $w$ and $w'$ can be written as
\begin{equation*}
\mergeset_{2k-r}\left(w,w'\right) \;=\; \left\{\left(I_w,I'_w\right)\;\middle|\;\;
{
\begin{aligned} 
& I_w = \{i_1, i_2, \dots, i_{k}\} \;\;\; i_1 < i_2 < \dots \;\;\; \\
& I'_w = \{i_1', i_2', \dots, i_{k}'\} \;\;\; i_1' < i_2' < \dots \;\;\; \\
& I_w \cup I'_w = \{1,2,\dots,2k-r\} \\
& i_j = i_{j'}' \;\implies\; w_j = w_{j'}' \neq\Ti \\
& i_j \in I_w \setminus I'_w \;\iff\; w_j = \Ti \\
& i_{j'}' \in I'_w \setminus I_w \;\iff\; w_{j'}' = \Ti  
\end{aligned}
}
\right\},
\end{equation*}

For one direction of the bijection, given $(w,w',I_w,I'_{w'})$ we define $I=I_w,\;I'=I'_{w'}$, and it follows that $(I,I')\in\mergeset_{\mathbf{r}}\left(f,f'\right)$. For the inverse direction, given such $(I,I')$, define $I_w=I,\;I'_{w'}=I'$, and define $w$ so that for any $j$ if $i_j\in I\setminus I'$ then $w_j=\Ti$ and otherwise $w_j=f_j$, and similarly $w'$ by $I'\setminus I$ and~$f'$. By the conditions of Definition \ref{mcoefs_multisamp}, indeed $w \in A$ and $w'\in A'$ can be obtained via $\Theta(T)$ from $f$ and~$f'$. By the definition of $w$ and $w'$ it is clear that $(I_w,I'_{w'}) = (I,I') \in \mergeset_{2k-r}\left(w,w'\right)$.
\end{proof}

The next step in the proof of Proposition~\ref{relmm} is to connect the merging coefficient of Definition \ref{mcoefs_multisamp} to the expectation in Definition~\ref{mrff}.
\begin{lem}
\label{lem:mergecnt=mergecntprob}
For $f,f'\in W_{\kpa}$ and $\mathbf{r} \leq \kpa$ as in Proposition~\ref{relmm},
$$\mergecnt_{\mathbf{r}}(f,f') \;=\; \frac{(2k-r)!}{\prod_{\TX\in\Sigma}(r_\TX!(k_\TX-r_\TX)!^2)}\, m_{\mathbf{r}}(f,f')$$
\end{lem}
\begin{proof}
Due to the bilinearity of $m_{\mathbf{r}}$ and $\mergecnt_{\mathbf{r}}$, we can again assume without loss of generality that $f,f'$ are words in $W_{\kpa}$. 

In this case, Definition \ref{mrff} means that $m_{\mathbf{r}}(f,f')$ is the probability that $X_{\kpa-\mathbf{r}}$, $X_{\kpa-\mathbf{r}}'$ and~$X_{\mathbf{r}}$, the $2k-r$ independent random variables uniformly distributed in~$[0,1]$, are such that $\word(X_{\kpa-\mathbf{r}}\cup X_{\mathbf{r}})=f$ and $\word(X_{\kpa-\mathbf{r}}'\cup X_{\mathbf{r}})=f'$. This condition depends only on the relative order of these $2k-r$ uniform random variables. Since these are iid, all orders have the same probability,~$1/(2k-r)!$. Therefore, in order to calculate the probability $m_{\mathbf{r}}(f,f')$ we count the possible orders that will satisfy the aforementioned condition.

First, note the following degrees of freedom. Permuting the values of $(X_{\TA 1},\dots,X_{\TA r_\TA})$ does not affect the resulting $\mathrm{word}(X_{\kpa-\mathbf{r}}\cup X_{\mathbf{r}})$ and $\mathrm{word}(X_{\kpa-\mathbf{r}}'\cup X_{\mathbf{r}})$. Neither does a permutation of the values $(X_{\TX 1},\dots,X_{\TX r_\TX})$ for any $\TX\in\Sigma$, or of $(X_{\TX (r_\TX+1)},\dots,X_{\TX k_\TX})$ or $(X_{\TX (r_{\TX}+1)}',\dots,X_{\TX k_\TX}')$. The number of such permutations is $\prod_{\TX} r_\TX!(k_\TX-r_\TX)!(k_\TX-r_\TX)!$. Taking into account this factor, we no longer distinguish between the indices within each of the sets $\{X_{\TX 1},\dots,X_{\TX r_\TX}\}$, $\{X_{\TX r_{(\TX+1)}},\dots,X_{\TX k_\TX}\}$ and $\{X_{\TX r_{(\TX+1)}}',\dots,X_{\TX k_\TX}'\}$ for~$\TX\in\Sigma$. It remains to count the ways these sets partition the set of relative positions $\{1,\dots,2k-r\}$ such that the induced words are $f$ and~$f'$. 

Denote by $I \subseteq \{1,\dots,2k-r\}$ the positions of $\bigcup_{\TX\in\Sigma}\{X_{\TX 1},\dots,X_{\TX k_\TX}\}$, and similarly by~$I'$ those of $\bigcup_{\TX\in\Sigma}(\{X_{\TX 1},\dots,X_{\TX r_\TX}\}\cup\{X_{\TX (r_{\TX}+1)}',\dots,X_{\TX k_\TX}'\})$. Obviously $I\cup I'=\{1,\dots,2k-r\}$. Let $I = \{i_1, i_2, \dots, i_k\}$ where $i_1 < i_2 < \dots$ and similarly for~$I'$. 
Construct two words $e,e'\in \tbinom{\Sigma}{\kpa}$ such that $e_j=\TX$ if and only if the position $i_j$ belongs to a random variable with a label~$\TX$, and $e_j'$ is determined similarly by the element in the position~$i_j'$.

Then the event that $\mathrm{word}\left(X_{\kpa-\mathbf{r}}\cup X_{\mathbf{r}}\right)=f$ and $\mathrm{word}(X_{\kpa-\mathbf{r}}'\cup X_{\mathbf{r}})=f'$ is equivalent to having $e=f$ and $e'=f'$. The conditions in the definition of $\mergeset_{\mathbf{r}}(f,f')$ easily follow, and we have a one-to-one correspondence between pairs $(I,I')\in\mergeset_{\mathbf{r}}(f,f')$ and the ways to distribute the positions to the above sets of random variables.

In conclusion, $m_{\mathbf{r}}(f,f') =  \left(\prod_{\TX} r_\TX!(k_\TX-r_\TX)!(k_\TX-r_\TX)!\right)/(2k-r)! \cdot |\mergeset_{\mathbf{r}}(f,f')|$, as claimed in the lemma.
\end{proof}

Proposition~\ref{relmm} now follows from Lemma~\ref{lem:Akr_thetaconj=mergecnt} and Lemma~\ref{lem:mergecnt=mergecntprob}.
\end{proof}

\section{Intransitivity of Dice}
\label{dice3}

Let $w$ be a random word distributed by $\mathcal{W}'(n,n,n)$, and $g = \TC\TB\TA + \TB\TA\TC + \TA\TC\TB - \TA\TB\TC - \TB\TC\TA - \TC\TA\TB$, as in Example~\ref{dice}. Here $g \in W_{\kpa}$ for $\kpa = (1,1,1)$. The Gepner statistic $\#g/n^3$ can be written as a generalized $U$-statistic as in~\S\ref{proof3}, letting the $3n$ values on the dice be independent $X_i \sim Y_j \sim Z_k \sim U(0,1)$ for all $i,j,k \in \{1,\dots,n\}$. Then $\#g$ is given by summation over the following kernel function, which tells the cyclic ordering of its~$a,b,c \in [0,1]$:
$$ g_{\text{word}}(a;b;c) \;=\; \begin{cases}
-1 & a<b<c \;\;\text{or}\;\; b<c<a \;\;\text{or}\;\; c<a<b \\
+1 & a>b>c \;\;\text{or}\;\; b>c>a \;\;\text{or}\;\; c>a>b
\end{cases} $$ 

An orthonormal basis with respect to $U(0,1)$, which generally works well for this type of statistics, is the usual Fourier basis. Up to the normalizing constant~$\sqrt{2}$, it is given by:
$$ 1,\; \sin 2\pi x,\; \cos 2\pi x,\; \sin 4\pi x,\; \cos 4\pi x,\; \sin 6\pi x,\; \cos 6\pi x,\; \dots $$

Using the notation of Definitions~\ref{mrff}-\ref{fr}, we rewrite the generalized U-statistic~$g_{\word}$, actually in the form of its Hoeffding decomposition. Clearly $g^{000} = \Exp[g_{\word}]=0$. Also the conditional expectation $g^{100}(a) = \Exp_{B,C}[g_{\word}(a;B;C)] \equiv 0$ by its antisymmetry under the transposition of $b$ and $c$. Similarly $g^{010}$ and $g^{001}$ vanish. This is a special case of Theorem~\ref{main3}. In order two, we obtain  
$$ g^{110}(a,b) \;=\; \Exp_C[g_{\word}(a;b;C)] \;=\; 1-2((a-b) \bmod 1) $$
and similarly for $g^{011}$ and $g^{101}$, with $(b-c)$ and $(c-a)$ respectively. Observe that we can write the function
$$ g_{\word}(a;b;c) \;=\; g^{110}(a;b) + g^{101}(a;c) + g^{110}(b;c) $$
without third order terms in this case. This can be verified by considering all possible orderings of the inputs, such as $a<b<c$, etc.

We now use the Fourier series of the sawtooth function, for each of the three functions in the sum above.
$$ 1 - 2(x \bmod 1) \;=\; \tfrac{2}{\pi} \sin 2 \pi x + \tfrac{2}{2\pi} \sin 4 \pi x + \tfrac{2}{3\pi} \sin 6 \pi x + \dots $$
We also use $\sin(\theta-\theta') = \sin \theta \cos \theta' - \sin \theta' \cos \theta$, to translate this to the orthonormal bases. Denoting for short $\phi_j(x) = \sqrt{2}\sin(2 \pi j x)$ and $\psi_j(x) = \sqrt{2}\cos(2 \pi j x)$, we obtain
\begin{align*}
g_{\word}(a;b;c)
\;&=\; \sum_{j=1}^{\infty} \tfrac{\phi_j(b)\psi_j(a)-\phi_j(a)\psi_j(b) + \phi_j(c)\psi_j(b)-\phi_j(b)\psi_j(c) + \phi_j(a)\psi_j(c)-\phi_j(c)\psi_j(a)}{\pi j} 
\end{align*}
We now apply a weak limit theorem for this U-statistic. The case of a multisample degenerate U-statistic was treated by Eagleson in~\cite{eagleson1979orthogonal}, for example. Theorem~3 in that paper may be slightly adapted by adding extra terms to include the form of the above kernel. Then the asymptotic distribution is given by
$$ \frac{1}{n^2}\sum_{i,k,l}g_{\word}(X_i,Y_k,Z_l) \;\;\xrightarrow[\;\;n \to \infty\;\;]{d}\;\; \sum_{j=1}^{\infty} \frac{\beta_j\alpha_j'-\alpha_j\beta_j'+\gamma_j\beta_j'-\beta_j\gamma_j'+\alpha_j\gamma_j'-\gamma_j\alpha_j'}{\pi j} $$
where $\alpha_j, \alpha_j', \beta_j, \beta_j', \gamma_j, \gamma_j'$ for $j \in \mathbb{N}$ are sequences of independent standard normal random variables. We rewrite it and change variables:
$$ \;=\; \sum_{j=1}^{\infty} \frac{(\beta_j-\alpha_j)(\alpha_j'+\beta_j'-2\gamma_j')-(\alpha_j+\beta_j-2\gamma_j)(\beta_j'-\alpha_j')}{2 \pi j} \\
\;=\; \frac{\sqrt{3}}{\pi} \;\sum_{j=1}^{\infty} \frac{\xi_j\eta_j'-\eta_j\xi_j'}{j}
$$
where $\xi_j,\xi_j',\eta_j,\eta_j'$ are all iid standard normal. All $\xi_j$ and $\eta_j$ are obtained by a three-dimensional rotation of $\alpha_j, \beta_j, \gamma_j$. This is the sum that gives the distribution of the closed L\'evy area, the signed area enclosed by a two-dimensional Brownian bridge~\cite{levy1951wiener}. Therefore, similar to the closed L\'evy area, $\#g/n^2$ asymptotically follows the logistic distribution, with the probability density function
$$ f(x) \;=\; \frac{\pi}{4\sqrt{3}}\; \text{sech}^2\left(\frac{\pi x}{2\sqrt{3}}\right) $$
This limit law for $\#g/n^2$ was discovered by Zeilberger~\cite{zeilberger2016doron} based on the leading terms of the first twelve moments.
 
\bibliographystyle{alpha}
\bibliography{main}

\end{document}